\documentclass[11pt]{amsart}
\author{Pete L. Clark}

\title{CM elliptic curves: volcanoes, reality and applications, Part I}
\usepackage{amsmath,amssymb,amsthm,mathtools,url}
\usepackage{geometry,graphicx}
\usepackage{tikz, rotating}
\usepackage{color}
\DeclareMathAlphabet{\curly}{U}{rsfs}{m}{n}
\newtheorem{thm}{Theorem}[section]

\newtheorem{example}[thm]{Example}
\newtheorem{cor}[thm]{Corollary}
\newtheorem{prop}[thm]{Proposition}
\newtheorem{lemma}[thm]{Lemma}

\theoremstyle{definition}

\newtheorem{remark}{Remark}[section]
\begin{document}
\newcommand\leg{\genfrac(){.4pt}{}}
\renewcommand{\labelenumi}{(\roman{enumi})}
\def\ff{\mathfrak{f}}
\def\N{\mathbb{N}}
\def\Q{\mathbb{Q}}
\def\Z{\mathbb{Z}}
\def\R{\mathbb{R}}
\def\OO{\mathcal{O}}
\def\aa{\mathfrak{a}}
\def\pp{\mathfrak{p}}
\def\qq{\mathfrak{q}}
\def\Aa{\curly{A}}
\def\Dd{\curly{D}}
\def\Gg{\curly{G}}
\def\Pp{\curly{P}}
\def\Ss{\curly{S}}
\def\End{\mathrm{End}}
\def\M{\curly{M}}
\newcommand{\Ok}{\mathcal{O}_K}
\newcommand{\tors}{\operatorname{tors}}
\newcommand{\Pic}{\operatorname{Pic}}
\newcommand{\ord}{\operatorname{ord}}
\newcommand{\GL}{\operatorname{GL}}
\newcommand{\Gal}{\operatorname{Gal}}
\newcommand{\gcdl}{\operatorname{gcd}}
\renewcommand{\gg}{\mathfrak{g}}

\newcommand{\abedit}[1]{{\color{blue} \sf  #1}}
\newcommand{\abbey}[1]{{\color{blue} \sf $\clubsuit\clubsuit\clubsuit$ Abbey: [#1]}}
\newcommand{\margAb}[1]{\normalsize{{\color{red}\footnote{{\color{blue}#1}}}{\marginpar[{\color{red}\hfill\tiny\thefootnote$\rightarrow$}]{{\color{red}$\leftarrow$\tiny\thefootnote}}}}}
\newcommand{\Abbey}[1]{\margAb{(Abbey) #1}}
\newcommand{\pcedit}[1]{{\color{brown} \sf #1}}
\newcommand{\ra}{\rightarrow}
\newcommand{\C}{\mathbb{C}}
\newcommand{\Ker}{\operatorname{Ker}}
\newcommand{\F}{\mathbb{F}}
\newcommand{\Spec}{\operatorname{Spec}}
\newcommand{\CM}{\operatorname{CM}}
\newcommand{\Aut}{\operatorname{Aut}}
\renewcommand{\aa}{\mathfrak{a}}
\newcommand{\bb}{\mathbf{b}}
\newcommand{\rad}{\operatorname{rad}}
\newcommand{\dd}{\mathfrak{d}}
\newcommand{\lcm}{\operatorname{lcm}}
\renewcommand{\P}{\mathbb{P}}
\newcommand{\cyc}{\operatorname{cyc}}
\newcommand{\PP}{\mathbb{P}}
\newcommand{\oddCM}{\operatorname{odd-CM}}
\newcommand{\cc}{\mathbf{c}}

\maketitle

\begin{abstract}
For positive integers $M \mid N$ and an order of discriminant $\Delta$ in an imaginary quadratic field $K$ with discriminant $\Delta_K < -4$, we determine the fiber of 
the morphism $X_0(M,N) \ra X(1)$ over the closed point $J_{\Delta}$ corresponding to $\Delta$.  We also show that the fiber of 
the natural map $X_1(M,N) \ra X_0(M,N)$ over $J_{\Delta}$ is connected.  Putting this together we deduce the number of points in the fiber of $X_1(M,N) \ra X(1)$ over $J_{\Delta}$ and their residual degrees.  In the continuation of this work \cite{CS22b}, these results 
will be extended to $\Delta_K \in \{-4,3\}$.  These works provide all the information needed to compute, for each positive integer $d$, all subgroups of $E(F)[\tors]$, where $F$ is a number field of degree $d$ and $E_{/F}$ is an elliptic curve with complex multiplication (CM).  
\end{abstract}

\tableofcontents

\section{Introduction}
\noindent
This paper continues work of the author and his collaborators on torsion points of elliptic curves with complex multiplication (CM) 
over number fields and CM points on elliptic modular curves \cite{CCS13}, \cite{CCRS14}, \cite{CP15}, \cite{BP17}, \cite{BCP17}, \cite{BCS17}, \cite{CP17}, \cite{BCI}, \cite{BCII}, \cite{CCM21}, \cite{CGPS}.

\subsection{Some Modular Curves} In the present work it will be convenient to take a geometric perspective, so we begin by recalling the notion of a modular curve $X(H)_{/\Q}$ attached to a subgroup 
$H \subset \GL_2(\Z/N\Z)/\{\pm 1\}$, as is developed in \cite{Mazur76}.\footnote{We also recommend \cite{Rohrlich97} for an especially 
careful exposition.}    For $N \in \Z^+$, let 
$E_{/\Q(t)}$ be an elliptic curve with $j$-invariant $t$.  Then \[\Aut(\Q(t,E[N])/\Q(t)) \cong \GL_2(\Z/N\Z), \] and 
we let $X(N)_{/\Q}$ be the smooth, projective integral (but not geometrically integral, if $N \geq 3$) curve with function field 
$\Q(X(N)) \coloneqq \Q(t,E[N])^{\{\pm 1 \}}$.
Thus $\Q(X(N))/\Q(t)$ is Galois with group $\GL_2(\Z/N\Z)/\{\pm 1\}$.  Identifying 
$\Q(t)$ with the function field of the $j$-line $X(1) \cong \P^1$, we get a Galois branched covering of curves $X(N) \ra X(1)$.  
To any $H \subset \GL_2(\Z/N\Z)/\{ \pm 1\}$ we get a subextension $\Q(X(H)) \coloneqq \Q(X(N))^H$ of $X(N) \ra X(1)$ and thus a corresponding 
intermediate covering 
\[ X(N) \ra X(H) \ra X(1). \]
The spectrum of the function field $\Q(X(H))$ is the fiber of $\pi_H: X(H) \ra X(1)$ over the generic point of $X(1)_{/\Q}$.  
The support of the fiber $\infty_H \coloneqq \pi_H^*(\infty)$ consists precisely of the cusps on $X(H)$, and we take $Y(H)_{/\Q}$ to be the 
smooth, integral affine curve $X(H) \setminus \infty_H$.  We have $Y(1) \cong \mathbb{A}^1_{/\Q} = \Spec \Q[t]$, so a closed 
point $J \in Y(1)$ is given by an irreducible polynomial $J(t) \in \Q[t]$.  For $J \neq t,t-1728$, computing the fibers $\pi_H^*(J)$ over 
$J$ for all $H$ is essentially the same as computing the ``adelic Galois $\pm$-representation'' 
\[\rho_{\pm}: \gg_{\Q[t]/(J)} \ra \GL_2 (\widehat{\Z}) / \{ \pm 1 \} \]
on any elliptic curve with $j$-invariant $j$, where $j$ is a root of $J$ in $\overline{\Q}$.  If $F$ is a field of characteristic $0$ 
and $E_{/F}$ is an elliptic curve such that the modulo $N$ Galois $\pm$-representation 
\[ \rho_{N,\pm}: \gg_K \ra \GL_2(Z/N\Z) \ra \GL_2(\Z/N\Z)/\{ \pm 1\} \]
has image in $H$, then $E$ induces an $F$-rational point on $X(H)$.  Two elliptic curves $(E_1)_{/F}$, $(E_2)_{/F}$ with $j$-invariants 
different from $0,1728$ whose $\pm$-modulo $N$ Galois representations lie in $H$ induce the same point on $X(H)(F)$ iff they are quadratic twists of each other, and every point of $X(H)(F)$ whose image under $\pi: X(H)(F) \ra X(1)(F) = \P^1(F)$ does not lie in $\{t,t-1728,\infty\}$ arises from such an elliptic curve $E_{/F}$.  Similarly, a closed point $P \in X(H)$ not lying over
$t$ or $t-1728$ on $X(1)$ with residue field $\Q(P)$ 
corresponds to an elliptic curve $E_{/\Q(P)}$ with $\rho_{N,\pm}(\gg_{\Q(P)}) \subset H$, well-defined up quadratic twist.  
Moreover, for a number field $F$, every noncuspidal $F$-rational point $P \in X(H)(F)$ (including $j = 0,1728$) is induced by at least one elliptic curve $E_{/F}$ for which the mod $N$ Galois representation lies inside $H \subset \GL_2(\Z/N\Z)$ \cite[Prop. VI.3.2]{Deligne-Rapoport73}.
\\ \\
The modular curves of interest to us here are the following ones:
\\ \\
$\bullet$ The curve $X(N)_{/\Q}$ itself, a $\GL_2(\Z/N\Z)/\{\pm 1\}$-Galois cover of the $j$-line $X(1)$.  
For a field $F$ of characteristic $0$, an elliptic curve $E_{/F}$ with $j(E) \neq 0,1728$ defines an $F$-rational point on $X(N)$ iff the mod $N$ Galois $\pm$-representation 
$\rho_{N,\pm}: \gg_F \ra \GL_2(\Z/N\Z)/\{ \pm 1\}$ is trivial iff for some quadratic twist $E^{\chi}$ of $E$ 
the group scheme $E^{\chi}[N]$ is constant (``full $N$-torsion'').   
\\ \\
$\bullet$ For positive integers $M \mid N$, we put $X_1(M,N) \coloneqq X(H_1(M,N))$, where $H_1(M,N) \subset \GL_2(\Z/N\Z)/\{ \pm 1\}$ is the subgroup 
\[ \bigg{\{} \pm \left( \begin{array}{cc} 1 & b \\ 0 & d \end{array} \right) \mid  \ b \equiv 0 \pmod{M},   \ d \equiv 1  \pmod{M}  \bigg{\}}.  \]
We have  $X(N) = X_1(N,N)$.  At the other extreme, $X_1(N) \coloneqq X_1(1,N)$.   We have \[ [\Q(X_1(M,N)):\Q(X(1))] = \begin{cases} 1 & (M,N) = (1,1) \\ 3 & (M,N) = (1,2) \\ 6 & (M,N) = (2,2) \\ \frac{M \varphi(M) \varphi(N) \psi(N)}{2} & N \geq 3 \end{cases}, \] 
(see e.g. \cite[\S 7.2]{CGPS}), where $\varphi(N) \coloneqq \# (\Z/N\Z)^{\times}$ and $\psi: \Z^+ \ra \Z^+$ is the unique multiplicative function such that $\psi(\ell^a) = (\ell+1)\ell^{a-1}$ for all prime powers $\ell^a$.  
For a field $F$ of characteristic $0$, an elliptic curve $E_{/F}$ with $j(E) \neq 0,1728$ defines an $F$-rational point on $X_1(M,N)$ 
iff for some quadratic twist $E^{\chi}$ there is an injective group homomorphism $\Z/M\Z \times \Z/N\Z \hookrightarrow E(F)$.  Thus 
the study of torsion subgroups of elliptic curves over number fields is very closely related to the study of closed points on $X_1(M,N)$.  
\\ \\
$\bullet$ For positive integers $M \mid N$, we put $X_0(M,N) \coloneqq X(H_0(M,N))$, where $H_0(M,N) \subset \GL_2(\Z/N\Z)/\{ \pm 1\}$ is the subgroup 
\[ \bigg{\{} \left( \begin{array}{cc} a & b \\ 0 & d \end{array} \right) \mid b \equiv 0 \pmod{M}, \ a \equiv d \pmod{M}  \bigg{\}}. \] 
We have $X_0(N) \coloneqq X_0(1,N)$.  We have 
\begin{equation}
\label{DEGREEX0MNEQ}
 [\Q(X_0(M,N)):\Q(X(1))]  = [\GL_2(\Z/N\Z)/\{ \pm 1\}:H_0(M,N)] = M \varphi(M) \psi(N) 
\end{equation}
and thus also 
\[ \deg(X_1(M,N)) \ra X_0(M,N)) = \begin{cases} 1 & N \leq 2 \\ \frac{\varphi(N)}{2} & N \geq 3 \end{cases}. \]
For a field $F$ of characteristic $0$, an elliptic curve $E_{/F}$ with $j(E) \neq 0,1728$ 
defines an $F$-rational point on $X_0(M,N)$ iff $E_{/F}$ admits an $F$-rational cyclic $N$-isogeny and also \emph{every} 
cyclic $M$-isogeny $\varphi: E \ra E'$ is $F$-rational: the latter condition is equivalent to Galois acting on $E[M]$ by scalar matrices.

\subsection{The $\Delta$-CM Locus}
An elliptic curve over a field of characteristic $0$ has \textbf{complex multiplication} if its geometric endomorphism ring is an order in an imaginary quadratic field.  Imaginary quadratic orders are classified up to isomorphism by their discriminant $\Delta$.  Each discriminant is a negative integer congruent to $0$ or $1$ mod $4$, each negative integer $\Delta \equiv 0,1 \pmod{4}$ is the discriminant of a unique imaginary quadratic 
order, and for each imaginary quadratic discriminant $\Delta$ there is a unique closed point $J_{\Delta} \in X(1)_{/\Q}$ corresponding 
to elliptic curves with CM by the order of discriminant $\Delta$.   Thus for any modular curve $X(H)_{/\Q}$, we have the 
\textbf{$\Delta$-CM locus}, which is the fiber of the map $\pi: X(H) \ra X(1)$ over the closed point $J_{\Delta}$.  This is a finite 
$\Spec \Q(J_{\Delta})$-scheme, and it is \'etale if $\Delta < -4$.  
\\ \\
A recent result of Bourdon-Clark nearly determines the $\Delta$-CM locus on $X(N)_{/\Q}$:

\begin{thm}[Bourdon-Clark \cite{BCI}]
\label{ITHM0}
Let $\OO$ be an order in the imaginary quadratic field $K$, with 
discriminant $\Delta$, which we may write as $\Delta = \ff^2 \Delta_K$ for $\ff \in \Z^+$.  Let
$P \in X(N)_{/\Q}$ be a closed $\Delta$-CM point, and let $\pi(P) = J_{\Delta}$ be its image on $X(1)_{/\Q}$.
\begin{itemize}
\item[a)] Suppose that $N \geq 3$ or $\Delta$ is odd.  Then we have 
\[ \Q(P) = K(J_{N^2\Delta})K^{(N)}, \]
where $K^{(N)}$ is the $N$-ray class field of $K$.  Thus $\Q(P)$ contains $K$. Also we have 
\[ [\Q(P):\Q(\pi(P))] = 2 [\Q(P):K(\pi(P))] = 2 \# (\OO/N\OO)^{\times}. \]
\item[b)] Suppose that $N = 2$ and $\Delta < -4$ is even.  Then we have 
\[\Q(P) \cong \begin{cases} \Q & \Delta = -4 \\ \Q(J_{4\Delta}) & \Delta < -4 \end{cases}. \]
In particular $\Q(P)$ does not contain $K$.  We also have 
\[ [\Q(P):\Q(\pi(P))] = \begin{cases} 1 & \Delta = -4 \\ 2 & \Delta < -4 \end{cases}. \]
\end{itemize}
\end{thm}

\begin{remark}
Let us compare Theorem \ref{ITHM0} to some other results, both classical and recent.
\begin{itemize}
\item[a)] When the order $\OO$ is maximal -- i.e., $\OO = \Z_K$, the full ring of integers of $K$ -- we have $\Delta = \Delta_K$ 
and $K(P) = K(J_{N^2 \Delta_K})K^{(N)} = K^{(N)}$, and this is a geometric phrasing of the ``First Main Theorem of CM'' 
\cite[Thm. II.5.6]{SilvermanII} in the case where the ideal is $N\Z_K$ for some $N \in \Z^+$.   
\item[b)] Closely related results have been obtained by Stevenhagen \cite{Stevenhagen01}, Lozano-Robledo \cite{LR22}, and 
Campagna-Pengo \cite{CP22}.  All three of these authors take the perspective that the residue field $K(P)$ of any $\Delta$-CM 
closed point on $X(N)_{/K}$ may be thought as the ``$N$-ray class field of the nonmaximal 
order $\OO$.''\footnote{This occurs in the final version of \cite{BCI} as well, but identifying $K(P)$ as the compositum of a ray class 
field and a ring class field is a component of the proof given there.} In fact, for any nonzero invertible $\OO$-ideal $I$, 
Campagna-Pengo show that the field $K(\mathfrak{h}(E[I]))$ obtained by adjoining to $K$ the values of a Weber function 
on the $I$-torsion kernel $E[I]$ is the $I$-ray class field of the 
nonmaximal order $\OO$ \cite[Thm. 4.7]{CP22}.  This is a full-fledged generalization of the First Main Theorem of Complex Multiplication to all imaginary quadratic orders.  
\end{itemize}
\end{remark}
\noindent
The map $X(N) \ra X(1)$ is unramified away from $0,1728,\infty$, hence the same holds for all morphisms of modular curves $X(H) \ra X(1)$.  For $\Delta < -4$, locally at $J_{\Delta}$ the map $X(N) \ra X(1)$ is an unramified Galois covering, so the number of closed points is the degree of the covering (which we know: cf. \S 1.1) divided by $[\Q(P):\Q]$.   When $\Delta \in \{-3,-4\}$ we fall short of a complete description of the fiber of $\pi: X(N) \ra X(1)$ over $J_{\Delta}$ only insofar as we do not address their scheme-theoretic structure when they are not reduced.

\subsection{Our Goal}
The main goal of this paper is to pursue analogues of Theorem \ref{ITHM0} for the $\Delta$-CM locus on the modular curves 
$X_0(M,N)_{/\Q}$ and $X_1(M,N)_{/\Q}$.   Since the coverings 
$X_1(M,N) \ra X(1)$ and $X_0(M,N) \ra X(1)$ are usually not Galois, the automorphism group need not act transitively on the fibers, so 
there may be more than one residue field of closed points in the $\Delta$-CM locus.   In fact we will see that as $N$ grows, the 
$\Delta$-CM locus on $X_0(N)$ and on $X_1(N)$ can contain closed points of arbitrarily many different degrees.  
\\ \\
The results obtained here generalize another recent work of Bourdon-Clark \cite{BCII}, which determines for all $\Delta$ and 
$M \mid N$, the least degree of a closed $\Delta$-CM point on $X_1(M,N)_{/K}$ and also over $X_1(M,N)_{/\Q}$.  The latter result is 
equivalent to the determination of the least degree of a number field $F$ for which there is a $\Delta$-CM elliptic curve $E_{/F}$ 
such that $\Z/M\Z \times \Z/N\Z \hookrightarrow E(F)$.  
\\ \\
Consider the following ambitious problem: for $d \in \Z^+$, determine all groups $T$ that arise as a subgroup of $E(F)[\tors]$ 
for some elliptic curve $E$ defined over a degree $d$ number field.  Work of Merel \cite{Merel96} shows that the set of (isomorphism classes of) such groups is finite for each $d$.  However, at present the complete list of such groups is known only for $d = 1$ by work of Mazur \cite{Mazur77}, for $d = 2$ by work of Kenku-Momose and Kamienny \cite{Kenku-Momose88}, \cite{Kamienny92} and for $d = 3$ by work of Derickx-Etropolski-van Hoeij-Morrow-Zureick-Brown \cite{DEvHMZB20}.  It might be possible to handle the case of $d = 4$ 
by similar methods, but to push things much farther than that seems to require a major theoretical breakthrough.
\\ \indent
In contrast, a sufficiently good understanding of the $\Delta$-CM locus on the family of curves $X_1(M,N)_{/\Q}$ will yield a complete solution to the above problem upon restriction to the class of CM elliptic curves.  In the CM case, much better bounds on $\# E(F)[\tors]$ in terms of $d = [F:\Q]$ are 
known by work of Silverberg \cite{Silverberg88}, \cite{Silverberg92} and Clark-Pollack \cite{CP15}, \cite{CP17}.  If one knows all degrees of closed $\Delta$-CM points on 
$X_1(M,N)_{/\Q}$ then one knows all pairs $(M,N)$ for which there is a closed $\Delta$-CM point of degree dividing $d$ on 
$X_1(M,N)$, and thus the groups $\Z/M\Z \times \Z/N\Z$ are precisely the $\Delta$-CM subtorsion groups in degree $d$.  
Moreover, since $[\Q(J_{\Delta}):\Q]$ tends to infinity with $\Delta$, only finitely many $\Delta$ arise in each degree $d$, so 
this yields a complete list of CM subtorsion groups in degree $d$.  
\\ \\
Notice that we do not actually need the list of all degrees of closed $\Delta$-CM points on $X_1(M,N)$ 
but rather only the list of all multiples of these degrees.  Otherwise put, for applications to the determination of CM subtorsion groups 
it is enough to know all \textbf{primitive} degrees $[\Q(P):\Q]$ of closed $\Delta$-CM points on $X_1(M,N)_{/\Q}$, namely those 
degrees that are not a proper multiple of any other such degree.  This turns out to simplify the answer considerably: Bourdon-Clark
 showed that every degree of a closed $\Delta$-CM point on $X_1(M,N)_{/K}$ is a multiple of the least degree, and thus there is 
always a unque primitive $\Delta$-CM degree.   On the other hand, \cite[Example 6.7]{BCII} gives a case in which there are at least 
two primitive degrees of $\Delta$-CM closed points on $X_1(N)_{/\Q}$, and therefore knowing the least degree does not give all 
degrees in which a $\Delta$-CM elliptic curve can have a subgroup isomorphic to $\Z/M\Z \times \Z/N\Z$.  The methods developed here allow one to determine all degrees of $\Delta$-CM closed points on $X_1(M,N)_{/\Q}$, but the tabulation of the answers gets complicated.   In a much more explicit way we will record all primitive degrees of $\Delta$-CM closed 
points on $X_1(M,N)_{/\Q}$.   We find that there are either one or two such degrees.  \\ \indent
Theorem \ref{ITHM0} implies that when $M \geq 3$ the residue field of $\Q(P)$ of every $\Delta$-CM point on $X_1(M,N)_{/\Q}$ contains $K$, 
so the work of Bourdon-Clark  computes the unique primitive degree of a $\Delta$-CM point.  Thus it remains to consider $M \in \{1,2\}$.

\subsection{Transition to $X_0(M,N)$}

Although most of the results of \cite{BCII} concern torsion subgroups of CM elliptic curves, a key ingredient in their proofs was the study of rational cyclic $N$-isogenies on CM elliptic curves.  As Bourdon and I worked 
on \cite{BCII}, they gradually became aware of the extent to which the torsion subgroups of CM elliptic curves are controlled 
by the existence or nonexistence of cyclic isogenies on CM elliptic curves rational over various fields.  The natural map $X_1(N) \ra X_0(N)$ is an isomorphism for $N \leq 2$ and a $(\Z/N\Z)^{\times}/\{\pm 1\}$-Galois cover for $N \geq 3$, which guarantees a connection between 
isogenies and torsion points: as is well known, if you have an elliptic curve defined over a number field $F$ with an $F$-rational cyclic $N$-isogeny, then this elliptic curve has a point of order $N$ rational over a field extension of $F$ of degree at most $\frac{\varphi(N)}{2}$.  
But the results of \cite{BCII} show a much tighter relationship in the CM case.  This phenomenon is elucidated by the following result that we will prove now, using the work of \cite{BCI}.  

\begin{thm}
\label{ITHM1}
\label{INERTNESSTHM}
Let $\Delta < -4$ be an imaginary quadratic discriminant, let $M \mid N$ be positive integers, and let $P \in X_0(M,N)_{/\Q}$ be a closed 
$\Delta$-CM point.  Then the map $\pi: X_1(M,N) \ra X_0(M,N)$ is inert over $P$: 
that is, writing the fiber $\pi^*(P)$ as $\Spec A$ for a finite-dimensional $\Q(P)$-algebra $A$, we have that $A$ is a field.  
\end{thm}
\begin{proof}
If $N \leq 2$, then the map $X_1(M,N) \ra X_0(M,N)$ is an isomorphism, so we may assume that $N \geq 3$, in which case 
it has degree $\frac{\varphi(N)}{2}$.  \\ \indent
Let $\pi: X_0(N,N) \ra X_0(M,N)$, and choose a point $\tilde{p} \in X_0(N,N)$ such that $\pi(\tilde{p}) = p$.  Because the  covering $X(N) = X_1(N,N) \ra X_0(M,N)$ is the fiber product of the coverings $X_1(M,N) \ra X_0(M,N)$ and $X_0(N,N) \ra X_0(M,N)$, it suffices 
to show that the fiber of $X(N) \ra X_0(N,N)$ over $\tilde{p}$ is inert.  (If $F$ is a number field, $L/F$ is a finite degree field 
extension and $A/F$ is a finite dimensional commutative $F$-algebra such that $L \otimes_F A$ is a field, then $A$ is a field.)  So we 
reduce to the case $M = N$ and write $p$ in place of $\tilde{p}$.  Similarly, it suffices to prove the inertness 
result for the map $X(N) \ra X_0(N,N)$, viewed as a morphism of curves over the imaginary quadratic field $K$. \\ \indent   
The closed point $p$ comes from a $\Delta$-CM elliptic curve $E_{/K(p)}$ for which the modulo $N$ Galois representation $\rho_N: 
\gg_{K(p)} \ra \GL_2(\Z/N\Z)$
 consists 
of scalar matrices.  The elliptic curve $E$ is well-determined up to a quadratic (since $\Delta < -4$) twist, and therefore the modulo $N$ $\pm$-Galois 
representation
\[ \rho_{N,\pm} = (\GL_2(\Z/N\Z) \ra \GL_2(\Z/N\Z)/\{ \pm 1\}) \circ \rho_N \]
is well-defined.  By \cite[Thm. 4.1]{BCI} we have $K(p) = K(N\ff)$, the $N$-ring class field of $K$.  The proof of \emph{loc. cit.} 
shows that 
\[ [K(p):K(\ff)] = \frac{ \# (\OO/N\OO)^{\times}}{\varphi(N)}, \]
while \cite[Thm 1.4]{BCI} gives $\# \rho_{N,\pm}(\gg_{K(\ff)}) = \frac{\# (\OO/N\OO)^{\times}}{2}$.  From these 
two facts it follows that $\# \rho_{N,\pm}(\gg_{K(p)})= \frac{\varphi(N)}{2} = \deg(X(N) \ra X_0(N,N))$, which establishes the result.
 \end{proof}
\noindent
Theorem \ref{ITHM1} implies that, when $\Delta < -4$, knowing the degrees and multiplicities of the $\Delta$-CM points 
on $X_0(M,N)_{/\Q}$, yields the degrees and multiplicities of the $\Delta$-CM points on $X_1(M,N)_{/\Q}$: if $N \leq 2$ 
the curves are the same, while for $N \geq 3$, multiply each degree by $\frac{\varphi(N)}{2}$.  The same holds 
for primitive degrees, with the upshot being that all the information we referred to above about $\Delta$-CM closed points on 
$X_1(M,N)_{/\Q}$ can be immediately deduce from the corresponding information about $\Delta$-CM closed poins on $X_0(M,N)_{/\Q}$.
\\ \\
For most of this paper we study $\Delta$-CM points on the curves $X_0(M,N)_{/\Q}$.  For this family of curves we can do 
more: we determine not only the multiplicities and degrees of closed points in the $\Delta$-CM locus but actually the residue fields themselves.  It turns out that there are only two classes of such fields.
As is standard, we call a field $K(J_{\Delta})$ a \textbf{ring class field} (here this is understood to be relative to the fixed imaginary quadratic 
field $K$).  We call a field isomorphic to $\Q(J_{\Delta}$) a \textbf{rational ring class field}.  The following is a preliminary version of one of our main results.

\begin{thm}
\label{ITHM2}
Let $\Delta = \ff^2 \Delta_K$ be an imaginary quadratic discriminant with $\Delta_K < -4$, and let $P \in X_0(M,N)_{/\Q}$ be a 
$\Delta$-CM closed point.  Then $\Q(P)$ is either a rational ring class field or a ring class field.  Moreover there is $N' \mid N$ such that $\Q(P)$ is isomorphic to either $\Q(J_{(N')^2 \Delta})$ or to
$K(J_{(N')^2 \Delta})$.
\end{thm}
\noindent
Our work will determine, for each $N' \mid N$, whether either, both or neither of $\Q(J_{(N')^2 \Delta})$ or
$K(J_{(N')^2 \Delta})$ actually arise up to isomorphism as such a field $\Q(P)$.  
\\ \\
\noindent
Our study of closed CM points on $X_0(M,N)$ comes roughly in four parts.
\\ \\
Step 1: We prove a result (Proposition \ref{PROP3.7}) that reduces us to the case $X_0(\ell^a,\ell^b)_{/\Q}$ for a prime number $\ell$.  This result does \emph{not} hold for the 
curves $X_1(M,N)_{/\Q}$, which gives a clue that we are on the right track by making the curves $X_0(M,N)_{/\Q}$ our primary focus.
\\ \\
Step 2: We are therefore reduced to considering (in general, pairs of) cyclic isogenies of prime power degree.  This places us 
in a position to make use of the fact that the $\ell$-powered isogeny graph on $K$-CM elliptic curves has a very simple structure, 
that of an \textbf{isogeny volcano}.
%
This is probably the single most important ingredient in our analysis, and it is really remarkable the extent to which use of volcanoes reduces difficult problems in arithmetic geometry to either straightforward enumerative combinatorics or simple (though sometimes tedious) bookkeeping.\footnote{It seems a bit strange that up until now volcanoes have mainly been used in the study of elliptic curves over finite fields.  I know of only one paper 
that uses isogeny volcanoes in characteristic zero, a recent one of Rosen-Shnidman \cite{Rosen-Shnidman17}.  Some of the enumerative work we do with volcanoes is related to their work.}
\\ \\
Steps 1 and 2 yield a complete description of the $\Delta$-CM locus on $X_0(M,N)_{/K}$: in this case (as follows from 
our discussion up until now, in fact) there is only one primitive residue field.  Just as in the recent paper \cite{BCII}, the greater 
part of the battle is to descend from modular curves over $K$ to modular curves over $\Q$.  In the present work, this amounts to an explicit understanding of the action of complex conjugation on the isogeny volcano.\footnote{This is not present in the work of \cite{Rosen-Shnidman17}, whereas the portion of the combinatorial analysis described in Step 2 above seems to be equivalent to what they did, 
although recorded somewhat differently.}  This comes in two parts.
\\ \\
Step 3: We develop the algebraic number theory of rational ring class fields.   Whereas the ring class fields are always Galois over $\Q$ but usually not abelian over $\Q$, 
the rational ring class fields are usually not Galois over $\Q$.  A key concept here is that of \textbf{coreality} of $K$-CM $j$-invariants $j,j' \in \C$, which means that the number field $\Q(j,j')$ has a real place.  In particular the compositum of two rational ring class 
fields is sometimes a rational ring class field and sometimes a ring class field, and we compute all such composita explicitly.  
\\ \\
Step 4: Using the algebraic work described in Step 3 we solve the graph-theoretic problem of how the involution of complex 
conjugation acts on the isogeny volcano.   The case of $\ell = 2$ requires a more intricate analysis than that of $\ell > 2$.
To complete the $\ell =2$ case, we make use of 
results of Kwon \cite{Kwon99} on cyclic isogenies.

\subsection{Contents}
We now give a more detailed description of the contents of the paper.
\\ \\
In Section 2 we introduce the algebraic number theory of ring class fields and rational ring class 
fields.  This is related to the genus theory of binary quadratic forms and to the notion of a ``real lattice'' $\Lambda \subset \C$.  
The most important results here are probably Propositions \ref{TEDIOUSALGEBRAPROP1} and \ref{TEDIOUSALGEBRAPROP2}, which will later be used to reduce the computation of 
the $\Delta$-CM locus on $X_0(M,N)$ to the case where $M = \ell^a$, $N = \ell^b$ are both prime powers.  
\\ \\
In Section 3 we study isogenies of elliptic curves in characteristic $0$.  Our initial setup includes the non-CM case.  
In \S 3.4 we recall some structural results on isogenies in the CM case 
with a focus on the conductor of the endomorphism rings.   In \S 3.5 we give the result that reduces us from the study of $X_0(M,N)$ to that of $X_0(\ell^a,\ell^b)$ (Proposition \ref{PROP3.7}).
\\ \\
In Section 4 we introduce the $\ell$-power isogeny graph of complex elliptic curves and explain its ``volcanic'' structure.  We claim no novelty here: all of these results can be found in the literature -- we especially recommend \cite{Sutherland12} -- but because this material is absolutely crucial for the rest of the paper we have decided to give an independent exposition.  
\\ \\
In Section 5, the theoretical heart of the paper, we explicitly determine the action of complex conjugation on the isogeny volcano.  
There is an algebraic preliminary: \S 5.1 on ``coreality,'' which studies when the number field $\Q(j,j')$ generated by two $K$-CM 
$j$-invariants contains $K$.  
\\ \\
Section 6 is a 
bit of an interregnum, in which we show that the results that we have developed so far lead to short, transparent proofs of several 
prior results in the literature, including Kwon's classification of cyclic $N$-isogenies over $\Q(j)$ \cite{Kwon99} and Bourdon-Clark's classification of 
cyclic $N$-isogenies over $K(j)$ \cite[Thm. 6.18]{BCII} (restricted to the $\Delta_K < -4$ case).  
 \\ \indent In \S 6.5 we analyze the projective $N$-torsion field of a CM elliptic curve (when $\Delta_K < -4$), strengthening 
a result of Parish that Bourdon-Clark used to prove Theorem \ref{ITHM0}.  We have already mentioned that for all $N \geq 3$, 
for any $K$-CM elliptic curve $E$ defined over a number field $F$, the $N$-torsion field $F(E[N])$ contains $K$.  Theorem 
\ref{BETTERPARISHTHM} implies that the \emph{projective} $N$-torsion field $F(\P E[N])$ -- i.e., the unique minimal extension of $F$ 
over which the modulo $N$ Galois representation consists entirely of scalar matrices -- already contains $K$.  This is 
the final ingredient we need in order to analyze the $\Delta$-CM locus on $X_0(M,N)_{/\Q}$.  
\\ \\
This analysis is performed in Sections 7 through 9.  Section 7 performs the graph-theoretic enumeration that handles the key case $X_0(\ell^a)$.  
In Section 8 we use the $X_0(\ell^a)$ case and Theorem \ref{BETTERPARISHTHM} to handle the $X_0(\ell^{a'},\ell^a)$ case.  Finally in Section 9 
we compile across prime powers and treat CM points on $X_0(M,N)$.  Although a complete description of the $\Delta$-CM locus on $X_0(M,N)$ seems unrewardingly 
complicated, in every case we explicitly determine all primitive residue fields and all primitive degrees $\Delta$-CM points on $X_0(M,N)$.  Because of Theorem \ref{INERTNESSTHM} 
this yields the corresponding information on $X_1(M,N)$, which, as mentioned above, is exactly what is needed to enumerate torsion subgroups of CM elliptic curves over number fields of any fixed degree.
\\ \\
In Section 10 we give a complete classification of odd degree CM points on $X_0(M,N)$ and $X_1(M,N)$: the latter is an elaboration of work of Bourdon and Pollack \cite{BP17}, but it is interesting 
that it can be deduced from the $X_0(M,N)$ case.  

\subsection{The Non-CM Case}
Although our main focus of this work and its sequel \cite{CS22b} is a complete analysis of isogenies on CM elliptic curves in characteristic $0$, occasionally our work touches upon the non-CM case and gives new results there as well.  In particular:
\\ \\
$\bullet$ Proposition \ref{PROP3.2} is a simple but definitive result on the field of moduli of an isogeny in the non-CM case.  Frankly, I am surprised not to have found this result in the literature.  As discussed in \S 3.3 there is a closely related, but distinct, result of Cremona-Najman \cite[Cor. A.5]{Cremona-Najman21}.
\\ \\
$\bullet$ Theorem \ref{LITTLECLEMMA} concerns cyclic isogenies on real elliptic curves.  It can be viewed as a computation of the 
fiber of the morphism of $\R$-schemes $Y_0(N) \ra Y(1)$ over any $\R$-point $P \in Y(1)$ when $4 \nmid N$ and giving some 
information on this fiber when $4 \mid N$: in particular, for all $N \geq 3$ this fiber always contains a closed point with residue field $\C$.  
For our applications to the CM case this result is, honestly, overkill: we use it only to give a uniform argument rather than appealing to 
our analysis of the isogeny volcano in various cases.  However the material seems interesting: it suggests room for further development 
of the theory of real elliptic curves.  
\\ \\
$\bullet$ Corollary \ref{LASTLVCOR} gives -- conditionally on the Generalized Riemann Hypothesis -- a necessary and sufficient condition 
on a number field $F$ in order for the set of positive integers $N$ such that there is an $F$-rational cyclic $N$-isogeny of elliptic curves to be finite: $F$ must not contain the Hilbert class field of any imaginary quadratic field.

\subsection{Acknowledgments} This work had its genesis in my collaboration with Abbey Bourdon and a key interaction 
with Drew Sutherland.  \\ \indent
Were it not for my prior collaboration with Abbey Bourdon he would neither have thought to pursue this work nor have 
been technically equipped to do so.   Insights gained from the work of \cite{BCI} and \cite{BCII} -- many of which 
came directly from Bourdon -- have been put to use here, both directly and otherwise.
\\ \indent
At the January 2019 AMS meeting in Baltimore, Drew Sutherland saw Bourdon speak 
on the work of \cite{BCII} and immediately suggested a volcanic approach to some of our 
work.  Sutherland's remark amounts to the material of \S 6.2 of the present paper.  This is a completely different proof of \cite[Thm. 6.18a)]{BCI} from the one given in \cite{BCI} and thus served as an illustration of the merits of the volcanic approach.
\\ \indent I thank Filip Najman for an interesting observation about $X_0(2,2N)$: see Remark \ref{NAJMANREMARK}.   I thank Frederick Saia 
for many helpful discussions, that in particular led to a correction in the statement of Proposition \ref{PROP3.7}.  Saia also made the figures in this work, replacing the decidedly rustic hand-drawn figures I had originally made.
\section{Orders, Class Groups, and Rational Ring Class Fields}

\subsection{Orders in a number field}

For a number field $K$ of degree $d$, a ($\Z$-)order in $K$ is a subring $\OO$ of $K$ that is free of rank $d$ as a $\Z$-module 
and has fraction field $K$.  The ring of integers $\Z_K$ is an order in $K$, and every order $\OO$ in $K$ is contained in $\Z_K$ 
with finite index.  Let $\Delta_K$ be the discriminant $K$ (more precisely, of $\Z_K$).  If $f \coloneqq [\Z_K:\OO]$ and $\Delta$ is the discriminant 
of $\OO$ (i.e., the discriminant of the trace form on $\OO$), then we have $\Delta = f^2 \Delta_K$.  
\\ \indent
The \textbf{class group (or Picard group)} $\Pic \OO$ of an order is the group of invertible fractional 
$\OO$-ideals modulo principal fractional $\OO$-ideals.  This is a finite commutative group \cite[Thm. I.12.12]{Neukirch}; its size is the \textbf{class number} 
$h_{\OO}$ of $\OO$.  There is a canonical finite abelian extension $K(\OO)/K$, the \textbf{ring class field} of $\OO$, such that 
$\Aut( K(\OO)/K)$ is canonically isomorphic to $\Pic \OO$ \cite[Thm. 4.2]{Lv-Deng15}.  We write $K(1)$ for $K(\Z_K)$, the 
\textbf{Hilbert class field} of $K$, and we put $h_K \coloneqq [K(1):K]$.  
\\ \\
An inclusion of orders $\OO \subset \OO' \subset K$ yields an inclusion of ring class fields $K(\OO') \subset K(\OO)$.  
Galois theory yields a surjection $\Pic \OO \ra \Pic \OO'$; this is also the map induced by the pushforward $I \mapsto I \OO'$ 
on invertible fractional ideals.  In particular we have $h_K \mid h_{\OO'} \mid h_{\OO}$. 
\\ \indent
For an order $\OO$ in a number field $K$, we define the \textbf{conductor ideal} 
\[\ff = (\OO:\Z_K) = \{x \in K \mid x \Z_K \subset \OO\}, \]
which is characterized as the largest ideal of $\Z_K$ that is contained in $\OO$.  The conductor of the abelian extension 
$K(\ff)/K$ divides $\ff$ \cite[Thm. 4.2]{Lv-Deng15}.  The conductor ideal also appears in the relative class number formula 
\cite[Thm. I.12.12]{Neukirch}
\begin{equation}
\label{CLASSNUMBERFORMULA}
\frac{h_{\OO}}{h_K} = \frac{ \# (\OO/\ff)^{\times}}{[\Z_K^{\times}:\OO^{\times}] \# (\Z_K/\ff)^{\times}}. 
\end{equation}
A nonzero fractional $\OO$-ideal $\aa$ is \textbf{proper} if 
\[ (\aa:\aa) \coloneqq \{x \in K \mid x \aa \subset \aa\} = \OO. \]
If $\aa$ is an invertible fractional ideal then $x \aa \subset \aa$ iff $x \in \aa \aa^{-1} = \OO$, so $\aa$ is a proper $\OO$-ideal.  
For an order $\OO$, every proper fractional $\OO$-ideal is invertible iff $\OO$ is a Gorenstein ring \cite[Characterization 4.2]{Jensen-Thorup15}, and an order is Gorenstein if it is monogenic over $\Z$, i.e., if $\OO = \Z[\alpha]$ for some $\alpha \in \OO$ \cite[Thm. 4.3]{Jensen-Thorup15}.

\subsection{Imaginary quadratic orders}
Henceforth we suppose that $K$ is an imaginary quadratic field.  This vastly simplifies the structure of orders in $K$:\footnote{Everything that we say in this subsection for imaginary quadratic orders 
holds verbatim for real quadratic orders, except that the discriminant is positive and the unit group is infinite.} if $[\Z_K:\OO] = f$ then 
\[ \OO = \Z + f \Z_K, \]
and it follows that $\ff = f \Z_K$.  (However $\ff$ is not principal, nor even invertible, as an ideal of $\OO$.)   Conversely, for 
any $f \in \Z^+$, we have that $\Z + f \Z_K$ is an order in $K$ with index $f$ and conductor ideal $f \Z_K$.  Because of this simple 
relationship between $f$ and $\ff$ in the quadratic case, from now on we will write $\ff$ for the positive integer $[\Z_K:\OO]$.   The discriminant of $\OO$ 
is $\ff^2 \Delta_K$, which is a negative integer congruent to $0$ or $1$ modulo $4$.  Distinct imaginary quadratic orders have different discriminants.  Conversely, if $\Delta$ is a negative 
integer congruent to $0$ or $1$ modulo $4$, we put
\[ \tau_{\Delta} \coloneqq \frac{\Delta + \sqrt{\Delta}}{2}, \]
and then $\Z[\tau_{\Delta}]$ is an order in $\Q(\sqrt{\Delta})$ of discriminant $\Delta$.  It follows that every imaginary quadratic order $\OO$ is monogenic, hence Gorenstein: proper fractional $\OO$-ideals are invertible.  
\\ \\
We denote the class number of the order of discriminant $\Delta$ 
by $h_{\Delta}$.  
\\ \\
If $\OO$ is an imaginary quadratic order of discriminant $\Delta$, we put 
\[ w_{\Delta} \coloneqq \# \OO^{\times} = \begin{cases} 6 & \Delta = -3 \\
4 & \Delta = -4 \\ 2 & \Delta < -4 \end{cases}. \]
We also put $w_K \coloneqq w_{\Delta_K}$.  

\subsection{Ring class fields}
For an imaginary quadratic field $K$ and $\ff \in \Z^+$, we denote by $K(\ff)$ the ring class field of the unique order in $K$ of 
conductor $\ff$.  We have -- either as a consequence of (\ref{CLASSNUMBERFORMULA}) or by 
\cite[Cor. 7.24]{Cox89}) -- that
\begin{equation}
\label{QUADCLASS}
\mathfrak{d}(\ff) \coloneqq [K(\ff):K(1)] = \begin{cases} 1 & \ff = 1 \\
\frac{2}{w_K} \ff \prod_{\ell \mid \ff} \left(1 - \left( \frac{\Delta_K}{\ell} \right) \frac{1}{\ell} \right) & \ff \geq 2
\end{cases}.
\end{equation}
\noindent
For fixed $K$, the function $\mathfrak{d}$ is multiplicative in $\ff$ iff $\Delta_K < -4$.
\\ \\
From (\ref{QUADCLASS}) we deduce the following formulas that will be useful later on.

\begin{cor}
\label{LCLASSCOR}
Let $K$ be an imaginary quadratic field with $\Delta_K \in \{-4,-3\}$, let $\ff \in \Z^+$, and let $\ell$ be a prime. 
\begin{itemize}
\item[a)] If $\ff^2 \Delta_K = -3$, then 
\[ [K(\ell \ff):K(\ff)] = [\Q(\sqrt{-3})(\ff):\Q(\sqrt{-3})] = 
\begin{cases} \frac{\ell-1}{3} & \ell \equiv 1 \pmod{3} \\
1 & \ell = 3 \\ \frac{\ell+1}{3} & \ell \equiv 2 \pmod{3} \end{cases}. \]
\item[b)] If $\ff^2 \Delta_K = -4$, then 
\[ [K(\ell \ff):K(\ff)] = [\Q(\sqrt{-1})(\ff):\Q(\sqrt{-1})] = 
\begin{cases} \frac{\ell-1}{2} & \ell \equiv 1 \pmod{4} \\ 1 & \ell = 2 \\ \frac{\ell+1}{2} & \ell \equiv 3 \pmod{4}
\end{cases} . \]
\item[c)] If $\ff^2 \Delta_K < -4$, then 
\[ [K(\ell \ff):K(\ff)] = 
\begin{cases} 
\ell -1 & \left( \frac{\ff^2 \Delta_K}{\ell} \right) = 1 \\ 
\ell & \left( \frac{\ff^2 \Delta_K}{\ell} \right) = 0 \\
\ell +1 & \left( \frac{\ff^2 \Delta_K}{\ell} \right) = -1
\end{cases}. \]
\end{itemize}
\end{cor}

\begin{prop}
\label{TEDIOUSALGEBRAPROP1}
Suppose $\Delta_K < -4$.  Let $\ff_1,\ff_2 \in \Z^+$, and put $m = \operatorname{gcd}(\ff_1,\ff_2)$ and $M = \operatorname{lcm}(\ff_1,\ff_2)$.  
As extensions of $K(m)$, the fields $K(\ff_1)$ and $K(\ff_2)$ are linearly disjoint and have compositum $K(M)$.
\end{prop}
\begin{proof}
Step 1: First we suppose that $m  = 1$.  For $i = 1,2$, the conductor of the abelian extension $K(\ff_i)/K$ divides $\ff_i$, so 
the conductor of $K(\ff_1) \cap K(\ff_2)$ divides both $\ff_1$ and $\ff_2$, hence it divides $m  = 1$, and it follows that 
$K(\ff_1) \cap K(\ff_2) = K(1)$.  Since $K(\ff_1)/K(1)$ is Galois, this implies the linear disjointness.  \\ \indent
 Certainly we have $K(\ff_1)K(\ff_2) \subset K(\ff_1 \ff_2)$.  Conversely, using the linear disjointness 
and the multiplicativity of $\mathfrak{d}$ we get
\[[K(\ff_1)K(\ff_2):K(1)] = [K(\ff_1):K][K(\ff_2):K] =  \mathfrak{d}(f_1) \mathfrak{d}(f_2) = \mathfrak{d}(f_1 f_2) = [K(\ff_1 \ff_2):K(1)], \]
so $K(\ff_1)K(\ff_2) = K(\ff_1 \ff_2)$. \\
Step 2: Again, we certainly have $K(\ff_1) K(\ff_2) \subset K(M)$. Now write $\ff_1 = \prod_{i=1}^r \ell_i^{a_i}$, $\ff_2 = \prod_{i=1}^r 
\ell_i^{b_i}$, so $M = \prod_{i=1}^r \ell_i^{\max(a_i,b_i)}$.  For all $1 \leq i \leq r$ the field $K(\ff_1) K(\ff_2)$ contains 
both $K(\ell_i^{a_i})$ and $K(\ell_i^{b_i})$ hence also $K(\ell_i^{\max(a_i,b_i)})$.  Using Step 1 and an easy induction, we get
\[ K(\ff_1) K(\ff_2) \supset K(\ell_1^{\max(a_1,b_1)}) \cdots K(\ell_r^{\max(a_r,b_r)}) = K(M). \]
Step 3: Since $K(\ff_1)$, $K(\ff_2)$ are Galois over $K$, they are linearly disjoint over $K(\ff_1) \cap K(\ff_2)$, so
\[ [K(M):K(\ff_1) \cap K(\ff_2)] = [K(\ff_1)K(\ff_2):K(\ff_1) \cap K(\ff_2)] \] \[ = [K(\ff_1):K(\ff_1) \cap K(\ff_2)][K(\ff_2):K(\ff_1) \cap K(\ff_2)]. \]
We claim that
\[ [K(M):K(m)] = [K(\ff_1):K(m)][K(\ff_2):K(m)]. \]
Since $K(m) \subset K(\ff_1) \cap K(\ff_2)$, the claim implies that $K(m) = K(\ff_1) \cap K(\ff_2)$, which is sufficient to complete the proof.  The claim equivalent to the identity 
\[ \dd(m) \dd(M) = \dd(\ff_1) \dd(\ff_2) \]
which holds for any multiplicative arithmetic function: using multiplicativity we reduce to $\ff_1 = \ell^a$, $\ff_2 = \ell^b$, in which 
case the claim is clear. 
\end{proof}
\noindent
When $\Delta_K \in \{-4,-3\}$, we still have the linear disjointness of $K(\ff_1)$ and $K(\ff_2)$ over $K(m)$, but the compositum $K(\ff_1)K(\ff_2)$ can be a proper subfield of $K(M)$ \cite[Prop. 2.1]{CS22b}.

\subsection{The connection with CM elliptic curves}
Let $E_{/\C}$ be an elliptic curve.  There is a lattice $\Lambda$ in $\C$, unique up to homothety, such that 
$E \cong \C/\Lambda$, and then we have
\[ \End(E) = \{\alpha \in \C \mid \alpha \Lambda \subset \Lambda\}. \]
We say that $E$ has \textbf{complex multiplication} if $\End(E)$ properly contains $\Z$, in which case it must be an imaginary 
quadratic order $\OO$ \cite[Cor. III.9.4]{SilvermanI}.  We say that $E$ has $\OO$-CM.  If $K$ is the fraction field of $\OO$ we also say that $E$ has $K$-CM.  
\\ \\
Let $\OO$ be an imaginary quadratic order, of discriminant $\Delta$.  Every $\OO$-CM elliptic curve is uniformized by a lattice $\Lambda \subset K$, i.e., such that $\Lambda$ is a fractional $\OO$-ideal.  
The uniformizing lattice $\Lambda$ must moreover be a proper (equivalently, invertible) $\OO$-ideal.  From this we deduce a bijection
from $\Pic \OO$ to the set of $\C$-isomorphism classes of $\OO$-CM elliptic curves: in particular there are $h_{\Delta}$ $\OO$-CM 
$j$-invariants.  The identity of $\Pic \OO$ corresponds to the elliptic curve $\C/\OO$, and we put 
\[ j_{\Delta} \coloneqq j(\C/\OO). \]
If $\OO$ is an order in $K$ of conductor $\ff$, then we have \cite[Thm. 11.1]{Cox89}
\begin{equation}
\label{KJEQ}
K(\ff) = K(j(E)). 
\end{equation}

\subsection{Reality, part I: real moduli} Complex conjugation acts on lattices in 
$\C$: 
\[ \Lambda \mapsto \overline{\Lambda} \coloneqq \{\overline{z} \mid z \in \Lambda\}. \]
A lattice $\Lambda \subset \C$ is \textbf{real} if $\overline{\Lambda} = \Lambda$.  

\begin{lemma}
\label{LAMBDAC}
Let $\Lambda$ be a lattice in $\C$.  Then we have $j(\C/\overline{\Lambda}) = \overline{j(\C/\Lambda)}$. 
\end{lemma}
\begin{proof}
The $j$-invariant of a lattice depends only on its homothety class, so we may assume that $\Lambda = \Z 1 \oplus \Z \tau$ for 
some $\tau \in \mathcal{H}$, and then $\overline{\Lambda} = \Z 1 \oplus \Z \overline{\tau} = \Z 1 \oplus \Z -\overline{\tau}$.  
Since $j(\tau) \in \Q(e^{2 \pi i \tau}) \subset \R(e^{2 \pi i \tau})$ and $e^{2\pi i (- \overline{\tau})} = \overline{e^{2 \pi i \tau}}$, 
we have 
\[ j(\C/\overline{\Lambda}) = j(-\overline{\tau}) = \overline{j(\tau)} = j(\C/\Lambda).  \qedhere\]
\end{proof}
\noindent
Let $\aa$ be a proper $\OO$-ideal.  Then we have \cite[Lemma 7.14]{Cox89} \begin{equation}
\label{NORMEQ}
\aa \overline{\aa} = |\aa| \OO,
\end{equation}
where $|\aa| = \# \OO/\aa$.  It follows that $[\overline{\aa}] = [\aa]^{-1}$ in $\Pic \OO(\Delta)$.   


\begin{lemma}
\label{BCS3.2LEMMA}
\cite[Lemma 3.2a)]{BCS17} For an elliptic curve $E_{/\C}$, the following are equivalent: \\
(i) There is an elliptic curve $(E_0)_{/\R}$ such that $(E_0)_{/\C} \cong E$.  \\
(ii) We have $j(E) \in \R$.  \\
(iii) There is a real lattice $\Lambda \subset \C$ such that $E \cong \C/\Lambda$. \\
An elliptic curve satisfying these equivalent conditions is said to be \textbf{real}.  
\end{lemma}
\noindent
From this we deduce:

\begin{cor}
\label{COR2.5}
\cite[Lemma 3.4]{BCS17}
For a proper fractional $\OO$-ideal $\aa$, the following are equivalent: \\
(i) The elliptic curve $\C/\aa$ is real.  \\
(ii) The ideal class $[\aa]$ is real: $[\overline{\aa}] = [\aa]$.  \\
(iii) The fractional ideal $\aa^2$ is principal, i.e., $[\aa] \in \Pic \OO(\Delta)[2]$.
\end{cor}
\noindent
These equivalent conditions certainly hold when $\aa$ is a real ideal.  Starting with a real ideal $\aa$ and scaling by $\alpha \in K^{\times}$ yields, in general, an ideal that is not real in the same class.  However, by \cite[Lemma 3.6]{BCS17} every real ideal class contains an integral ideal $\aa$ that is real, proper and \textbf{primitive}: i.e., such that the isogeny $\C/\OO(\Delta) \ra \C/\aa^{-1}$ is cyclic.  
\\ \\
For an order $\OO$ in a number field $K$, letting $z \mapsto \overline{z}$ denote complex conjugation, we have $\overline{K(\OO)} = \overline{K}(\overline{\OO})$.  Since for an imaginary quadratic order $\OO$ we have $\overline{\OO} = \OO$, this shows that $K(\ff)$ is stable under complex 
conjugation, which acts nontrivially as it does so on the subfield $K$.  Thus $\Aut(K(\ff)/K)$ is a proper subgroup of $\Aut(K(\ff)/\Q)$, from 
which it follows that $K(\ff)/\Q$ is Galois.  By (\ref{NORMEQ}) complex conjugation acts on $\Aut(K(\ff)/K)$ as inversion, and this 
yields an isomorphism of $\Aut(K(\ff)/\Q)$ with the semidirect product $\Pic \OO \rtimes \langle c \rangle$.  
\\ \\
Once again we have $\overline{\OO} = \OO$, and thus $j_{\Delta} = j(\C/\OO) \in \R$.  This shows that $\Q(\ff) \subset K(\ff)^c$, 
and since both are number fields of degree $h_{\Delta}$, we have $\Q(\ff) = K(\ff)^c$.
From this we get:

\begin{prop}
The number of roots $j$ of $H_{\Delta}$ that lie in $\Q(\ff)$ is $\# (\Pic \OO)[2]$.
\end{prop}

\begin{cor}
\label{IDONEAL}
\label{2.7}
For an imaginary quadratic discriminant $\Delta$, the following are equivalent: 
\begin{itemize}
\item[(i)] The extension $\Q(\ff)/\Q$ is Galois. 
\item[(ii)] The extension $\Q(\ff)/\Q$ is totally real.  
\item[(iii)] Every $\Delta$-CM $j$-invariant lies in $\R$. 
\item[(iv)] We have $\Pic \OO = (\Pic \OO)[2]$.
\end{itemize}
\end{cor}
\begin{proof}
The only part that does not follow immediately is (iv) $\implies$ (i).  For this: if $\Pic \OO$ is $2$-torsion, then $\Pic \OO \rtimes \langle c \rangle$ is actually a direct product, so $K(\ff)/\Q$ is abelian and thus every subextension is Galois.
\end{proof}
\noindent
It follows easily from work of Heilbronn that as we range over all imaginary quadratic orders, the ratio $\frac{ \# \Pic \OO}{\# (\Pic \OO)[2]}$ tends to $\infty$ with $|\Delta|$: see \cite[Lemma 2.2]{CGPS} for a more quantitatively precise version of this.   In \cite{Voight07}, 
Voight supplies a list of $101$ imaginary quadratic discriminants $\Delta$ satisfying the equivalent conditions of Corollary \ref{IDONEAL} 
and shows that the Generalized Riemann Hypothesis (GRH) implies that this list is complete.

\begin{lemma}
\label{GENUSLEMMA}
\label{LEMMA2.8}
For an imaginary quadratic discriminant $\Delta$, let $r$ be the number of distinct
odd prime divisors of $\Delta$.  We define $\nu \in \N$ as follows: \\
\[\nu =
\begin{cases}
r-1, \ \Delta \equiv 1 \pmod{4} \text{ or }  \Delta \equiv 4 \pmod{16} \\
r, \ \Delta \equiv 8,12 \pmod{16} \text{ or } \Delta \equiv 16 \pmod{32} \\
 r+1, \ \Delta \equiv 0 \pmod{32}.
\end{cases} \]
Then we have $\Pic \OO(\Delta)[2] \cong (\Z/2\Z)^{\nu}$. 
\end{lemma}
\begin{proof}
This is essentially due to Gauss and is part of the genus theory of binary quadratic forms.  For a modern treatment see
\cite[Prop. 3.11]{Cox89} or \cite[Thm. 5.6.11]{Halter-Koch}.  
\end{proof}
\noindent
A nonzero $\OO$-ideal $I$ is \textbf{primitive} if the additive group of $\OO/I$ is cyclic.  This holds if and only if there is no $N \geq 2$ such that $N^{-1} I$ remains an integral $\OO$-ideal.

\begin{thm}
\label{KWONLEMMA3.1} 
Let $\OO$ be an imaginary quadratic order of discriminant $\Delta = \ff^2 \Delta_K$.  
\begin{enumerate}
\item If there is a primitive, proper real $\OO$-ideal of index $N$, then $N \mid \Delta$.  
\item Let $N = \ell_1^{a_1} \cdots \ell_r^{a_r}$.  There is a primitive, proper real $\OO$-ideal $I$ such that $[\OO:I] = N$ if and only if for all $1 \leq i \leq r$, there is a primitive,
proper real $\OO$-ideal $I_i$ such that $[\OO:I_i] = \ell_i^{a_i}$.  
\item Let $\ell > 2$, and let $a \in \Z^+$.  There is a primitive, proper real $\OO$-ideal $I$ such that $[\OO:I] = \ell^a$ if and only if
$a = \ord_{\ell} (\Delta)$.  
\item Let $\ell = 2$, and let $a \in \Z^+$. 
\begin{enumerate}
\item[i)] Suppose $16 \mid \Delta$.  Then there is a primitive, proper real $\OO$-ideal $I$ such that $[\OO:I] = 2^a$ if and only if
$a = 2$ or $a = \ord_2(\Delta)-2$.  
\item[ii)] Suppose $2 \mid \Delta$ and $16 \nmid \Delta$.  Then there is a primitive, proper real $\OO$-ideal $I$ such that $[\OO:I] =2^a$ if and only if
$a = 1$.
\end{enumerate}
\end{enumerate}
\end{thm}
\begin{proof}
This is essentially due to S. Kwon \cite{Kwon99}.  We have given a somewhat more explicit treatment following \cite[Lemma 5.6]{BCII}.
\end{proof}

\subsection{Rational ring class fields}

Let $\Delta = \ff^2 \Delta_K$ be an imaginary quadratic discriminant.  The \textbf{Hilbert class polynomial} $H_{\Delta}(t) \in \C[t]$ is the monic polynomial 
whose roots are the $j$-invariants of $\Delta$-CM elliptic curves.  It is known that $H_{\Delta} \in \Z[t]$ and that $H_{\Delta}$ is irreducible over $K$ \cite[\S 13]{Cox89}, so $K[t]/(H_{\Delta}) \cong K(\ff)$.  
\\ \\
The Hilbert class polynomial $H_{\Delta}$ is in particular irreducible over $\Q$ so is the minimal polynomial of $j_{\Delta}$.  We 
define the \textbf{rational ring class field} 
\[ \Q(\ff) = \Q(j_{\Delta}) \cong \Q[t]/(H_{\Delta}). \]
(Our notation suppresses $K$.  It might at first seem better to index the rational ring class field by its discriminant, 
from which we can recover $K$.  But in practice the $\Q(\ff)$ notation, which was introduced in \cite{BCI} and also used in \cite{BCII}, 
seems more convenient.)
\\ \\
For fixed $K$, the rational ring class fields form a directed system in $\ff$: if $\ff_1 \mid \ff_2$ then $\Q(\ff_1) \subseteq \Q(\ff_2)$.  
To see this, let $\OO$ and $\OO'$ be the orders in $K$ of conductor $\ff_1$ and $\ff_2$, so $\OO' \subset \OO$.  It follows 
from \cite[\S 2.6]{BCII} that the isogeny $\C/\OO' \ra \C/\OO$ can be defined over $\Q(j(\C/\OO')) = \Q(\ff_2)$.  Therefore 
$\Q(\ff_2) \supset \Q(j(\C/\OO)) = \Q(\ff_1)$.  
\\ \\
Next we give analogue of Proposition \ref{TEDIOUSALGEBRAPROP1} for rational ring class fields.  

\begin{prop}
\label{TEDIOUSALGEBRAPROP2}
\label{SAIAPROP2}
Suppose $\Delta_K < - 4$.  Let $\ff_1,\ff_2 \in \Z^+$, and put $m = \operatorname{gcd}(\ff_1,\ff_2)$ and $M = \operatorname{lcm}(\ff_1,\ff_2)$. 
\begin{itemize}
\item[a)] The fields $\Q(\ff_1)$ and $\Q(\ff_2)$ are linearly disjoint over $\Q(m)$, and 
\[ \Q(M) = \Q(\ff_1) \Q(\ff_2) \cong \Q(\ff_1) \otimes_{\Q(m)} \Q(\ff_2). \]
\item[b)] We have 
\[ \Q(\ff_1) \otimes_{\Q(\ff)} K(\ff_2) \cong K(M). \]
\item[c)] We have 
\[K(\ff_1) \otimes_{\Q(\ff)} K(\ff_2) \cong K(M) \times K(M). \]
\end{itemize}
\end{prop}
\begin{proof}
a) The $\Q(m)$-algebra $E \coloneqq \Q(\ff_1) \otimes_{\Q(m)} \Q(\ff_2)$ is \'etale, hence isomorphic to a product of field extensions of $\Q(m)$.  Extending scalars from $\Q(m)$ to $K(m)$, we get the $K(m)$-algebra $K(\ff_1) \otimes_{K(m)} K(\ff_2)$ \cite[Thm. 6.22]{Conrad2}, which by Proposition \ref{TEDIOUSALGEBRAPROP1} is isomorphic to the field $K(M)$.  It follows that $E$ is a field, which shows that $\Q(\ff_1)$ and $\Q(\ff_2)$ are linearly disjoint over $\Q(m)$.  Thus the subextension $\Q(\ff_1)\Q(\ff_2)$ 
of $\Q(M)/\Q(m)$ has degree 
 \[\frac{\dd(\ff_1)}{\dd(m)} \frac{\dd(\ff_2)}{\dd(m)} = \frac{\dd(M)}{\dd(m)} = [\Q(M):\Q(m)], \] so 
$\Q(\ff_1)\Q(\ff_2) = \Q(M)$. \\
b) The $\Q(m)$-bilinear map $\Q(\ff_1) \times K(\ff_2) \ra K(M)$ given by $(x,y) \mapsto xy$ induces a surjective homomorphism
$\Phi: \Q(\ff_1) \otimes_{\Q(m)} K(\ff_2) \ra K(M)$ of $\Q(m)$-algebras of equal, finite dimension, so $\Phi$ is an isomorphism.  \\
c)  Using part b), we have 
\[ K(\ff_1) \cong \Q(\ff_1) \otimes_{\Q(m)} K(m), \ K(\ff_2) \cong \Q(\ff_2) \otimes_{\Q(m)} K(m), \]
so 
\[ K(\ff_1) \otimes_{\Q(m)} K(\ff_2) \cong (\Q(\ff_1) \otimes_{\Q(m)} K(m))_{\otimes_{\Q(m)}} (\Q(\ff_2) \otimes_{\Q(\ff)} K(m)) \]
\[ \cong   (\Q(\ff_1) \otimes_{\Q(m)} \Q(\ff_2))\otimes_{\Q(m)} (K(M) \otimes_{\Q(m)} K(M))\]
\[ \cong \Q(M) \otimes_{\Q(m)} (K(M) \times K(M)) \cong K(M) \times K(M). \qedhere \]
\end{proof}
\noindent
The analogue of Corollary \ref{TEDIOUSALGEBRAPROP2} for $\Delta_K \in \{-4,-3\}$ is \cite[Prop. 2.2]{CS22b}.

\section{Isogenies of Elliptic Curves}
\noindent
In this section we recall some basic facts and establish some results on isogenies of elliptic curves in characteristic $0$.  Some of our results  pertain all elliptic curves, and one result (Proposition \ref{PROP3.2}), focus specifically on elliptic curves \emph{without} complex multiplication.  
\\ \\
For an elliptic curve $E$ defined over a field $K$ of characteristic $0$ with algebraic closure $\overline{K}$ and $N \in \Z^+$, we identify the finite \'etale $K$-group scheme $E[N]$ with the $\Aut(\overline{K}/K)$-module $E[N](\overline{K})$.

\subsection{Basic facts on isogenies}
Let $F$ be a subfield of $\C$.  For $i = 1,2$, let $\iota_i: E_i \ra E_i'$ be $F$-rational isogenies of elliptic curves.  We say 
that $\iota_1$ and $\iota_2$ are \textbf{isomorphic} if there are $F$-rational isomorphisms $\alpha: E_1 \ra E_2$ and $\beta: 
E_1' \ra E_2'$ such that 
\[ \iota_2 =      \beta \circ \iota_1 \circ \alpha^{-1}. \]
Two isogenies $\iota_1$ and $\iota_2$ are isomorphic over the algebraic closure $\overline{F}$ of $F$ in $\C$ if and only if 
they are isomorphic over $\C$.  
\\ \\
Let $\iota: E \ra E'$ be an $F$-rational isogeny of degree $N$.  Then its kernel $K$ is an $F$-rational finite (necessarily \'etale, since we are in characteristic zero) subgroup scheme 
of $E$ of order $N$.   Let $q: E \ra E/K$ be the quotient map.  Then there is an $F$-isomorphism $\beta: E/K \ra E'$ such that 
$\iota = \beta \circ q$, so $q$ is isomorphic to $\iota$ over $F$.  Now let $\iota_i: E \ra E_i'$ for $i = 1,2$ -- that is, 
this time the two source elliptic curves are the same -- and for $i =1,2$, let $K_i = \Ker \iota_i$.  If $K_1 = K_2 = K$ say, 
then $\iota_1$ and $\iota_2$ are both isomorphic to $E \ra E/K$ and thus to each other.  Conversely, if $\iota_1$ and $\iota_2$ 
are isomorphic then there is $\alpha \in \Aut E$ and $\beta: E_1' \ra E_2'$ such that $\iota_2 = \beta \circ \iota_1 \circ \alpha^{-1}$.  
If $\alpha = \pm 1$, then $K_1 = K_2$.  It follows that if $\Aut_F(E) = \{ \pm 1\}$ then isomorphism classes of $F$-rational isogenies with 
source elliptic curve $E$ correspond bijectively to finite $F$-subgroup schemes of $E$.  When $j = 0,1728$ and $F$ 
contains $\Q(\sqrt{-3})$ (resp. $\Q(\sqrt{-1})$) we get that isomorphism classes of $F$-rational isogenies with source elliptic curve 
$E$ correspond bijectively to $\Aut(E)$-orbits of finite $F$-subgroup schemes of $E$.  


\subsection{Factorization of isogenies}

Let $\iota: E \ra E'$ be an isogeny of complex elliptic curves, with kernel $K$.  There are positive integers $M \mid N$ such that 
\[ K \cong \Z/M\Z \times \Z/N\Z. \]
We say that $\iota$ is \textbf{cyclic} if $M = 1$.
\\ \indent
Let $\iota': E \ra E''$ be another isogeny, with kernel $K'$.  Then $\iota$ factors through $\iota'$ -- i.e., 
there is an isogeny $\alpha: E'' \ra E'$ such that $\iota = \alpha \circ \iota'$ -- if and only if $K' \subset K$, and if this holds 
then $\Ker \alpha \cong K/K'$.  In particular, since 
$E[M] \subset K$, we get a factorization 
\[ \iota = \iota_{\cyc} \circ [M], \]
where $\iota_{\cyc}: E \ra E'$ is a cyclic $\frac{N}{M}$-isogeny.  
\\ \\
Let $\iota: E \ra E'$ be a cyclic $N$-isogeny, and factor $N = \ell_1 \cdots \ell_r$ into a product of not necessarily distinct primes which need not be in nondecreasing order.  Put $E_1 = E$ and $E_{r+1} = E'$.  We get a unique factorization $\iota = \iota_r \circ \cdots \circ \iota_1$ with 
$\iota_i: E_i \ra E_{i+1}$ an $\ell_i$-isogeny.

\subsection{The field of moduli of an isogeny}

An isogeny $\iota: E \ra E'$ of complex elliptic curves has a \textbf{field of moduli} $\Q(\iota)$, characterized by each of the following properties:  
\\ \\
$\bullet$ Let $H(\iota)$ be the subgroup of $\Aut \C$ such that for all $\sigma \in H(\iota)$, the isogeny $\iota$ is isomorphic over $\C$ 
to $\sigma(\iota): \sigma(E) \ra \sigma(E')$.  Then $\Q(\iota) = \C^{H(\iota)}$.  \\
$\bullet$ If $E$, $E'$ and $\iota$ are defined over a subfield $F$ of $\C$, then $\Q(\iota) \subset F$, and there is a model 
of $E$, $E'$ and $\iota$ defined over $\Q(\iota)$. \\
$\bullet$ If $\iota$ is moreover cyclic of degree $N$, then $\Q(\iota)$ is isomorphic to the residue field of the induced point on the modular curve $X_0(N)$.  
\\ \\
For any isogeny $\iota: E \ra E'$, let $\iota^{\vee}: E' \ra E $ be the dual isogeny: the unique isogeny such that $\iota^{\vee} \circ 
\iota = [\deg(\iota)]$.  
\\ \\
If $\iota_1: E_1 \ra E_2$ and $\iota_2: E_2 \ra E_3$ are isogenies of complex elliptic curves, then 
\[ \Q(\iota_2 \circ \iota_1) \subset \Q(\iota_1)\Q(\iota_2). \]
In general we need not have equality: e.g. if $\iota_1: E_1 \ra E_2$ is an isogeny such that $\Q(\iota) \supsetneq \Q(j(E_1))$, then 
taking $\iota_2: E_2 \ra E_1$ to be the dual isogeny, we have $\Q(\iota_2 \circ \iota_1) = \Q([\deg \iota_1]) = \Q(j(E_1))$.  However:

\begin{prop}
\label{PROP3.1}
Let $E,E',E_1,E_2,E_3$ be complex elliptic curves.
\label{NEWISOGENYPROP}
\begin{itemize}
\item[a)] For an isogeny $\iota: E \ra E'$ with associated cyclic isogeny $\iota_{\cyc}: E \ra E'$, we have $\Q(\iota) = \Q(\iota_{\cyc})$.  
\item[b)] Let $\iota: E_1 \ra E_3$ is a cyclic isogeny.  Suppose $\iota = \iota_2 \circ \iota_1$ with $\iota_1: E_1 \ra E_2$ and $\iota_2: E_2 \ra E_3$.  Then:
\begin{itemize}
\item[(i)]
We have $\Q(\iota) \supseteq \Q(\iota_1)\Q(\iota_2)$.
\item[(ii)] If $j(E_2) \notin \{0,1728\}$, then $\Q(\iota) = \Q(\iota_1)\Q(\iota_2)$.
\end{itemize}
\end{itemize}
\end{prop}
\begin{proof}
a) Suppose that $\Ker \iota \cong \Z/M\Z \times \Z/N\Z$.  Then as above we have $\iota = \iota_{\cyc} \circ [M]_E$, 
where $\iota_{\cyc}: E \ra E'$ is cyclic of degree $\frac{N}{M}$.  Then we have $\Q(\iota)  \Q(\iota_{\cyc})$.  
Indeed, as above we have $\Q(\iota_{\cyc}) = \Q(\iota_{\cyc})\Q([M]_E) \supseteq \Q(\iota)$.  Conversely, there is a model of 
$E$ defined over $\Q(\iota)$ and a $\Q(\iota)$-subgroup scheme $K$ of $E$ such that $E/K \cong_{\C} E'$.  Then the morphism 
$E/E[M] \ra E/K$ gives a $\Q(\iota)$-rational model of $\iota_{\cyc}$.   \\
b) As above it is clear that $\Q(\iota) \supseteq \Q(\iota_1)\Q(\iota_2)$.  Because $\Q(\iota_1)\Q(\iota_2) \supseteq \Q(\iota_1)$, there is a 
model for $\iota_1$ defined over $\Q(\iota_1)\Q(\iota_2)$: in particular, this determines $\Q(\iota_1)\Q(\iota_2)$-rational models 
$\mathbf{E}_1$ and $\mathbf{E}_2$ of $E_1$ and of $E_2$.    Similarly, since $\Q(\iota_1)\Q(\iota_2) \supseteq \Q(\iota_2)$, there 
is a model for $\iota_2$ defined over $\Q(\iota_1)\Q(\iota_2)$: in particular this determines $\Q(\iota_1)\Q(\iota_2)$-rational models 
$\mathcal{E}_2$ and $\mathcal{E}_3$ of $E_2$ and of $E_3$.  Because $(\End \ E_2)^{\times} = \{ \pm 1\}$, the elliptic curves 
$(\mathbf{E}_2)_{/\Q(\iota_1)\Q(\iota_2)}$ and $(\mathcal{E}_2)_{/\Q(\iota_1)\Q(\iota_2)}$ are quadratic twists of each other.  
Since quadratic twists do not change the kernel of an isogeny, $\mathbf{E}_2 \ra \mathbf{E}_2/(\Ker (\mathcal{E}_2 \ra \mathcal{E}_3))$ 
is a model of $\iota_2$ and thus \[\mathbf{E}_1 \ra \mathbf{E}_2 \ra \mathbf{E}_2/(\Ker (\mathcal{E}_2 \ra \mathcal{E}_3))\] is a 
$\Q(\iota_1)\Q(\iota_2)$-rational model for $\iota$.  
\end{proof}

\begin{remark}
\label{REMARK3.1}
\textbf{}
\begin{itemize}
\item[a)] Let $K = \Q(\sqrt{-3})$.  There is a $2$-isogeny $\iota_1: E_1 \ra E_2$ with $E_1$ a $-12$-CM elliptic curve, $E_2$ a $-3$-CM 
elliptic curve, and field of moduli $\Q(\iota_1) = \Q$ \cite[Prop. 2.2]{BP17}.  And there is a cyclic $9$-isogeny $\iota_2: E_2 \ra E_3$ 
with field of moduli $\Q(\iota_2) = \Q$ \cite[Prop. 5.10b)]{BCII}.  But by \cite[Thm. 6.18a)]{BCI} there is no cyclic $18$-isogeny $\iota$ with source elliptic curve having $-12$-CM and with $\Q(\iota) \subseteq K$.  Thus 
\[ \Q(\iota_2 \circ \iota_1) \supsetneq \Q = \Q(\iota_1)\Q(\iota_2) \text{ and } \Q(\iota_2 \circ \iota_1) \neq K = K(\iota_1)K(\iota_2). \]
The work \cite{CS22b} will provide an explanation, using the fact that $K(6) \supsetneq K(2)K(3)$.  
\item[b)] Let $K = \Q(\sqrt{-1})$.  There is a $2$-isogeny $\iota_1: E_1 \ra E_2$ with $E_1$ a $-16$-CM elliptic curve, $E_2$ a 
$-4$-CM elliptic curve, and field of moduli $\Q(\iota_1) = \Q$ \cite[Prop. 2.2]{BP17}.  And there is a cyclic $4$-isogeny $\iota_2: E_2 \ra E_3$ with field of moduli $\Q(\iota_2) = \Q$ \cite[Prop. 5.9]{BCII}.  However, by \cite[Remark 5.2]{BCII} there is no cyclic $8$-isogeny 
with source elliptic curve having $-16$-CM and having field of moduli $\Q$.  Thus 
\[ \Q(\iota_2 \circ \iota_1) \supsetneq \Q = \Q(\iota_1)\Q(\iota_2).\]
The work \cite{CS22b} will provide an explanation (different from the one of part a)).
\end{itemize}
\end{remark}
\noindent
If $\iota: E \ra E'$ is an isogeny of elliptic curves, then clearly any field of definition for $\iota$ must also be a field of definition for 
$E$ and for $E'$, and thus we get $\Q(\iota) \supset \Q(j(E),j(E'))$.  It turns out that in the absence of complex multiplication, this evident lower bound for the field of moduli is an equality.  

\begin{prop}
\label{PROP3.2}
\label{3.2}
Let $\iota: E \ra E'$ be an isogeny of elliptic curves defined over the complex numbers.   If $E$ (hence also $E'$) does \emph{not} 
have complex multiplication, then we have 
\[ \Q(\iota) = \Q(j(E),j(E')). \]
\end{prop}
\begin{proof}
Let $N = \deg \iota$.  
Put $F \coloneqq \Q(j(E),j(E'))$, and choose a model $E_{/F}$.  Let $C$ be the kernel of $\iota$.  We may view $C$ as an 
$\Aut E = \{ \pm 1\}$, so automorphisms of $E$ preserve $C$.  Since $E' \cong E/C$ we have $j(E/C) = j(E') \in F$. \\ \indent 
We must show that for all $\sigma \in \gg_F$, we have $\sigma(C) = C$.  Suppose that we have $\sigma(C) \neq C$ for some 
$\sigma \in \gg_F$, and consider the ``transported'' isogeny $\sigma(\iota): E^{\sigma} \ra (E/C)^{\sigma}$.  Because $E$ 
is defined over $F$, we have $E^{\sigma} = E$, and because $j(E/C) \in F$ we have that $(E/C)^{\sigma}$ is isomorphic to 
$E/C$ hence also to $E'$, so after composing $\sigma(\iota)$ with such an isomorphism we get a cyclic $N$-isogeny $\psi: E \ra E'$ 
with kernel $\sigma(C)$.  \\ \indent
Thus we are in the following situation: we have two complex elliptic curves $E$, $E'$ and two cyclic $N$-isogenies $\iota,\psi: E \ra E'$ 
with distinct kernels.  We claim that this implies that $E$ has complex multiplication.  To see this, consider 
\[ \xi \coloneqq \iota^{\vee} \circ \psi: E \ra E, \]
an endomorphism of $E$ of degree $N^2$.  Since $E$ does not have CM, we must have $\xi = \pm [N]$.  But we have 
$ \iota^{\vee} \circ \psi = [N]$ if and only if $\iota^{\vee} = \psi^{\vee}$ if and only if $\iota = \psi$.  Similarly, we have $\iota^{\vee} \circ \psi = 
-[N]$ if and only if $\iota^{\vee} \circ (-\psi) = [N]$ if and only if $\iota = - \psi$.  Either way we get $\ker \iota = \ker \psi$, a contradiction.  
\end{proof}
\noindent
Recently Cremona-Najman showed the following \cite[Cor. A.5]{Cremona-Najman21}: if $E_1$ and $E_2$ 
are elliptic curves defined and isogenous over $\overline{\Q}$, then there are models for $E_1$ and $E_2$ over $\Q(j(E_1),j(E_2))$ 
and a $\Q(j(E_1),j(E_2))$-rational isogeny $E_1 \ra E_2$.  Compared to Proposition \ref{PROP3.2} their result has a weaker hypothesis --- they allow CM --- and a weaker conclusion: they assert the \emph{existence} of an isogeny $E_1 \ra E_2$ defined over $\Q(j(E_1),j(E_2))$, whereas Proposition \ref{PROP3.2} says that \emph{every} isogeny $E_1 \ra E_2$ is defined over $\Q(j(E_1),j(E_2))$.  
The following example shows that the stronger conclusion can fail in the CM case.

\begin{example}
\label{MYFAVEX}
Let $\Delta$ be an imaginary quadratic discriminant, let $K = \Q(\sqrt{\Delta})$, let $\OO$ be the order in $K$ with discriminant 
$\Delta$, and let $p$ be a prime number such that $p \nmid \Delta$ and $p$ splits completely in the Hilbert class field of $K$ 
(the set of such primes has density $\frac{1}{2h_K}$ so is infinite).  The hypotheses ensure that there is an element $\pi \in \OO$ 
such that $p \OO = (\pi) (\overline{\pi})$ with $(\pi) \neq (\overline{\pi})$.  Let $E_{/\C}$ be any $\OO$-CM elliptic curve.  
Let $\iota$ be multiplication by $\pi$, viewed as an isogeny from $E$ to $E$.  Its kernel is $E[(\pi)]$.  Complex conjugation takes $E[(\pi)]$ to the distinct order $p$ 
subgroup scheme $E[(\overline{\pi})]$.  It follows that 
\[ \Q(\iota) = K(j(E)) \supsetneq \Q(j(E)) = \Q(j(E),j(E')). \]
Here $E$ and $E/E[(\pi)]$ are \emph{isomorphic}, so there is no contradiction to \cite[Cor. A.5]{Cremona-Najman21}.
\end{example}
\noindent
In Lemma \ref{LEMMA7.1} we will see that if $\iota: E \ra E'$ is an isogeny between $K$-CM elliptic curves with $\Delta_K < -4$, then 
$\Q(\iota)$ is $\Q(j(E),j(E'))$ or $K(j(E),j(E'))$.  In contrast, Remark \ref{REMARK3.1}a) 
exhibits a cyclic $18$-isogeny $\iota: E \ra E'$ such that $E$ has CM by the imaginary quadratic order of discriminant $-12$ 
such that $\Q(j(E),j(E')) = \Q$ but $\iota$ cannot be defined over $K$.  A much larger class of examples of isogenies $\iota: E \ra E'$ 
of $K$-CM elliptic curves with $\Delta_K \in \{-3,-4\}$ such that $\Q(\iota) \not \subseteq K(j(E),j(E'))$ will be given in \cite[Cor. 8.5]{CS22b}.


\subsection{Proper and pleasant isogenies}
All of the results of this section are recalled from \cite[\S 2.6]{BCII}: the only novelty here is the introduction of the 
adjective ``pleasant.''
\\ \\
We call an isogeny $\varphi: E \ra E'$ of $K$-CM elliptic curves a \textbf{proper isogeny} if $\End(E) \cong \End(E')$.  For every 
proper isogeny there is a unique proper $\OO$-ideal $\aa$ such that $\varphi$ is isomorphic over $\C$ to $q: E \ra E/E[\aa]$.  Conversely, 
for every proper $\OO$-ideal $\aa$ we have $\End(E/E[\aa]) = E$, so $q: E \ra E/[\aa]$ is a proper isogeny, of degree 
$|\aa| \coloneqq \# \OO/\aa$.  Moreover, for a proper isogeny $\varphi: E \ra E'$ its field of moduli is $\Q(j(E))$ if the corresponding ideal 
$\aa$ is real -- $\overline{\aa} = \aa$ -- and is $K(j(E))$ otherwise.
\\ \\
Let $\ff' \mid \ff$, and let $\OO$ (resp. $\OO'$) be the order in $K$ of conductor $\ff$ (resp. $\ff'$).  If $E_{/\C}$ is an $\OO$-CM elliptic curve, then there is an $\OO'$-CM elliptic curve $\tilde{E}_{/\C}$ and an isogeny 
\[ \iota_{\ff,\ff'}: E \ra \tilde{E} \]
that is universal for isogenies from $E$ to an $\OO'$-CM elliptic curve $E'_{/\C}$. Moreover $\iota_{\ff,\ff'}$ is cyclic of order $\frac{\ff}{\ff'}$ and the field of moduli of $\iota_{\ff,\ff'}$ is $\Q(\ff)$.  

\begin{lemma}
Let $\Delta$ be an imaginary quadratic discriminant, and let $E_{/\C}$ be a $\Delta$-CM elliptic curve such that $j(E) \in \R$.  
Let $\iota: E \ra E'$ be a proper isogeny, with associated proper $\OO$-ideal $\aa$.  The following are equivalent:  
\begin{itemize}
\item[(i)] We have $j(E') \in \R$.  
\item[(ii)] We have $[\aa] \in (\Pic \OO)[2]$.
\end{itemize}
\end{lemma}
\begin{proof}
By \cite[Lemma 3.6]{BCS17} there is a proper, real $\OO$-ideal $I$ such that $E \cong \C/I$, and then the isogeny $E \ra E'$ 
is isomorphic to $\C/I \ra \C/\aa^{-1} I$.  By Lemma \ref{LAMBDAC}, we have $j(E') \in \R$ if and only if $j(\C/\aa^{-1} I) = 
j(\C/\overline{\aa^{-1} I}) = j(\C/ \overline{\aa}^{-1} I)$.  This holds if and only if $[\overline{\aa}^{-1} I] = [\aa^{-1} I] \in \Pic \OO$ 
if and only if $[\aa] = [\overline{\aa}] \in \Pic \OO$.  Since $\aa \overline{\aa} = N(\aa) \OO$, we also have $[\overline{\aa}] = [\aa]^{-1} \in \Pic \OO$, hence these conditions hold if and only if $[\aa] \in (\Pic \OO)[2]$.   
\end{proof}
\noindent
We say an isogeny $\iota: E \ra E''$ of $K$-CM elliptic curves over $\C$ is \textbf{pleasant} if the conductor $\ff'$ of $\End(E'')$ divides the conductor $\ff$ of $\End(E)$.  Thus $\iota$ factors as $\varphi \circ \iota_{\ff,\ff'}$ for a proper isogeny 
$\varphi: \tilde{E} \ra E'$ that corresponds to a proper $\OO'$-ideal $\aa$.  It follows that 
\[ \Q(\iota) = \begin{cases} \Q(\ff) & \aa \text{ is real } \\ 
K(\ff) & \text{otherwise} \end{cases}. \]
A factor isogeny of a proper isogeny need not be proper, and similarly for pleasant isogenies.  On the other hand, for a prime 
degree isogeny $\iota: E \ra E'$, exactly one of the following holds: $\iota$ and $\iota^{\vee}$ are both proper and pleasant; 
$\iota$ is pleasant but not proper and $\iota^{\vee}$ is neither proper nor pleasant; $\iota^{\vee}$ is pleasant but not proper 
and $\iota$ is neither proper nor pleasant.  

\subsection{Reduction to the prime power case}
Let $\varphi: E \ra E'$ be a cyclic $N$-isogeny of elliptic curves defined over a subfield $F \subset \C$, so $\varphi$ is 
equivalent to $E \ra E/C$, for $C \subset E(\overline{F})$ a $\gg_F$-stable cyclic subgroup of order $N$.  If $N = \ell_1^{a_1} \cdots 
\ell_r^{a_r}$ for primes $\ell_1 < \ldots < \ell_r$ and positive integers $a_1 < \ldots < a_r$, then $C = \bigoplus_{i=1}^r C_i$, 
where $C_i$ is the unique subgroup of order $\ell_i^{a_i}$.  The uniqueness guarantees that $C_i$ is $\gg_F$-stable, 
hence $\varphi_i \coloneqq E \ra E/C_i$ is an $F$-rational cyclic $\ell_i^{a_i}$-isogeny. Conversely, given for each $1 \leq i \leq r$ 
an $F$-rational cyclic $\ell_i^{a_i}$-isogeny $\varphi_i: E \ra E_i$, then taking $C_i = \Ker \varphi_i$, the map 
$E \ra E/\langle C_1,\ldots,C_r \rangle$ is an $F$-rational cyclic $\ell_i^{a_i}$-isogeny.  (Throughout we could have replaced 
prime powers by pairwise coprime positive integers $n_1,\ldots,n_r$, but no generality is gained by doing so.)  
\\ \\
We now consider a subtly different setup: suppose that for $1 \leq i \leq r$ we are given a cyclic $\ell_i^{a_i}$-isogeny 
of complex elliptic curves $\varphi_i: E \ra E_i$, each with field of moduli contained in a subfield $F \subset \C$.  As above, let 
$C_i = \Ker \varphi_i$ and put $\varphi: E \ra E/\langle C_1,\ldots,C_r \rangle$.  Is the field of moduli of $\varphi$ contained in $F$?  It 
turns out that this is always the case when $j(E) \notin \{0,1728\}$.  Indeed, if $F$ is a subfield 
of $\C$ and $E_{/\C}$ is an elliptic curve with $j(E) \in F \setminus \{0,1728\}$, then $\Aut E = \{ \pm 1\}$, so for any two $F$-rational 
models of $E$ and any $N \in \Z^+$, the two modulo $N$ Galois representations differ by a quadratic character, and thus whether 
a finite subgroup of $E[N](\overline{F})$ is $\gg_F$-stable is independent of the chosen $F$-rational model.  This shows that 
for any $F$-rational model of $E$ and $1 \leq i \leq r$, there is a $\gg_F$-stable finite subgroup $C_i$ of $E(\overline{F})$ of order 
$\ell_i^{a_i}$, and thus $E \ra E/\langle C_1,\ldots,C_r \rangle$ is an $F$-rational cyclic $\ell_1^{a_1} \cdots \ell_r^{a_r}$-isogeny.  
\\ \\
\textbf{However:} by \cite[Thm. 6.18c)]{BCI}, there are elliptic curves $(E_1)_{/\Q(\sqrt{-3})}$ and $(E_2)_{/\Q(\sqrt{-3})}$ such that $j(E_1) = j(E_2) = 0$, the curve $E_1$ admits a $\Q(\sqrt{-3})$-rational $2$-isogeny and the curve $E_2$ admits a $\Q(\sqrt{-3})$-rational cyclic $9$-isogeny, 
but there is no elliptic curve $E_{/\Q(\sqrt{-3})}$ that admits a $\Q(\sqrt{-3})$-rational cyclic $18$-isogeny.  By \cite[Cor. 5.11b)]{BCII} 
the same holds with the ground field $\Q(\sqrt{-3})$ replaced by $\Q$ throughout the assertions of the previous sentence. 
These additional complications are one of the main reasons why the $\Delta_K \in \{-4,-3\}$ cases are deferred to \cite{CS22b}.  In contrast, the following scheme-theoretic result ensures that primary decomposition of isogenies is much more straightforward 
when $\Delta < -4$.

\begin{prop}
\label{PROP3.7}
Let $K$ be a field of characteristic $0$, and let $p \in X(1)_{/K}$ be a closed point on the $j$-line with $p \notin \{0,1728,\infty\}$.  
Let $N_1,\ldots,N_r \in \Z^+$ be pairwise coprime, and for $1 \leq i \leq r$ let $M_i \in \Z^+$ be such that $M_i \mid N_i$.  Put 
\[N \coloneqq \prod_{i=1}^r N_i, \ M \coloneqq \prod_{i=1}^r M_i. \]
For $1 \leq i \leq r$, let $\pi_i: X_0(M_i,N_i) \ra X(1)$ be the natural map of $K$-schemes, let $F_i$ be the fiber of $\pi_i$ over $p$, and let $F$ be the fiber of the $K$-morphism $\pi: X_0(M,N) \ra X(1)$ over $p$.  Then $F$ is the fiber product of $F_1,\ldots,F_r$ over $\Spec K(p)$.  
\end{prop}
\begin{proof}
It is no loss of generality to assume that $K = K(p)$, so that $p$ is a $K$-rational point on $X(1)_{/K}$.  We will do so.  
For all $N \in \Z^+$ the $K$-morphism $X(N) \ra X(1)$ is unramified at $p$, so any morphism of intermediate modular 
curves is also unramified in every fiber over $p$.  Thus $F = \Spec A$, where $A$ is a finite \'etale $K$-algebra (a finite 
product of number fields, each containing $K$) and similarly for $1 \leq i \leq r$ we have $F_i = \Spec A_i$, where $A_i$ is a finite \'etale 
$K$-algebra, and we wish to show that the natural map 
\[ \Phi: A_1 \otimes_{K} \ldots \otimes_{K} A_r \ra A\] is a $K$-algebra isomorphism.  Let $\overline{K}$ be an algebraic closure of $K$, and let $\mathfrak{g} \coloneqq \Aut(\overline{K}/K)$.  Then we get a continuous morphism of finite $\mathfrak{g}$-sets 
\[ \Phi^{\vee}: F(\overline{K}) \ra \prod_{i=1}^r F_i(\overline{K}), \]
and by the Main Theorem of ``Grothendieck's Galois Theory'' \cite[Thm. 1.5.4]{Szamuely}, it is enough to show that $\Phi^{\vee}$ 
is a bijection.  
For any subgroup $ \{ \pm 1\} \subseteq \Gamma \subseteq \GL_2(\Z/N\Z)$ we use the following description of the $X(\Gamma)(\overline{K})$: its points are equivalence classes of pairs $(E,\gamma)$ where 
$E_{/\overline{K}}$ is an elliptic curve and $\gamma: E[N](\overline{K}) \ra \Z/N\Z \oplus \Z/N\Z$ is an isomorphism of $\Z$-modules; two such pairs $(E,\gamma)$ and $(E',\gamma')$ are equivalent if there is an isomorphism $\alpha: E \ra E'$ such that the composite 
isomorphism 
\[ \Z/N\Z \oplus \Z/N\Z \stackrel{\gamma^{-1}}{\ra} E[N] \stackrel{\alpha}{\ra} E'[N] \stackrel{\gamma'}{\ra} \Z/N\Z \oplus \Z/N\Z \]
lies in the subgroup $\Gamma$ of $\GL_2(\Z/N\Z)$.
\\ \indent
We now apply this description with $\Gamma$ equal to one of the subgroups $H_0(M_i,N_i)$ (cf. \S 1.1) or to $H_0(M,N)$.  The first 
thing to observe is that we may view all of these groups as subgroups of $\GL_2(\Z/N\Z)/\{ \pm 1\}$ via pulling back using the 
homomorphism $\GL_2(\Z/N\Z)/{\pm 1} \ra \GL_2(\Z/N_i \Z)/\{ \pm 1\}$ and then we have 
\begin{equation}
\label{GROUPEQ}
 H_0(M,N) = \bigcap_{i=1}^r H_0(M_i,N_i). 
\end{equation}
So let $(P_1,\ldots,P_r) \in \prod_{i=1}^r F_i(\overline{K})$.  Thus for all $1 \leq i \leq r$, $P_i$ is represented by a pair 
$(E_i,\gamma_i)$ where $E_i$ is an elliptic curve defined over $\overline{K}$ with $j$-invariant $p$ and $\gamma_i: E_i[N_i] \stackrel{\sim}{\ra} \Z/N_i \Z \oplus \Z/N_i \Z$.  All these elliptic curves have $j$-invariant $p$, so if $E_{/\overline{K}}$ is a fixed elliptic curve
with $j$-invariant $p$, then for all $1 \leq i \leq r$, by choosing an isomorphism $\alpha_i: E \ra E_i$ we get a pair
$(E,\gamma_i \circ \alpha_i)$ that induces the same element of $F_i(\overline{K})$ as $(E_i,\gamma_i)$.  
\\ \indent
Now we have $E[N] = \bigoplus_{i=1}^r E[N_i]$ and $\Z/N\Z \oplus \Z/N\Z = \bigoplus_{i=1}^r (\Z/N_i \Z \oplus \Z/N_i \Z)$, 
so we get an isomorphism
\[ \gamma = (\gamma_1 \circ \alpha_1,\ldots,\gamma_r \circ \alpha_r): E[N] \ra \Z/N\Z \oplus \Z/N\Z. \]
Thus $(E,\gamma)$ induces a point $P$ of $F(\overline{K})$ such that $\pi_i(P) = P_i$ for all $1 \leq i \leq r$.  It follows 
from (\ref{GROUPEQ}) that if $P' \in F(\overline{K})$ is such that $\pi_i(P') = P_i$ for all $1 \leq i \leq r$, then $P' = P$.  
\end{proof}

\section{Isogeny Volcanoes}

\subsection{The isogeny graph $\mathcal{G}_{K,\ell,\ff_0}$} Let $K$ be an imaginary quadratic field, and let $\ell$ be a prime number.  
We define a directed multigraph $\mathcal{G}_{K,\ell}$ as follows: 
the vertex set $\mathcal{V}$ of $\mathcal{G}_{K,\ell}$ is the set of $j$-invariants $j \in \C$ of \textbf{K-CM} elliptic curves, i.e., $j$-invariants of complex elliptic curves with endomorphism ring an order in the imaginary quadratic field $K$.  In general, for 
$j \in \mathcal{V}$ we denote by $E_j$ a complex elliptic curve with $j$-invariant $j$.  The edges are obtained as follows: 
let $\pi_1: X_0(\ell) \ra X(1)$ be the natural map, let $w_N \in \Aut( X_0(N))$ be the Atkin-Lehner involution, and let $\pi_2: X_0(\ell) \coloneqq \pi_1 \circ w_N$.  For $j,j' \in \mathcal{V}$, write 
\[ (\pi_2)_* \pi_1^*([j]) = \sum_P e_P [P]. \] 
Then the number of directed edges from $j$ to $j'$ is $e_{j'}$.  
\\ \indent
Here are two more characterizations of the multiplicity: \\
$\bullet$ Let $\Phi_{\ell}(X,Y) \in \Z[X,Y]$ be the $\ell$th modular polynomial.  Then $e_{j'}$ 
is the multiplicity to which $j'$ occurs as a root of the univariate polynomial $\Phi_{\ell}(j,Y)$.  \\ 
$\bullet$ Let $E_{/\C}$ be any elliptic curve with $j$-invariant $j$.  Then the number of edges from $j$ to $j'$ is the number 
of cyclic order $\ell$ subgroups $C$ of $E$ such that $j(E/C) = j'$.  
\\ \indent
Each $j \in \mathcal{V}$ has outward degree $\ell+1$.  When $j \neq 0,1728$, the map $\pi_1$ is unramified over $j$, and we may identify the edges emanating from the vertex $j$ with isomorphism classes of $\ell$-isogenies with target elliptic curve $E_j$ and thus also with order $\ell$ subgroups of $E_j$.  There is at least one directed edge from $j$ to $j'$ if and only if there is an $\ell$-isogeny $\iota: E_j \ra E_{j'}$.  The dual isogeny $\iota^{\vee}: E_{j'} \ra E_j$ then shows that there is at least one directed edge from $j'$ to $j$.  When $j,j' \notin \{0,1728\}$, taking the dual isogeny gives a bijection from the set of directed edges $j \ra j'$ to the set of directed edges $j' \ra j$, so one may safely neglect orientations of such edges.  This need not be the case if $j,j' \in \{0,1728\}$. 
\\ \\
Let $\iota: E \ra E'$ be an $\ell$-isogeny of $K$-CM elliptic curves.  We put $\OO \coloneqq \End(E)$ and $\OO' \coloneqq \End(E')$.  
Thus $\OO$ and $\OO'$ are each orders in $K$.  Let $\ff$ (resp. $\ff'$) be the conductor of $\OO$ (resp. of $\OO'$).  Then by e.g. 
\cite[\S 5.5]{BCII} 
we have 
\begin{equation}
\label{LITTLELISOGEQ}
\frac{\ff}{\ff'} \in \{1,\ell,\ell^{-1} \}. 
\end{equation}

\begin{lemma}
\label{CONRADLEMMA}
Let $\OO$ be an order in an imaginary quadratic field $K$ of conductor divisible by a prime $\ell$.  Then no $\OO$-CM elliptic 
curve $E_{/\C}$ admits a proper $\ell$-isogeny.  Equivalently, there is no proper $\OO$-ideal 
of norm $\ell$.
\end{lemma}
\begin{proof}
{\sc First proof:} Let $\pp$ be an ideal of norm $\ell$, i.e., $\# \OO/\pp = \ell$.  Then $\pp$ is maximal and contains $\ell$, so 
$\langle \pp, \ff \rangle \subset \langle \pp, \ell \rangle = \pp$.  By \cite[Thm. 6.1]{Conrad}, this means that $\pp$ is not invertible, and 
thus by \cite[Thm. 3.4]{Conrad} we have that $\pp$ is not proper.  \\
{\sc Second proof:} Following \cite[Example 3.15]{Conrad} we will show there is a unique ideal of $\OO$ of norm $\ell$ and explicitly construct it.  Then we will find an explicit element of $(\pp:\pp) \setminus \OO$ and thereby show that $\pp$ is not proper. 
Write \[\OO_K = \Z \cdot 1 \oplus \Z \cdot \tau_K. \]
Put $\tau \coloneqq \ff \tau_K$, so 
\[ \OO = \Z \cdot 1 \oplus \Z \cdot \tau. \]
\\ \indent  Let
\[ \pp \coloneqq \ell \Z + \ff \OO_K. \]
Then as a $\Z$-module we have 
\[ \pp = \langle \ell, \ff, \tau \rangle_{\Z} = \Z \cdot \ell \oplus \Z \cdot \tau, \]
so $[\OO:\pp] = \ell$.  Since $\ell \pp \subset \pp$ and $\tau \pp \subset \pp$, $\pp$ is an ideal of $\OO$.  Since it has norm $\ell$, 
it is maximal.  We claim that it is the unique ideal of norm $\ell$: if $\qq$ is an ideal of $\OO$ of norm $\ell$, then $\ell = 1 \cdot [\OO:\qq] \in \qq$.  Since $\ell \mid \ff$ we have 
\[ \pp^2 = (\ell \Z + \ff \OO_K)^2 = \ell^2 \Z + \ell \ff \OO_K + \ff^2 \OO_K = \ell^2 \Z + \ell \ff \OO_K \subset \ell \OO \subset \qq. \]
Since $\pp$ is prime this gives $\pp \subset \qq$; we have an inclusion of nonzero prime ideals in a one-dimensional domain, so 
$\pp = \qq$.  \\ \indent Now we claim that $\frac{\tau}{\ell} \in (\pp:\pp)$, hence $(\pp:\pp) \supsetneq \OO$ and $\pp$ is not proper.  
Let $a,b \in \Z$ be such that $\tau_K^2 = a + b \tau_K$.  Then we have 
\[ \frac{\tau}{\ell} \ell = \tau \in \pp, \  \frac{\tau}{\ell} \tau = \frac{\ff^2}{\ell} \tau_K^2 = \frac{\ff^2}{\ell}(a + b \tau_K) = \frac{a \ff^2}{\ell} + \frac{b\ff}{\ell} \tau 
\in \Z \ell + \Z \tau = \pp. \qedhere \]
\end{proof}
\noindent
It follows from (\ref{LITTLELISOGEQ}) that if $\varphi: E \ra E'$ is a cyclic $\ell$-power isogeny of $K$-CM elliptic curves, $\ff$ is 
the conductor of $\End(E)$ and $\ff'$ is the conductor of $\End(E')$, then $\frac{\ff}{\ff'} \in \ell^{\Z}$.  Thus the prime to $\ell$ part 
of the conductor is constant on each connected component of the graph $\mathcal{G}_{K,\ell,\ff_0}$, so we may as well fix $\ff_0$, a 
positive integer prime to $\ell$ and let $\mathcal{G}_{K,\ell,\ff_0}$ be the induced subgraph consisting of all vertices with 
prime to $\ell$ conductor $\ff_0$.  
\\ \\
The \textbf{level} $L$ of a vertex $j \in \mathcal{G}_{\ell,K}$ is $\ord_{\ell}(\ff)$, where $\ff$ is the conductor of $\End(E_j)$.   
The \textbf{surface} of $\mathcal{G}_{K,\ell,\ff_0}$ is the induced subgraph consisting of all vertices at level $0$.  Lemma \ref{CONRADLEMMA} tells us that horizontal edges can only occur at the surface of $\mathcal{G}_{K,\ell,\ff_0}$.

\subsection{Volcanoes} Fix a prime number $\ell$.  An $\ell$-\textbf{volcano} $V$ is a directed graph -- possibly with loops and/or multiple edges -- with a partitioning of its vertex set into levels $\{V_i\}_{i \in I}$ and a partitioning of its edge set into three 
subsets, called \textbf{horizontal}, \textbf{ascending} and \textbf{descending}.   Here $I$ is a nonempty downward closed subset of the natural numbers $\N$, and thus it is either $\N$ or $\{0,1,\ldots,d\}$ for some $d \in \N$.  The \textbf{depth} of $V$ is the largest $d$ such that $V_d$ is nonempty if such a $d$ exists and $\infty$ otherwise.  If $V$ has finite 
depth $d$, we we call the level $V_d$ the \textbf{floor}.  We require the following additional properties: \\
(0) For every edge $e: v \mapsto w$ 
there is a canonical inverse edge $\overline{e}: w \mapsto v$.\footnote{Thus by identifying $e$ and $\overline{e}$ we get an undirected graph, from which $V$ can be recovered, so this is an equivalent perspective.  When we speak of the degree of a vertex, we mean 
of the underlying undirected graph, and we take the convention that an edge from $x$ to $x$ contributes one to the degree.} \\
(i) The subgraph $V_0$, called the \textbf{surface} is a regular graph of degree $s(V) \in \{0,1,2\}$.  An edge is horizontal if and only if it 
connects two vertices in $V_0$.  \\
(ii) The ascending edges are precisely as follows: for all positive $i \in I$ and each vertex $x$ in $V_i$ there is a unique $y \in V_{i-1}$ 
and an ascending edge $e: x \mapsto v$.  The descending edges are precisely the inverses of the ascending edges.  \\
(iii) Every vertex in the floor (if any) has degree $1$.\footnote{This follows from the above 
properties, but is worth stating explicitly.}  Every other vertex $x$ has degree $\ell+1$.  Thus if $x$ has level $i \geq 1$ then it is connected to one vertex in level $i-1$ and $\ell$ distinct vertices in level $i+1$.  If $x$ has level $0$ then it has $s(V)$ horizontal undirected edges and $\ell+1-s(V)$ descending edges emanating from $x$ to distinct vertices in $V_1$.  
\\ \\
For a fixed surface vertex $v_0$, the induced subgraph on the set of vertices that can by repeatedly descending 
from $v_0$ has the structure of a rooted tree with $v_0$ as the root.
\\ \\
For the rest of this paper we will only consider the isogeny graph $\mathcal{G}_{K,\ell,\ff_0}$ when $\ff_0^2 \Delta_K < -4$, and we claim that in these cases $\mathcal{G}_{K,\ell,\ff_0}$ is an isogeny volcano of infinite depth for which each vertex at the surface has degree $1 + \left( \frac{\Delta_K}{\ell} \right)$.\footnote{The converse is also true: if $\ff_0^2 \Delta_K \in \{-3,-4\}$ then $\mathcal{G}_{K,\ell,\ff_0}$ is \emph{not} an $\ell$-volcano.  The structure of the graph is still known, and it will be recalled in \cite{CS22b}.} In this regard, we have already shown (0).  
As for (i), the surface edges emanating outward from a surface vertex correspond to proper $\OO$-ideals of norm $\ell$.  Since $\ell \nmid \ff_0$, 
every $\OO$-ideal of norm $\ell$ is proper \cite[Thm. 3.1]{Conrad} and the pushforward map $\aa \mapsto \aa \Z_K$ is a norm-preserving 
bijection from $\OO$-ideals of norm $\ell$ to $\Z_K$-ideals of norm $\ell$ \cite[Thm. 3.8]{Conrad}.  Thus we reduce to the $\ff_0 = 1$ 
case, in which we certainly have $2$, $1$ or $0$ ideals of norm $\ell$ according to whether $\ell$ is split, ramified or inert in $K$. Let us look more carefully at the three cases:
\\ \\
\textbf{Inert Case:} If $\left( \frac{\Delta_K}{\ell} \right) = -1$, then in the order $\OO$ of conductor $\ff_0$ we have no ideals 
of order $\ell$ hence no surface vertices.
\\ \\
\textbf{Ramified Case:} If $\left( \frac{\Delta_K}{\ell} \right) = 0$, then in the order $\OO$ of conductor $\ff_0$ we have a 
unique prime ideal $\pp$ that is proper and of norm $\ell$.  If $\pp$ is principal then every surface vertex has a unique, self-inverse 
loop.  Otherwise, since $\pp^2 = (\ell)$ we have that $[\pp]$ has order $2$ in $\Pic \OO$, and the set of surface vertices is partitioned into pairs $v_1,v_2$, such that if $v_1$ corresponds to $E$ then $v_2$ corresponds to $E/E[\pp]$, and each of these surface edges is self-inverse.  Because the ideal $\pp$ is real, the field of moduli of $E \ra E/E[\pp]$ is $\Q(\ff_0)$.  
\\ \\
\textbf{Split Case:} If $\left( \frac{\Delta_K}{\ell} \right) = 1$, then in the order $\OO$ of conductor $\ff_0$ we have $(\ell) = \pp \overline{\pp}$ 
for distinct prime ideals $\pp$, $\overline{\pp}$ of norm $\ell$.  Let $r$ be the order of $[\pp]$ in $\Pic \OO$.  If $r = 1$ then every 
surface vertex has two distinct, mutually inverse loops corresponding to $\pp$ and $\overline{\pp}$.  If $r = 2$ -- equivalently, if 
$[\pp] = [\overline{\pp}]$ then set of surface vertices is partitioned into pairs $v_1,v_2$ such that if $v_1$ corresponds to $E$ 
then $v_2$ corresponds to $E/E[\pp] \cong E/E[\overline{\pp}] \cong E'$, say, and there are two edges $\pp$, $\overline{\pp}$ running from $v_1$ to $v_2$.  The inverses of these edges correspond to the isogenies $E' \ra E'/[\overline{\pp}]$ and $E' \ra E'/E'[\pp]$ respectively.  If $r \geq 3$ then the set of surface vertices is naturally partitioned into $r$-cycles.  Finally, as a special case of 
the results on proper isogenies of the previous section, because the ideal $\pp$ is not real, the field of moduli of $E \ra E/E[\pp]$ is $K(\ff_0)$.
\\ \\
Now let $L \geq 1$, and let $v \in V(\mathcal{G}_{K,\ell,\ff_0})$ be a vertex at level $L$, let $\ff = \ell^L \ff_0$, $\Delta = \ff^2 \Delta_K$, and let $E$ be the corresponding $\Delta$-CM elliptic curve.  From \S 3.3 we get that the unique (up to equivalence) $\ell$-isogeny 
from $E$ to an elliptic curve with conductor $\frac{\ff}{\ell}$ is $\iota_{\ff,\frac{\ff}{\ell}}$, which has field of moduli $\Q(j(E))$.  This shows in particular the exsistence and uniqueness of ascending edges emanating from a non-surface vertex, and clearly the descending 
edges are the inverses of the ascending edges.  This shows (ii).
\\ \\
In our setup we have no floor, and we saw in \S 4.1 that every vertex in $\mathcal{G}_{K,\ell,\ff_0}$ has outward degree $\ell+1$.  
The rest of (iii) amounts to the claim that the only multiple edges and loops in $\mathcal{G}_{K,\ell,\ff_0}$ lie in the surface.  That 
there are no loops below the surface follows from Lemma \ref{CONRADLEMMA}.  Consider first a surface vertex $v_0$.  
By Corollary \ref{LCLASSCOR}c), the number of vertices at level $1$ is equal to $\ell - \left( \frac{\Delta_K}{\ell} \right)$ times the 
number of vertices on the surface, which is also equal to the total number of edges descending from the surface.  So if two distinct downard edges emanating from $v_0$ had the same terminal vertex $v_1$, then some other vertex on level $1$ would not be connected to any surface vertex, contrary to what we already know.  Now let $L \geq 1$ and consider a vertex $v_L$ at level $L$.  Then 
Corollary \ref{LCLASSCOR}c) shows that the number of vertices at level $L+1$ is $\ell$ times the number of vertices at level $\ell$, which is 
also equal to the total number of edges descending from level $L$, so again no two edges emanating from $v_L$ can have the same 
terminal vertex.  This completes our proof that when $\ff_0^2 \Delta_K < -4$ the isogeny graph $\mathcal{G}_{K,\ell,\ff_0}$ is an $\ell$-volcano.

\subsection{Paths and $\ell^a$-isogenies}

A \textbf{path} in a directed graph consists of a finite sequence of directed edges $e_1,\ldots,e_N$ such that for all $1 \leq i \leq N-1$ the terminal vertex of $e_i$ is the initial vertex of $e_{i+1}$.  In a directed graph in which each edge has a canonical inverse edge, a path has \textbf{backtracking} if for some $i$ we have that $e_{i+1}$ is the inverse edge of $e_i$.  This is the case for $\mathcal{G}_{K,\ell,\ff_0}$ 
in the case $\ff_0^2 \Delta_K < -4$ that we consider in this paper.
\\ \\
The notion of a backtracking path becomes a bit more subtle in the presence of loops or multiple edges.  In the 
ramified case, for a path that includes surface edges $v_0 \stackrel{e_1}{\ra} w_0 \stackrel{e_2}{\ra} v_0$, we have 
$e_2 = e_1^{-1}$ and thus have backtracking (whether $v_0 = w_0$ or not).  In the split case, a backtracking involving surface edges comes from traversing an edge corresponding to a prime ideal $\pp$ followed by an edge corresponding to its conjugate ideal $\overline{\pp}$.  Thus in the split case, for a nonbacktracking path ending with a surface edge, exactly one of the two edges can be appended to retain a nonbacktracking path.

\begin{lemma}
\label{LEMMA4.4}
Let $\Delta = \ff_0^2 \Delta_K < -4$ be an imaginary quadratic discriminant, and let $(E_0)_{/\C}$ be a $\Delta$-CM elliptic curve.  
There is a bijective correspondence from the set of cyclic $\ell^a$ isogenies $\varphi: E_0 \ra E_a$ modulo isomorphism on the target 
and length $a$ nonbacktracking paths in the isogeny graph $\mathcal{G}_{K,\ell,\ff_0}$ with initial vertex $j(E_0)$. 
\end{lemma}
\begin{proof}
Let $\varphi: E_0 \ra E_a$ be a cyclic $\ell^a$-isogeny of elliptic curves, each with $K$-CM and prime-to-$\ell$-conductor $\ff_0$.  As in \S 3.2, the isogeny $\varphi$ factors uniquely into a length $a$ sequence of  $\ell$-isogenies $\varphi_i: E_i \ra E_{i+1}$, and in this way we get a path of length $a$ in $\mathcal{G}_{K,\ell,\ff_0}$. A backtracking in $\mathcal{G}_{K,\ell,\ff_0}$ corresponds to performing an $\ell$-isogeny $\iota$ followed by its 
dual isogeny $\iota^{\vee}$ and thus $\varphi$ would factor through $[\ell]$ and fail to be cyclic.   \\ \indent
Conversely, a path of length $a$ in $\mathcal{G}_{K,\ell,\ff_0}$ yields a sequence of cyclic $\ell$-isogenies $\{\varphi_i: E_i \ra E_{i+1}\}_{i=0}^{a-1}$ and then $\varphi \coloneqq \varphi_{a-1} \circ \ldots \varphi_0$ is an $\ell^a$-isogeny.  It remains to see 
that the lack of backtracking implies that $\varphi$ is cyclic.  This comes down to: for $1 \leq A \leq a$, if $\varphi = \varphi_{A-1} \circ \ldots \circ \varphi_0: E_0 \ra E_{A}$ is a cyclic $k$-isogeny and $\psi: E_{A} \ra E'$ is an $\ell$-isogeny, then if 
$\Ker \psi \neq \Ker \varphi_{A-1}^{\vee}$ then $\psi \circ \varphi$ is a cyclic $\ell^{A+1}$-isogeny.  Via uniformizing lattices this 
translates to the following claim: let 
\[ \Lambda_0 \subset \Lambda_1 \subset \ldots \subset \Lambda_A \subset \C \]
be lattices with $[\Lambda_{i+1}:\Lambda_i] = \ell$ for all $i$ and $\Lambda_A/\Lambda_0$ cyclic of order $\ell^{A}$.  Then of the $\ell+1$ lattices $\Lambda' \subset \C$ that contain $\Lambda$ with index $\ell$, there is exactly one such that 
$\Lambda'/\Lambda_0$ is not cyclic, namely $\frac{1}{\ell} \Lambda_{A-1}$.  Indeed, by the structure theory of finitely generated 
$\Z$-modules there is a $\Z$-basis $e_1,e_2$ for $\Lambda_A$ such that $e_1, \ell^A e_2$ is a $\Z$-basis for $\Lambda_0$ and 
thus $e_1,\ell e_2$ is a $\Z$-basis for $\Lambda_{A-1}$.  The $\ell$-torsion subgroup of $(\Lambda_0 \otimes \Q)/\Lambda_0$ is 
generated by $\frac{1}{\ell} e_1$ and $\ell^{A-1} e_2$, so a subgroup of $(\Lambda_0 \otimes \Q)/\Lambda_0$ that contains 
$\Lambda_1/\Lambda_0$ has full $\ell$-torsion if and only if it contains $\frac{1}{\ell} e_1$.  But $\frac{1}{\ell} \Lambda_{A-1}$ is the unique 
lattice containing $\Lambda_A$ with index $\ell$ and containing $\frac{1}{\ell} e_1$.
\end{proof}
\noindent
Lemma \ref{LEMMA4.4} allows us to enumerate and then study cyclic prime power isogenies of CM elliptic curves in terms of nonbacktracking 
finite paths in $\mathcal{G}_{K,\ell,\ff_0}$.  The key to this enumeration is that nonbacktracking paths in $\mathcal{G}_{K,\ell,\ff_0}$ have a restricted form: every such path $P$ uniquely decomposes as a concatenation $P_3 \circ P_2 \circ P_1$ of three paths, though some of them may have length $0$.  Namely, $P_1$ consists entirely of ascending edges, $P_2$ consists entirely of horizontal edges and $P_3$ 
consists entirely of descending edges.  To see this we observe that an equivalent statement is that a surface edge can only be followed by another surface edge or 
descending edge, while a descending edge can only be followed by a descending edge.  The former statement is clear -- as we are 
at the surface, we cannot ascend -- and once we have descended, we are not at the surface so have no horizontal edges and a descent 
followed by an ascent is a backtrack.

\subsection{Reality, part II: coreality} In this section we will determine the isomorphism class of the number field generated by the $j$-invariants of two elliptic curves 
with CM by the same imaginary quadratic field $K$, assuming that $\Delta_K < -4$.  
\\ \\
Let $j,j' \in \C$ be $K$-CM $j$-invariants, of discriminants $\Delta,\Delta'$ and conductors $\ff,\ff'$. 
\\ \\
We say $j,j'$ are \textbf{coreal} if the number field $\Q(j,j')$ admits a real embedding.  
\\ \\
Suppose first that there is a prime number $\ell$ and integers $a \geq a' \geq 1$ such that $\ff = \ell^a$, $\ff' = \ell^{a'}$.  In this case 
we claim that $j,j'$ are coreal if for any field embedding $\iota: \Q(j,j') \hookrightarrow \C$, we have $\iota(j) \in \R \implies \iota(j') \in \R$.  
The latter condition is clearly sufficient for coreality: indeed, there is an embedding $\iota: \Q(j) \hookrightarrow \R \subset \C$, so 
our assumption gives \[\iota(\Q(j,j')) = \Q(\iota(j),\iota(j')) \subset \R. \]  Conversely, suppose that $\Q(j,j')$ admits a 
real embedding, and let $\iota: \Q(j,j) \hookrightarrow \C$ be such that $\iota(j) \in \R$.  Then $\iota(j) \in K(j_{\Delta})^c = 
\Q(\ff)$.  Since $a' \leq a$ we have $\Q(\iota(j')) \subset K(\ff)$ and $\Q(\iota(j),\iota(j')) = K(\ff)$ if and only if $\iota(j') \notin \R$.  
Since $j,j'$ are coreal we have $\Q(\iota(j),\iota(j')) \cong \Q(j,j')$ admits a real embedding, so it cannot contain $K$ and thus $\iota(j') \in \R$.  
\\ \indent
In the above setup, to determine the coreality of $j,j'$ we may reduce via simultaneous Galois conjugacy to the case $j = j_{\Delta}$, and then the set of 
$\Delta'$-CM $j$-invariants $j'$ such that $j,j'$ are coreal are precisely the real roots of the Hilbert class polynomial $H_{\Delta'}(t)$, 
of which there are precisely $h_2(\Delta') \coloneqq \# \Pic \OO(\Delta')[2]$.  If $j,j'$ are coreal we have $\Q(j,j') \cong \Q(\ff)$, 
while if $j,j'$ are not coreal, we have $\Q(j,j') = K(\ff)$.  
\\ \\
Everything done above goes through verbatim in the somewhat more general case that $\ff' \mid \ff$: namely $j,j'$ are coreal 
if and only if for all $\iota: \Q(j,j') \hookrightarrow \C$ we have $\iota(j) \in \R \implies \iota(j') \in \R$, so we may Galois conjguate to the 
case $j = j_{\Delta}$, and then $\Q(j,j') \cong \Q(\ff)$ if $j,j'$ are coreal and $\Q(j,j') = K(\ff)$ otherwise.  This includes the case 
in which $\ff' = 1$ and in particular the case in which $\ff = \ff' = 1$.  
\\ \\
We remark that if $\ff' \nmid \ff$ it can happen that $j,j'$ are coreal, $j \in \R$ and $j' \notin \R$.  For instance, suppose that 
$K = \Q(\sqrt{-7})$ and $\ff = 1$, so $j \in \Q$.  Then $\Q(j,j') = \Q(j') \cong \Q(\ff')$ has a real embedding, so $j,j'$ are coreal, 
but for all sufficiently large $\ff'$ there are non-real $(\ff')^2\Delta_K$-CM $j$-invariants.
\\ \\
We now return to the general case: we have
$j = j(E)$, $j' = j(E')$, where $E$ (resp. $E'$) is a $\Delta$-CM elliptic curve (resp. $\Delta'$-CM elliptic curve) for $\Delta = \ff^2 \Delta_K$ (resp. $\Delta' = (\ff')^2 \Delta_K$).  We put $M \coloneqq \lcm(\ff,\ff')$.  
\\ \indent
There is a canonical $\Q(j)$-rational isogeny from $E$ to a $\Delta_K$-CM elliptic curve 
$E_0$ and we put $j_0 \coloneqq j(E_0)$; similarly we define $j_0' = j(E_0')$.   If $j_0,j_0'$ are coreal then $\Q(j_0,j_0') = \Q(1)$, 
while if $j_0,j_0'$ are not coreal then $\Q(j_0,j_0') = K(1)$, the Hilbert class field of $K$.  Since $\Q(j_0,j_0') \subset \Q(j,j')$, 
if $j_0,j_0'$ are not coreal then $\Q(j,j')$ contains $K$ and then it is easy to see that $\Q(j,j') = K(M)$.  \\ \indent 
From now on we assume that $j_0,j_0'$ are coreal and thus $\Q(j_0) = \Q(j_0') = F_0$, say.   We treat this case by a primary decomposition argument.    Write 
\[ \ff = \ell_1^{a_1} \cdots \ell_r^{a_r}, \ \ff' = \ell_1^{a_1'} \cdots \ell_r^{a_r'}, a_i,a_i \in \N, A_i \coloneqq \max(a_i,a_i') \in \Z^+. \]
For all $1 \leq i \leq r$, there is a canonical $\Q(j)$-rational isogeny from $E$ to a $(\Delta_i = \ell_i^{2a_i} \Delta_K)$-CM elliptic curve 
$E_i$, and we put $j_i = j(E_i)$, and in a similar way we define $E_i'$ and $j_i' = j(E_i')$.  If we Galois 
conjugate $j$ to $j_{\Delta}$ then each $j_i$ gets Galois conjugated to $j_{\Delta_i}$, so it follows from Proposition \ref{TEDIOUSALGEBRAPROP2} that $\Q(j) = \Q(j_1,\ldots,j_r)$, and similarly we have $\Q(j') = \Q(j_1',\ldots,j_r')$.  Put 
$M = \lcm(\ff,\ff')$.  If for some $1 \leq i \leq r$ we have that $j_i,j_i'$ are \emph{not} coreal, then 
\[ \Q(j,j') = K(M) = K(\ell_1^{A_1} \cdots \ell_r^{A_r}). \]
Otherwise we have that $j_i,j_i'$ are coreal for all $1 \leq i \leq r$, so for each $1 \leq i \leq r$, we have $F_i \coloneqq \Q(j_i,j_i') \supset 
\Q(j_0)$.  It follows from Proposition \ref{TEDIOUSALGEBRAPROP2} and an easy inductive argument that 
\[ \Q(j,j') = F_1 \cdots F_r \cong F_1 \otimes_{F_0} F_2 \otimes_{F_0} \cdots \otimes_{F_0} F_r \cong \Q(M). \]
In particular, we get:

\begin{thm}
\label{COREALTHM}
\label{3.9}
Let $K$ be an imaginary quadratic field with $\Delta_K < -4$, and let $j,j'$ be $K$-CM $j$-invariants, of conductors $\ff,\ff''$, and put $M \coloneqq \lcm(\ff,\ff')$.  
\begin{itemize}
\item[a)] If $j,j'$ are coreal, then $\Q(j,j') \cong \Q(M)$.
\item[b)] If $j,j'$ are not coreal, then $\Q(j,j') = K(M)$.
\end{itemize}
\end{thm}
\begin{proof}
We saw above that $\Q(j,j')$ is either isomorphic to $\Q(M)$ or equal to $K(M)$.  Since $\Q(M)$ has a real embedding and 
$K(M)$ does not, the result follows.
\end{proof}
\noindent
The following is an immediate consequence.

\begin{cor}
\label{COREALCOR}
\label{3.10}
Let $K$ be an imaginary quadratic field with $\Delta_K < -4$.  
\begin{itemize}
\item[a)] The compositum of finitely many ring class fields of $K$ is a ring class field of $K$.
\item[b)] The compositum of finitely many rational ring class fields of $K$ is either a rational ring class field of $K$ or a ring class 
field of $K$.
\end{itemize}
\end{cor}

\section{Action of Complex Conjugation on $\mathcal{G}_{K,\ell,\ff_0}$} 
\noindent
Fix $K$ and $\ff_0$. Let $\ell$ be a prime number.   In this section and the next we will define and then explicitly determine an 
action of complex conjugation -- that is, of the group $\gg_{\R} = \{1,c\}$ -- on $\mathcal{G}_{K,\ell,\ff_0}$ by graph automorphisms.\footnote{There are in fact two cases --- $(\ff_0^2 \Delta,\ell)  \in \{(-4,2),(-3,3)\}$ --- in which we \emph{cannot} define this action on the graph $\mathcal{G}_{K,\ell,\ff_0}$ but only on a certain double cover.}
This action is one of the key points of the entire work: it gives us the leverage we need to analyze isogenies of CM elliptic curves over 
$\Q(j)$ and not just over $K(j)$.  
\\\ \\
We begin with the observation that in all cases there is a natural transitive action of $\Aut(\C)$ on the vertex set of $\mathcal{G}_{K,\ell,\ff_0}$: it is indeed just the action of $\Aut(\C)$ on $\C$ restricted to the subset of elliptic curves with $\ell^{2L} \ff_0^2 \Delta_K$-CM.  This action factors through an action of $\gg_{\Q} = \Aut(\overline{\Q})$.  We pause to observe that this action preserves the level and for all $L \in \Z^{\geq 0}$ acts transitively on the set of vertices at level $L$.
\\ \indent
If vertices $v,w \in \mathcal{G}_{K,\ell,\ff_0}$ correspond to elliptic curves $E_v$ and $E_w$ over $\C$, then there is an edge from $v$ to $w$ if and only if there is a cyclic order 
$\ell$-subgroup $C$ of $E_v$ such that $E_v/C \cong E_w$.  Then for any $\sigma \in \Aut(\C)$ we have $\sigma(E_v)/\sigma(C) \cong 
\sigma(E_w)$, showing there is an edge from $\sigma(v)$ to $\sigma(w)$.  More precsiely $\sigma$ gives a bijection from the subset 
of order $\ell$ subgroups $C$ of $E_v$ such that $E_v/C \cong E_w$ to the subset of order $\ell$ subgroups $C$ of $E_{\sigma(v)}$ 
such that $E_{\sigma(v)}/C \cong E_{\sigma(w)}$.  
\\ \\
What we have done so far defines an action of $\Aut \C$ on $\mathcal{G}_{K,\ell,\ff_0}$ precisely in the absence of multiple directed 
edges running between the same pair of vertices.  In these cases we need to say more, because the isogeny graph takes into account 
only the \emph{number} of edges from $v$ to $w$: we have a bijection from the set of edges from $v$ to $w$ to the set of 
order $\ell$ subgroups $C$ of $E_v$ such that $E_v/C \cong E_w$ but not (yet) a \emph{canonical} bijection.  
\\  \\
Under our assumption that $\ff_0^2 \Delta_K < -4$, the only possible multiple edges are multiple surface edges.  Such edges exist precisely when we are in the split case --- $\left( \frac{\Delta_K}{\ell} \right) = 1$ --- and the two prime ideals 
$\pp$ and $\overline{\pp}$ of the order of $\OO$ of conductor $\ff_0$ lying over $\ell$ are principal, in which case we have two surface loops at each surface vertex $v$.  In this case, for $\sigma \in \Aut(\C)$ we decree that $\sigma$ fixes each of the two surface loops if and only if it acts trivially on $K$.  \\ \indent
To justify our definition, let $F$ be a field with $\Q(j(E_v)) \subseteq F \subseteq \C$ and let $E_{/F}$ be \emph{any} 
elliptic curve with $j$-invariant $j(E_v)$.    
The two surface loops are realized by the subgroup schemes $E[\pp]$ and $E[\overline{\pp}]$ (in some order!), and for any 
ideal $I$ of $\OO$ and $\sigma \in \gg_{F}$, we have
\[ \sigma: E[I] \mapsto E[\sigma(I)]. \]

\begin{remark}
\label{GEOPOINTREMARK}
Let $v \in \mathcal{G}_{K,\ell,\ff_0}$, and let $j_v$ be the corresponding point on $X(1)(\overline{\Q})$.  Since $\ff_0^2 \Delta_K < 4$, the  set of edges $e$ with initial vertex $v$ is in $\gg_{\Q}$-equivariant bijection with the fiber of $X_0(\ell) \ra X(1)$ over the geometric 
point $j_v$.  
\end{remark}
\noindent
Now let $c$ be the image of complex conjugation in $\gg_{\Q}$.  We call a vertex or an edge of $\mathcal{G}_{K,\ell,\ff_0}$
\textbf{real} if it is fixed by complex conjuation and \textbf{complex} otherwise.  Complex vertices and edges occur in conjugate pairs.   
\\ \\
We begin with some simple but useful observations concerning this definition:
\\ \\
$\bullet$ Complex conjugation maps ascending edges to ascending edges, horizontal edges to horizontal edges, and descending 
edges to descending edges.
\\ \\
$\bullet$ An edge $e \in \mathcal{G}_{K,\ell,\ff_0}$ determines a point $P_e \in X_0(\ell)(\C)$.  Like any curve defined over $\Q$, $X_0(\ell)$ has a canonical $\R$-model, which determines an action of complex conjugation on $X_0(\ell)(\C)$.  Under this action we have 
$c(P_e) = P_{c(e)}$, a special case of Remark \ref{GEOPOINTREMARK}. In particular, $e$ is real if and only if $P_e \in X_0(\ell)(\R)$.  
\\ \\
$\bullet$ An edge is real if and only if its inverse edge is real.  
\\ \\
$\bullet$ If an edge $e: v \mapsto w$ is real, then both $v$ and $w$ are real.  The converse also holds except for surface 
edges in the split case: as explained above, these edges are complex, but they exist for every surface vertex.   In the split case, 
let $v$ be a real surface vertex, and let $E_v$ be the corresponding complex elliptic curve.  Let $\pp$ and $\overline{\pp}$ be the 
two primes of $\OO$ lying over $\ell$.  Then $c(E/E[\pp]) = c(E)/c(E[\pp]) = E/E[\overline{\pp}]$.  We have 
$E/E[\overline{\pp}] \cong_{\C} E/[\pp]$ if and only if $[\pp] = [\overline{\pp}]$ in $\Pic \OO$.  Since $\pp \overline{\pp} = (\ell)$ 
this holds if and only if $\pp \in (\Pic \OO)[2]$.  
\\ \\
$\bullet$ For vertices $v$ and $w$ in $\mathcal{G}_{K,\ell,\ff_0}$, there is a unique edge $e$ from $v$ to $w$ if and only if there is a unique 
edge $c(e)$ from $c(v)$ to $c(w)$.  When this occurs, knowing $c(v)$ and $c(w)$ determines $c(e)$.  In particular, in this case 
the converse of the above observation holds: $e: v \mapsto w$ is real if and only if $v$ and $w$ are real.  
\\ \\
$\bullet$ In the ramified case, a surface edge 
$e: v \mapsto w$ is real if and only if $v$ is real if and only if $w$ is real. 
\\ \\
An ascending edge $e: v \mapsto w$ is real if and only $v$ is real: clearly if 
$e$ is real, then so is $v$, and conversely, if $v$ is real, then $c(e)$ is an ascending edge emanating from $v$, of which $e$ is the only one.  By passing to inverses, we deduce that a descending edge $e: v \mapsto w$ is real if and only if $w$ is real.

\subsection{The field of moduli of a cyclic $\ell^a$-isogeny} The following result computes the 
field of moduli of a cyclic $\ell^a$-isogeny of CM elliptic curves in the $\ff_0^2 \Delta_K < -4$ case.  

\begin{thm}
\label{NICEFMTHM}
\label{5.1}
Let $\varphi: E \ra E'$ be a cyclic $\ell^a$-isogeny of CM elliptic curves with $\ff_0^2 \Delta_K < -4$.  Let $\ff$ be the maximum 
of the conductor of $\End(E)$ and the conductor of $\End(E')$.  
\begin{itemize}
\item[a)] If $\ell$ splits in $K$ and $\varphi$ factors through an $\ell$-isogeny of $\ff_0 \Delta_K$-CM elliptic curves, then 
$\Q(\varphi) = K(\ff)$.  In every other case we have $\Q(\varphi) = \Q(j(E),j(E'))$.  
\item[b)] If $j(E)$ and $j(E')$ are not coreal then we have $\Q(\varphi) = K(\ff)$.  
\item[c)] Suppose that $j(E)$ and $j(E')$ are coreal.  Then if the conductor of $E'$ divides the conductor of $E$ we have 
$\Q(j(E),j(E')) = \Q(j(E))$, while if the conductor of $E$ divides the conductor of $E'$ we have $\Q(j(E),j(E')) = \Q(j(E'))$.
\end{itemize}
\end{thm}
\begin{proof}
Let $\Delta = \ell^{2L} \ff_0^2 \Delta_K$ (resp. $\Delta' = \ell^{2L'} \ff_0^2 \Delta_K$) be the discriminant of the endomorphism ring 
of $E$ (resp. of $E'$).  We have $\Q(\varphi) = \Q(\varphi^{\vee})$, and 
the assertions of Theorem \ref{NICEFMTHM} hold for $\varphi$ if and only if they hold for $\varphi^{\vee}$, so by replacing $\varphi$ with 
$\varphi^{\vee}$ is necessary we may assume that $L \geq L'$, so that the conductor $\ell^{L'} \ff_0$ of $E'$ divides the conductor 
$\ell^{L} \ff_0$ of $E$.  \\ \indent For $\sigma \in \gg_{\Q}$, we have $\Q(\sigma(\varphi)) = \sigma(\Q(\varphi)) \cong \Q(\varphi)$, so up to replacing the field of moduli by an isomorphic number field we may replace $\iota$ by $\sigma(\iota): \sigma(E) \ra \sigma(E')$ and thus we may assume that $j(E) = j_{\Delta}$. 
\\ \indent
As in \S 4.3, $\varphi$ determines a nonbacktracking path $P(\varphi)$ of length $a$ in $\mathcal{G}_{K,\ell,\ff_0}$.  Now put $E_0 \coloneqq E$, $E_a \coloneqq E'$; for $0 \leq i \leq a$, let $\varphi_i: E_i \ra E_{i+1}$ be the $\ell$-isogeny corresponding to the $i$th edge of the path $P$.  By Proposition \ref{NEWISOGENYPROP}b) we have $\Q(\varphi) = \Q(\varphi_1) \cdots \Q(\varphi_a)$.  Then $K(\ff) \supseteq \Q(j(E),j(E'))$.  Each ascending $\ell$-isogeny $\varphi_i: E_i \ra E_{i+1}$ is defined over 
$\Q(j(E_i)) \subseteq K(\ff)$; each horizontal edge is defined over $K(\ff_0) \subseteq K(\ff)$; and each descending 
$\ell$-isogeny $\varphi_i: E_i \ra E_{i+1}$ is defined over $\Q(j(E_{i+1}) \subseteq K(\ff)$.  So we have 
\[ \Q(\ff) = \Q(j(E)) \subseteq \Q(j(E),j(E')) \subseteq \Q(\varphi) \subseteq K(\ff). \]
If $\ell$ splits in $K$ and for some $i$ the $\ell$-isogeny $\varphi_i: E_i \ra E_{i+1}$ induces a horizontal edge, then $K \subset \Q(\varphi_i) 
\subset \Q(\varphi)$, and thus we must have $\Q(\varphi) = K(\ff)$.  Next suppose that $P(\varphi)$ contains no such edge.  Then 
$\Q(\varphi) = \Q(j(E_0),\ldots,j(E_a))$.  \\ \indent 
Let $v$ and $w$ are two vertices in $\mathcal{G}_{K,\ell,\ff_0}$ corresponding 
to elliptic curves $E_v$ and $E_w$.  If there is a path from $v$ to $w$ consisting entirely of ascending edges, then 
$\Q(j(E_v)) \supset \Q(j(E_w))$.  In the case that $\ell$ ramifies in $K$, if $e: v \ra w$ is a horizontal edge, then $\Q(j(E_v)) = \Q(j(E_w))$.  From this we deduce:
\[ \Q(\varphi) = \Q(j(E_0),\ldots,j(E_a)) = \Q(j(E_0),j(E_a)) = \Q(j(E),j(E')). \]
This establishes part a).  Part b) is immediate from Theorem \ref{COREALTHM}.  As for part c), we have reduced to the case $L' \leq L$, 
so if $j(E)$ and $j(E')$ are coreal then after our Galois conjugation we have $j(E') \in K(\ff) \cap \R = \Q(\ff) = \Q(j(E))$, so $\Q(j(E),j(E')) = \Q(j(E))$.  But this identity is unchanged by replacing $j(E)$ and $j(E')$ by $\sigma(j(E))$ and $\sigma(j(E'))$ for any $\sigma \in \gg_{\Q}$, so indeed we have $\Q(j(E),j(E')) = \Q(j(E))$.
\end{proof}
\noindent
Our next major task is to fix an imaginary quadratic discriminant $\Delta = \ell^{2L} \ff_0^2 \Delta_K$ with $\ff_0^2 \Delta_K < -4$ and a prime power $\ell^a$ and to compute the fiber of $X_0(\ell^a) \ra X(1)$ over $J_{\Delta}$.  In order to do this, as above we may consider 
cylic $\ell^a$-isogenies $\varphi: E \ra E'$ such that $j(E) = j_{\Delta}$, and as we range over all length $a$ nonbacktracking paths in 
$\mathcal{G}_{K,\ell,\ff_0}$ with terminal vertex $w$ corresponding to an elliptic curve $E'$, we need to understand for which of these 
paths we have that $j_{\Delta}$ and $j(E')$ are coreal.  For this we need a more explicit description of the action of $\gg_{\R}$ 
on $\mathcal{G}_{K,\ell,\ff_0}$, which we provide in the next section.  We also need to modify the above approach slightly, since 
switching to the dual isogeny so as to ensure that $j(E)$ has level at least as large as the level of $j(E')$ is not a good approach to the coming combinatorial problem.  We handle the latter first:
\\ \\
Suppose that we have a nonbacktracking path $P$ of length $a$ in $\mathcal{G}_{K,\ell,\ff_0}$ corresponding to $\varphi: E \ra E'$, 
and such that $E$ is $\Delta = \ell^{2L} \ff_0^2 \Delta_K$-CM and $E'$ is $\Delta' = \ell^{2L'} \ff_0^2 \Delta_K$-CM with $L' > L$, and 
put $\ff = \ell^{L} \ff_0$, $\ff' = \ell^{L'} \ff_0$.  By the above analysis, the field of moduli $\Q(\varphi)$ is either $\Q(j(E'))$ (which is isomorphic though not necessarily equal to $\Q(\ff')$) or $K(j(E')) = K(\ff')$.  If the path $P$ contains a horizontal edge in the split case 
then we have $\Q(\varphi) = K(\ff')$, so suppose that is not the case.  Then we have $\Q(\varphi) = \Q(j(E'))$ if and only if $j_{\Delta}$ and 
$j(E')$ are coreal.  Let $P_1$ be the maximal initial segment of the path $P$ that terminates at a vertex in level $L$, and let $P_2$ be the 
rest of the path, so $P_2$ consists entirely of $L'-L$ descending edges.   Let $a_1 < a$ be the length of $P_1$, and let 
$\varphi_1: E \ra E_1$ be the corresponding factor isogeny.  Then $\Q(j(E_1)) \subset \Q(j(E'))$, so if $j_{\Delta}$ and $j(E')$ are 
coreal then so are $j_{\Delta}$ and $j(E_1)$, and since $E$ and $E_1$ have the same endomorphism ring, this occurs if and only if $j(E_1) \in \R$.  
Conversely, if $j(E_1) \in \R$ then $\Q(j(E_1)) = \Q(j_{\Delta})$
and thus  
\[ \Q(j_{\Delta},j(E')) = \Q(j(E_1),j(E')) = \Q(j(E')). \]
Since we wish to count closed points in the fiber of $X_0(\ell^a) \ra X(1)$ over $J_{\Delta}$, we need to impose an equivalence relation 
on paths: any path in the same $\gg_{\Q(\ff)}$-orbit as $P(\varphi)$ determines the same closed point on $X_0(\ell^a)$ as 
$P(\varphi)$.  The size of this Galois orbit is 
\[ d_{\varphi} \coloneqq [\Q(\varphi):\Q(\ff)]. \]
Complex conjugation acts on paths in $\mathcal{G}_{K,\ell,\ff_0}$, and a path is real if and only if each of its edges is real.

\begin{lemma}
\label{CLOSEDPOINTLEMMA}
We maintain notation as above.  Put \[ \epsilon_{\varphi} \coloneqq \begin{cases} 1 & P_1 \text{ is real} \\ 2 & \text{ otherwise} \end{cases}. \]  Then:
\begin{itemize}
\item[a)] If $L \geq L'$, then we have $d_{\varphi} = \epsilon_{\varphi}$.
\item[b)] If $L = 0$ and $L' > L$, then we have $d_{\varphi} = \epsilon_{\varphi} (\ell + \left(\frac{\Delta_K}{\ell} \right))\ell^{L'-L-1}$.  
\item[c)] If $0 < L < L'$, then we have $d_{\varphi} = \epsilon_{\varphi} \ell^{L'-L}$.
\end{itemize}
\end{lemma}
\begin{proof}
This follows easily from the description of $\Q(\varphi)$ we have given.
\end{proof}
\noindent
We can also explicitly describe the $\gg_{\Q(\ff)}$-orbit on $P(\varphi)$: it consists of all paths obtained from $P_1$ by descending $\max(L'-L,0)$ times as well as all paths obtained from $\overline{P_1}$ by descending $\max(L'-L,0)$ times.  We will say that 
two paths in the same $\gg_{\Q(\ff)}$-orbit are \textbf{closed point equivalent}.

\subsection{Cyclic $N$-isogenies on real elliptic curves}

\begin{thm}
\label{LITTLECLEMMA}
Let $E_{/\R}$ be an elliptic curve.  For $N \in \Z^+$, let $\mathcal{C}_N(E)$ denote the set of cyclic order $N$ subgroups 
of $E(\C)$, and let $\mathcal{C}_N(E)^c$ denote the subset that is fixed under the natural action of $\gg_{\R} = \{1,c\}$.  Let $t$ be the number of real roots $f$, where $y^2 = f(x)$ is a polynomial defining $E_{/\R}$.  Suppose that $N$ has $r$ distinct odd prime divisors. 
\begin{itemize}
\item[a)] If $N$ is odd, then $\# \mathcal{C}_N(E)^c = 2^r$.  
\item[b)] If $N \equiv 2 \pmod{4}$, then $\# \mathcal{C}_N(E)^c = t \cdot 2^r$.
\item[c)] If $N \equiv 4 \pmod{8}$, then $\# \mathcal{C}_N(E)^c \in \{2^{r+1},2^{r+2}\}$.  
\end{itemize}
\end{thm} 
\begin{proof}
Step 1: Write $N = \ell_1^{a_1} \cdots \ell_R^{a_R}$ with $2 \leq \ell_1 < \ell_2 < \ldots < \ell_R$.  There is a $\gg_{\R}$-equivariant 
bijection from $\prod_{i=1}^R \mathcal{C}_{\ell_i^{a_i}}(E)$ to $\mathcal{C}_{N}(E)$ obtained by mapping the tuple $(C_1,\ldots,C_R)$ to the subgroup $\langle C_1,\ldots,C_R \rangle$.  This reduces us to the case of $N = \ell^a$ a prime power.  \\
Step 2: Suppose that $N = \ell$ is an odd prime.  We recall a version of a well-known fact: let $R$ be a commutative ring with $2 \in R^{\times}$, let $M$ be an $R$-module, and let $c \in \End_R(M)$ satisfy 
$c^2 = 1$, and put
\[ M_+ \coloneqq \{x \in M \mid c(x) = x\} \text{ and } M_- \coloneqq \{x \in M \mid c(x) = -x\}. \]
Then we have $M = M_+ \oplus M_-$.  Indeed, if $x \in M_+ \cap M_-$, then $2x = 0$, so $x = 0$.  Also, for $x \in M$ we may 
write $x = \frac{x+c(x)}{2} + \frac{x-c(x)}{2}$. \\
Step 3: Suppose that $N = \ell^a$ is an odd prime power.   We apply Step 2 with $R = \Z/\ell^a\Z$, $M = E[\ell^a](\C)$ and take $c$ to 
be complex conjugation, so 
\[ E[\ell^a](\C) = E[\ell^a](\C)_+ \oplus E[\ell^a](\C)_-. \]
We have $E[\ell^a](\C)_+ = E[\ell^a](\R)$.  Since $E(\R)$ is isomorphic to either $\R/\Z$ or $\R/\Z \oplus \Z/2\Z$, we have (using that $\ell > 2$) that $E[\ell^a](\R) \cong \Z/\ell^a\Z$; since $E[\ell^a](\C) \cong (\Z/\ell^a\Z)^2$, it follows that $E[\ell^a](\C)_{-} \cong \Z/\ell^a \Z$.  Thus $E[\ell^a](\C)_+$ and $E[\ell^a](\C)_-$ are two elements of $\mathcal{C}_{\ell^a}(E)^c$.  If there were any others, 
there would be $x \in E[\ell^a](\C)$ and an integer $n$ prime to $\ell$ such that $c(x) = nx$ and, in the unique 
representation of $x$ as $x_+ + x_-$ with $x_+ \in E[\ell^a](\C)_+$ and $x_-  \in E[\ell^a](\C)_0$, we have $x_+,x_- \neq 0$, so each has order divisible by $\ell$.  Then 
\[ x_+ - x_- = c(x) = nx = nx_+ + n x_-, \]
which shows that $n \equiv 1 \pmod{\ell}$ and $n \equiv -1 \pmod{\ell}$, a contradiction.  \\ \indent 
This completes the proof of part a).  Part b) follows immediately, since $t = \# \mathcal{C}_2(E)$. \\
Step 4: Let $a \in \Z^+$.  The homomorphism 
\[ [2]: E(\C)[2^{a+1}] \ra E(\C)[2^a] \] 
is $\gg_R$-equivariant, surjective and such that each fiber is a principal homogeneous space under $E(\C)[2]$.  There is an induced 
$\gg_{\R}$-equivariant map 
\[ [2]: \mathcal{C}_{2^{a+1}}(E) \ra \mathcal{C}_{2^a}(E) \]
that is surjective, with each fiber of size $2$.  \\
Case 1: Suppose that $t = 1$, so $E(\R) \cong \R/\Z$.  We claim that $\#\mathcal{C}_{4}(E)^c = 2$.  If $C \in \mathcal{C}_4(E)^c$, then also $[2]C \in \mathcal{C}_2(E)^c$, and there is a unique such order $2$ subgroup, 
say, $C_2$, so it follows that $\# \mathcal{C}_4(E)^c \leq 2$.  On the other hand, since $E(\R)[4] \cong \Z/4\Z$ we certainly have $ \# \mathcal{C}_4(E)^c \geq 1$.  It follows that at least one of the two subgroups $C$ such that $[2] C = C_2$ is $\gg_{\R}$-stable, so 
the other such subgroup must be $\gg_{\R}$-stable as well.  \\
Case 2: Suppose that $t = 3$, so $E(\R) \cong \R/\Z \times \Z/2\Z$ and thus $E(\R)[4] \cong \Z/4\Z \times \Z/2\Z$.  
Then $\# \mathcal{C}_4(E)^c \in \{0,2,4,6\}$, depending upon how many of the $3$ elements of $\mathcal{C}_2(E)^c$ lift to a pair of elements 
on which complex conjugation acts trivially.  The $4$ elements of $E(\R)$ of order $4$ map to $2$ elements of $\mathcal{C}_2(E)^c$, 
so $\# \mathcal{C}_4(E)^c \geq 2$.  For any other element $C \in \mathcal{C}_4(E)^c$, if $P$ generates $C$ then $\overline{P} = -P$.  So if $\# \mathcal{C}_4(E)^c = 6$ then $E(\C)$ has eight elements of order $4$ that are inverted by complex conjugation.  But in the group 
$E(\C)[4] \cong \Z/4\Z \times \Z/4\Z$ the largest number of elements of order $4$ that generate a proper subgroup is $4$, so if 
$E(\C)$ had eight elements of order $4$ inverted by complex conjugation, then complex conjugation would have to act on $E(\C)[4]$ 
by $-1$, which we know is not the case.  
\end{proof} 

\begin{cor}
\label{REALPROJCOR}
Let $E_{/\R}$ be an elliptic curve, and let $N \geq 3$.  Then the projective $N$-torsion field $\R(\P E[N])$ of $E$ is $\C$.
\end{cor}
\begin{proof}It follows from Theorem \ref{LITTLECLEMMA} that $\# \mathcal{C}_N(E)^c < \# \mathcal{C}_N(E)$.
\end{proof}

\subsection{Explicit action of complex conjugation on $\mathcal{G}_{K,\ell,\ff_0}$}
\textbf{} \\ \indent
Suppose $\ff_0^2 \Delta_K < -4$.  We give an explicit description of the action of complex conjugation on the isogeny volcano -- up to $\gg_{\R}$-equivariant graph-theoretic isomorphism -- 
in all cases.  For $L \geq 0$, put 
\begin{equation}
\label{RLEQ}
\mathfrak{r}_L \coloneqq \# \Pic \OO(\ell^{2L} \ff_0^2 \Delta_K)[2]. 
\end{equation}
By Corollary \ref{COR2.5}, $\mathfrak{r}_L$ is the number of real vertices in $\mathcal{G}_{K,\ell,\ff_0}$ at level $L$.  Lemma \ref{GENUSLEMMA} computes $\mathfrak{r}_L$ in terms of 
$\mathfrak{r}_0$.

\begin{figure}
\includegraphics[scale=0.6]{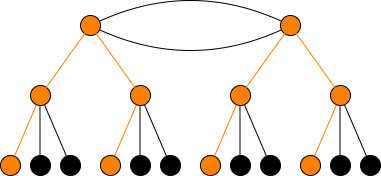}
\caption{Corollary \ref{LITTLECCOR}}{$\ell > 2$ split}
\end{figure}
\begin{figure}
\includegraphics[scale=0.6]{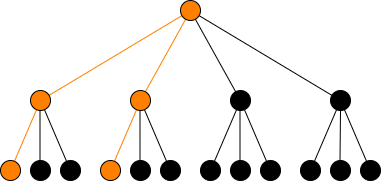}
\caption{Corollary \ref{LITTLECCOR}}{$\ell > 2$ inert}
\end{figure}

\begin{figure}
\includegraphics[scale=0.6]{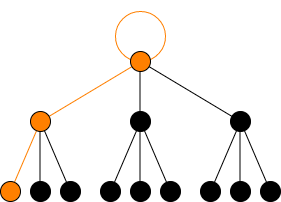}
\caption{Corollary \ref{LITTLECCOR}}{$\ell > 2$ ramified, Part I}
\end{figure}

\begin{figure}
\includegraphics[scale=0.6]{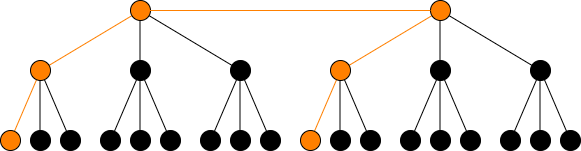}
\caption{Corollary \ref{LITTLECCOR}}{$\ell > 2$ ramified, Part II}
\end{figure}

\begin{cor}
\label{LITTLECCOR}
Let $\ff_0^2 \Delta_K < -4$, let $\ell > 2$, and let $v$ be a real vertex in the volcano $\mathcal{G}_{K,\ell,\ff_0}$.
\begin{itemize}
\item[a)] Suppose $\ell$ is unramified in $K$.  If $v$ is a surface vertex, then there are precisely two real descending edges with initial 
vertex $v$.  If $v$ lies below the surface, there is a unique real descending edge with initial vertex $v$. 
\item[b)] Suppose that $\ell$ is ramified in $K$.  Then there is a unique real descending edge with initial vertex $v$.

\end{itemize}
\end{cor}
\begin{proof} a) Suppose $v$ lies at the surface.  If $\ell$ is inert in $K$ there are no surface edges, so the result follows from Theorem 
\ref{LITTLECLEMMA}.  If $\ell$ is split there are two surfaces edges, which are interchanged by complex conjugation, so again Theorem \ref{LITTLECLEMMA} applies.  If $v$ lies below the surface, then there is a unique ascending edge with initial vertex $v$, 
so it must be stable under complex conjugation.  By Theorem \ref{LITTLECLEMMA} there is exactly one descending real edge.  \\
b) Suppose $v$ lies at the surface.  Then there is a unique surface edge with initial vertex $v$, so by Theorem \ref{LITTLECLEMMA} 
there is exactly one descending real edge with initial vertex $v$.  If $v$ lies below the surface, the argument is the same as that 
of part a).
\end{proof}

\begin{lemma}
\label{EXPCLEMMA3}
\label{LEMMA5.5}
Let $\ff_0^2 \Delta_K < -4$, and suppose that $\ell = 2$ does not ramify in $K$.  Then in the volcano $\mathcal{G}_{K,2,\ff_0}$:
\begin{itemize}
\item[a)] Every real surface vertex has a unique real descendant.  
\item[b)] For $L \in \{1,2\}$, both of the descendants of every real vertex of level $L$ are real.
\item[c)] For $L \geq 3$, we partition the real vertices of level $L$ into pairs of vertices $\{v_L,w_L\}$, such that $v_L$ and $w_L$ 
are adjacent to the same vertex $u_{L-1}$ in level $L-1$.  Then exactly one of $v_L$ and $w_L$ has two real descendants and the 
other has no real descendants.
\end{itemize}

\begin{figure}
\includegraphics[scale=0.6]{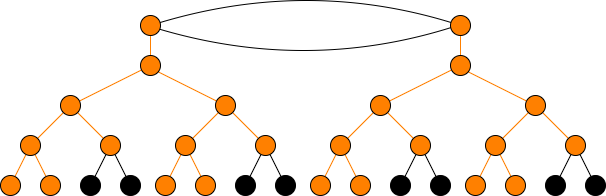}
\caption{ Lemma \ref{LEMMA5.5}}{$\ell =2$ split}
\end{figure}

\begin{figure}
\includegraphics[scale=0.5]{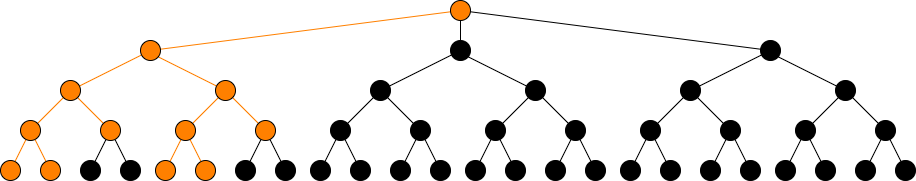}
\caption{ Lemma \ref{LEMMA5.5}}{$\ell =2$ inert}
\end{figure}

\end{lemma}
\begin{proof}
In this case Lemma \ref{GENUSLEMMA} gives 
\[ \mathfrak{r}_1 = \mathfrak{r}_0, \ \mathfrak{r}_2 = 2 \mathfrak{r}_1, \ \mathfrak{r}_3 = 2 \mathfrak{r}_2, \ \mathfrak{r}_{c+1} = \mathfrak{r}_c \ \forall c \geq 3. \]
a) If $2$ is inert in $K$, then every real surface vertex has three descendants.  Since $3$ is odd, at least one must be real.  Since the number of real vertices in level $1$ is the same as in level $0$, exactly one must be real, establishing part a) in this case.  If $2$ splits 
in $K$, then every real surface vertex has a unique descending vertex, which must therefore be real.  \\
b) For all $L \geq 1$, every vertex at level $L$ has exactly two descendant vertices.  Since $\mathfrak{r}_2 = 2 \mathfrak{r}_1$ and $\mathfrak{r}_3 = 2 \mathfrak{r}_2$, it must be that for $L \in \{1,2\}$ every real vertex at level $L$ has both of its descendants real.  \\
c) Suppose now that $L \geq 3$ and let $v_L,w_L$ be a pair of real vertices at level $L$ as in the statement of the result.  (If $v_L$ 
is real and is incident to $u_{L-1}$ in level $L-1$, then $u_{L-1}$ is real with at least one of its two descendant vertices real, so 
the other one, $w_L$, must also be real.)  If we can show that $v_L$ and $w_L$ do not both have descendant real vertices, then 
an easy counting argument using $\mathfrak{r}_{L+1} = \mathfrak{r}_L$ establishes the desired conclusion.  So assume now, and let $
v_{L+1}$ be a real descendant of $v_L$ and $w_{L+1}$ be a real descendant of $w_L$.  Then 
\[ v_{L+1} \mapsto v_L \mapsto u_{L-1} \mapsto w_L \mapsto w_{L+1} \]
is a proper, real cyclic $2^4$-isogeny with source elliptic curve having discriminant $\Delta = 2^{2L+2} \ff_0^2 \Delta_K$, so 
there is a primitive, proper real $\OO(\Delta)$-ideal of index $16$.  But by Theorem \ref{KWONLEMMA3.1}, if there is a primitive, proper cyclic real $\OO(\Delta)$-ideal of index $2^a$, then $a = 2$ or $a = \ord_2(\Delta) - 2 = 2L \geq 6$, a contradiction.
\end{proof}
\noindent
To describe the structure in the next case we need a simple preliminary result.

\begin{lemma}
\label{LEMMA5.6}
Let $\Delta$ be an even imaginary quadratic discriminant, let $\OO$ be the imaginary quadratic order of discriminant $\Delta$, and 
let $\pp$ be the unique ideal of $\OO$ of norm $2$.  Then $\pp$ is principal if and only if $\Delta \in \{-4,-8\}$.
\end{lemma}
\begin{proof}
For any $N \in \Z^+$, there is a principal $\OO$-ideal of norm $N$ if and only if $N$ is integrally represented by the quadratic form $x^2 + |\frac{\Delta}{4}|y^2$.  This form represents $2$ if and only if $\Delta \in \{-4,-8\}$.  
\end{proof}

\begin{lemma}
\label{EXPCLEMMA4}
\label{LEMMA5.7}
Let $\ff_0^2 \Delta_K < -4$, and suppose that $2$ ramifies in $K$ and that $\ord_2(\Delta_K) = 2$.  Then in the volcano 
$\mathcal{G}_{K,2,\ff_0}$:
\begin{itemize}
\item[a)] The set of real surface vertices is canonically partitioned into pairs $\{v_0,w_0\}$ such that $v_0$ has two real descendants 
and $w_0$ has no real descendants.  
\item[b)] Every descendant vertex of a real vertex in level $1$ is real.
\item[c)]  For $L \geq 2$, we partition the real vertices of level $L$ into pairs of vertices $\{v_L,w_L\}$, such that $v_L$ and $w_L$ 
are adjacent to the same vertex $u_{L-1}$ in level $L-1$.  Then exactly one of $v_L$ and $w_L$ has two real descendants and the 
other has no real descendants.  
\end{itemize}

\begin{figure}
\includegraphics[scale=0.6]{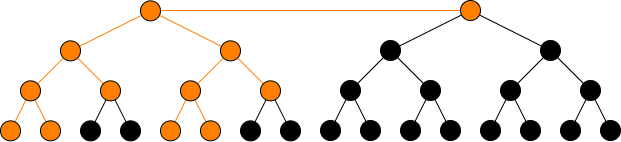}
\caption{ Lemma \ref{LEMMA5.7}}{$\ell =2, \ \ord_2(\Delta_K) = 2$}
\end{figure}

\end{lemma}
\begin{proof}
In this case Lemma \ref{GENUSLEMMA} gives 
\[ \mathfrak{r}_1 = \mathfrak{r}_0, \ \mathfrak{r}_2 = 2 \mathfrak{r}_1, \ \mathfrak{r}_{L+1} = \mathfrak{r}_L \ \forall L \geq 2. \]
a) Let $\OO_0$  be the imaginary quadratic order of discriminant $\ff_0^2 \Delta_K$.  Since $2$ divides $\Delta_K$ and does not 
divide $\ff_0$, the ring $\OO_0$ has a unique prime ideal $\pp$ of norm $2$, which is moreover proper.  Since $\pp^2 = (2)$, 
the class $[\pp] \in \Pic \OO$ has order at most $2$, and since $\Delta_K = -4$ has been excluded, by Lemma \ref{LEMMA5.6} 
the class $[\pp]$ has order exactly $2$.  This implies that there are no surface loops and indeed gives the partition of the 
set of surface vertices into pairs: it is the decomposition of $\Pic \OO_0$ into cosets of $\{1,[\pp]\}$.  Because $\mathfrak{r}_1 = 
\mathfrak{r}_0$, a counting argument shows that if the conclusion of part a) did not hold there would be a pair real 
vertices $v_0,w_0$ linked by a surface edge such that neither $v_0$ nor $w_0$ have any real descendants.  But in this case 
the real elliptic curve corresponding to $v_0$ would admit no real cyclic $4$-isogeny, whereas as we have observed above, 
every real elliptic curve admits a real cyclic $N$-isogeny for all positive integers $N$.  \\
b) This follows from $\mathfrak{r}_2 = 2 \mathfrak{r}_1$. \\
c) The argument for this is the same as for Lemma \ref{EXPCLEMMA3}c).
\end{proof}

\begin{lemma}
\label{EXPCLEMMA5}
\label{LEMMA5.8}
In the volcano $\mathcal{G}_{\Q(\sqrt{-8}),2,1}$:
\begin{itemize}
\item[a)] There is one surface vertex $v_0$, which is real.    
\item[b)] Both of the descendants $v_1,w_1$ of $v_0$ are real.
\item[c)] For all $L \geq 1$, there are two real vertices $v_L$, $w_L$ in level $L$ that are descendants of the same real 
vertex $v_{L-1}$ in level $L$.  The vertex $v_L$ has two real descendants $v_{L+1}$, $w_{L+1}$, and the vertex $w_L$ 
has no real descendants.
\end{itemize}
\begin{figure}
\includegraphics[scale=0.6]{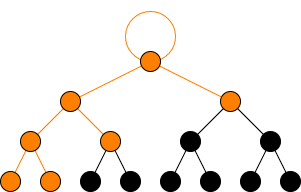}
\caption{Lemma \ref{LEMMA5.8}}{$\ell =2, \Delta_K = -8$}
\end{figure}

\end{lemma}
\begin{proof}
We have $\mathfrak{r}_1 = 2 \mathfrak{r}_0 = 2$ and $\mathfrak{r}_{L+1} = \mathfrak{r}_L = 2$ for all $L \geq 1$.  Also we have 
$h_{-8} = 1$.  The claimed structure follows from this, and we leave the details to the reader.  
\end{proof}

\begin{lemma}
\label{EXPCLEMMA6}
\label{LEMMA5.9}
Suppose that $\Delta_K < -8$ and $\ord_2(\Delta_K) = 3$.  Then in the volcano $\mathcal{G}_{K,2,\ff_0}$:
\begin{itemize}
\item[a)] Every descendant vertex of a real surface vertex is real. 
\item[b)] For $L \geq 1$, we partition the real vertices of level $L$ into pairs of vertices $\{v_L,w_L\}$, such that $v_L$ and $w_L$ 
are adjacent to the same vertex $u_{L-1}$ in level $L-1$.  Then exactly one of $v_L$ and $w_L$ has two real descendants and the 
other has no real descendants.  
\end{itemize}
\begin{figure}
\includegraphics[scale=0.6]{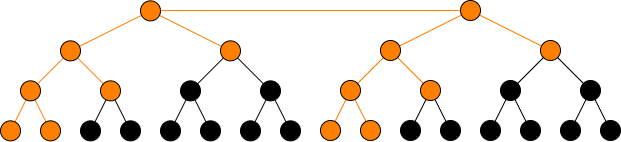}
\caption{ Lemma \ref{LEMMA5.9}}{$\ell =2, \ord_2(\Delta_K) = 3$}
\end{figure}

\end{lemma}
\begin{proof} We have $\mathfrak{r}_1 = 2 \mathfrak{r}_0$ and $\mathfrak{r}_{L+1} = \mathfrak{r}_L$ for all $L \geq 1$.  
The arguments are similar to those in the previous cases and may be left to the reader.
\end{proof}

\section{Some Applications}

\subsection{The Field of Moduli of an Isogeny}

\begin{lemma}
\label{LEMMA7.1}
Let $K$ be an imaginary quadratic field with $\Delta_K < -4$, and let $\varphi: E \ra E'$ be an isogeny of complex $K$-CM elliptic curves.  
Then we have 
\[ \Q(j((E),j(E')) \subseteq \Q(\varphi) \subseteq K(\varphi) = K(j(E),j(E')). \]
\end{lemma}
\begin{proof}
Proposition \ref{PROP3.1}a) reduces us to the case in which $\varphi$ is a cyclic $N$-isogeny.  Certainly we have that $K(\varphi)$ contains both $\Q(\varphi)$ and $K(j(E),j(E'))$, so it is enough to show that $K(\varphi) \subset K(j(E),j(E'))$.  \\ \indent
If $N = \ell_1^{a_1} \cdots \ell_r^{a_r}$ then $\varphi = \varphi_r \circ \ldots \circ \varphi_1$, where $\varphi_i: E_{i-1} \ra E_i$ is a cyclic $\ell_i^{a_i}$-isogeny.  (Thus $E_0 = E$ and $E_r = E'$.)  For $0 \leq i \leq r$, let $\ff_i$ be the conductor of $E_i$.  Put $M \coloneqq \lcm(\ff_0,\ff_r)$.  By
 (\ref{KJEQ}) and Proposition \ref{TEDIOUSALGEBRAPROP1} we have 
\[ K(j(E),j(E')) = K(\ff_0)K(\ff_r) = K(M). \]
From our isogeny volcano analysis, we know that 
\[K(\varphi_i) = K(j(E_i),j(E_{i+1})) = K(\lcm(\ff_i,\ff_{i+1})). \]
Moreover, for all $1 \leq i \leq r$ we have $\ff_{i}/\ff_{i-1} \in \ell_i^{\Z}$, from which it follows that for all $1 \leq i \leq r$ and 
$0 \leq j \leq r$ we have $\ord_{\ell_i}(\ff_j) \leq \ord_{\ell_i}(M)$ and thus 
\[ K(\varphi) = K(\varphi_1) \cdots K(\varphi_r) \subset K(M).  \qedhere\]
\end{proof}

\begin{thm}
\label{7.2}
Let $K$ be an imaginary quadratic field with $\Delta_K < -4$, and let $\varphi: E \ra E'$ be an isogeny of complex $K$-CM elliptic curves.  Then: 
\begin{itemize}
\item[a)] If $j(E)$ and $j(E')$ are not coreal, then $\Q(\varphi) = \Q(j(E),j(E')) = K(j(E),j(E'))$.
\item[b)] If $j(E)$ and $j(E')$ are coreal and there is a prime $\ell$ that splits in $K$ such that $\varphi_{\cyc}$ factors through a 
proper $\ell$-isogeny then $\Q(\varphi) = K(j(E),j(E')) \supsetneq \Q(j(E),j(E'))$.
\item[c)] Suppose that $\varphi_{\cyc}$ has prime power degree. If $j(E)$ and $j(E')$ are coreal and for no prime $\ell$ that splits in $K$ does $\varphi_{\cyc}$ factor through a proper 
$\ell$-isogeny, then $\Q(\varphi) = \Q(j(E),j(E')) \subsetneq K(j(E),j(E'))$. 

\end{itemize}
\end{thm}
\begin{proof} Once again, by Proposition \ref{PROP3.1}a) we may assume that $\varphi$ is a cyclic $N$-isogeny.  In view of Lemma \ref{LEMMA7.1}, 
we have 
\begin{equation}
\label{7.2EQ}
\Q(j(E),j(E')) \subset \Q(\varphi) \subset K(\varphi) = K(j(E),j(E')). 
\end{equation}
Moreover, by Theorem \ref{3.9} we have $\Q(j(E),j(E')) = K(j(E),j(E'))$ if and only if $j(E)$ and $j(E')$ are not coreal.  From this and (\ref{7.2EQ}) 
part a) follows immediately.  If $\ell$ splits in $K$ and $\varphi$ factors through a proper $\ell$-isogeny $\iota$, we know that $\Q(\varphi) \supset \Q(\iota) \supset K$ and thus $\Q(\varphi) = K(j(E),j(E'))$, which shows part b).   \\
c) Finally, assume that $N = \ell^a$ is a prime power, that $j(E)$ and $j(E')$ are coreal, and that if $\ell$ splits in $K$ the 
isogeny $\varphi$ does not factor through a proper $\ell$-isogeny.  In this case, our theory of isogeny volcanoes shows 
that $\Q(\varphi) = \Q(j(E),j(E'))$: indeed, the acending part of the corresponding path yields an isogeny that is defined over 
$\Q(j(E))$, which shows that the $j$-invariant of each vertex in this portion of the path lies in $\Q(j(E))$.  Surface vertices, 
which exist only in the ramified case, are therefore defined over $\Q(j(E))$, while the descending part of the corresponding path 
yields an isogeny that is defined over $\Q(j(E'))$.  So the isogeny $\varphi$ is defined over $\Q(j(E),j(E'))$.  
\end{proof}

\begin{remark}
I suspect that Theorem \ref{7.2}c) holds without the hypothesis that $\varphi_{\cyc}$ has prime power degree.  Otherwise and perhaps more simply put, for any cyclic $N$-isogeny of CM elliptic curves that does not factor through a proper $\ell$-isogeny for a prime $\ell$ 
that splits in $K$, I suspect that $\Q(\varphi) = \Q(j(E),j(E'))$, just as Proposition \ref{PROP3.2} shows is the case for all elliptic curves without complex multiplication. 
\end{remark}

\begin{cor}
\label{FOMCOR}
Let $\Delta = \ff^2 \Delta_K$ with $\Delta_K < -4$, let $E_{/\C}$ be a $\Delta$-CM elliptic curve and let $\varphi: E \ra E'$ be an 
isogeny.  Then there is $M \in \Z^+$ such that the field of moduli $\Q(\varphi)$ of $\varphi$ is isomorphic to either $\Q(M \ff)$ or to 
$K(M\ff)$.  
\end{cor}

\subsection{$K(\ff)$-rational cyclic $N$-isogenies}
The following is a result of Bourdon-Clark \cite[Thm. 6.18a)]{BCI}.  

\begin{thm}
\label{BCTHM6.18}
Let $\Delta < -4$, and write $\Delta = \ff^2 \Delta_K$.  For $N \in \Z^+$, there is a $\Delta$-CM elliptic curve $E_{/K(\ff)}$ 
and a $K(\ff)$-rational cyclic $N$-isogeny $\varphi: E \ra E'$ if and only if $\Delta$ is a square in $\Z/4N\Z$.
\end{thm}
\noindent
Here we suppose moreover that $\Delta_K < -4$ and give a different proof of Theorem \ref{BCTHM6.18}.  As explained in \S 3.5, 
since $\Delta_K < -4$, for a positive integer $N = \ell_1^{a_1} \cdots \ell_r^{a_r}$ there is a $K(\ff)$-rational cyclic 
$N$-isogeny with source elliptic curve $\Delta$-CM if and only if for all $1 \leq i \leq r$ there an $K(\ff)$-rational cyclic 
$\ell_i^{a_i}$-isogeny with source elliptic curve $\Delta$-CM.  Moreover, we have that $\Delta$ is a 
square in $\Z/4\ell_1^{a_1} \cdots \ell_r^{a_r}\Z$ if and only if $\Delta$ is a square in $\Z/4 \ell_i^{a_i} \Z$ for all $1 \leq i \leq r$, so we 
reduce to the case in which $N = \ell^a$ is a prime power.   Let $M(\Delta,\ell)$ be the supremum of positive integers 
$a$ such that some (equivalently, every) $\Delta$-CM elliptic curve has a $K(\ff)$-rational cyclic $\ell^a$-isogeny.  If $L = \ord_{\ell}(\ff)$, 
this quantity can be understood in terms of the volcano $\mathcal{G}_{K,\ell,\ff_0}$ as the supremum of all lengths of nonbacktracking 
paths starting at a vertex of level $L$ and ending at a level $L' \leq L$.  
\\ \\
$\bullet$ Suppose that $\left( \frac{\Delta_K}{\ell} \right) = 1$.  In this case $M(\Delta,\ell) = \infty$: indeed, we can ascend to the surface and follow by a nonbacktracking path of arbitrary length on the surface.  Writing $\Delta = \ff^2 \Delta_K$, it is enough to check 
that $\Delta_K$ is a square modulo $4 \ell^a$ for all positive integers $a$, which holds if and only if it is a square modulo $\ell^a$ for all 
positive integers $a$ and is indeed the case by Hensel's Lemma since this holds modulo $\ell$ if $\ell > 2$ (resp. modulo $8$ if $\ell = 2$).
\\ \\
$\bullet$ Suppose that $\left( \frac{\Delta}{\ell} \right) = -1$.  In this case we have $L = 0$ so we are on the surface and in the inert 
case so have no surface edges, so every nonbacktracking path of positive length ends lower than it starts, so the field of moduli 
of the corresponding $\ell^a$-isogeny contains $K(\ell \ff) \supsetneq K(\ff)$.  
If $\ell > 2$ then our assumption gives that $\Delta$ is not a square modulo $\ell$.  If $\ell = 2$ then our assumption gives 
$\Delta \equiv 5 \pmod{8}$ hence is not a square modulo $8 = 4 \ell$.  
\\ \\
$\bullet$ Suppose that $L = \ord_{\ell}(\ff) \geq 1$ and $\left( \frac{\Delta_K}{\ell} \right) = -1$.  In this case the longest nonbacktracking 
path starting at level $L$ and ending no deeper than level $L$ is a path that ascends to the surface and descends back down to 
level $L$, so $M(\Delta,\ell) = 2L$.  One checks that the largest $a$ such that $\Delta = \ell^{2L} \ff_0^2 \Delta_K$ is a 
square modulo $\ell^a$ is $a = 2L$: see the end of \cite[\S 7.4]{BCI} for the details.
\\ \\
$\bullet$ Suppose that $\left( \frac{\Delta_K}{\ell} \right) = 0$.  In this case the longest nonbacktracking path starting at level 
$L$ and ending no deeper than level $L$ is a path that ascends to the surface, takes the unique surface edge and then descends 
back down to level $L$, so $M(\Delta,\ell) = 2L+1$.  One checks that the largest $a$ such that $\Delta = \ell^{2L} \ff_0^2 \Delta_K$ 
is a square modulo $4\ell^a$ is $a = 2L+1$  see the end of \cite[\S 7.4]{BCI} for the details.

\begin{remark}
Suppose $\Delta_K \in \{-4,-3\}$ but $\Delta = \ff^2 \Delta_K < -4$.  Then using the known structure of 
the isogeny graph $\mathcal{G}_{K,\ell,\ff_0}$ \cite[\S 3]{CS22b}, it is easy to carry the above argument over to 
give a full proof of Theorem \ref{BCTHM6.18}.  \\ \indent
That isogeny volcanoes give a different and arguably more transparent approach to Theorem \ref{BCTHM6.18} 
when was pointed out to me by A. Sutherland in January 2019.  This conversation led to the present work.
\end{remark}

\subsection{$\Q(\ff)$-rational cyclic $N$-isogenies}
Let $\Delta = \ff^2 \Delta_K$ be an imaginary quadratic order with $\Delta_K < -4$.  For a positive integer $N$, let $I(\Delta,N)$ 
be the condition that there is a $\Delta$-CM elliptic curve $E$ defined over a number field $F$ isomorphic to $\Q(\ff)$ (and thus 
$F = \Q(j(E))$) such that $E$ admits an $F$-rational cyclic $N$-isogeny.  In \cite{Kwon99}, S. Kwon determined all pairs $(\Delta,N)$ 
(again, with $\Delta_K < -4$) for which $I(\Delta,N)$ holds.   We will deduce this result, along with some later generalizations, from the present 
work.  
\\ \\
As we have seen, it is no loss of generality to assume that we have a $\Delta$-CM elliptic curve defined over $\Q(\ff) = \Q(j(\C/\OO)$.  As in the previous section we immediately reduce to the case $N = \ell^a$ is a prime power.  Indeed, for each prime $\ell$, we let 
$m_{\ell}(\Delta)$ be the largest non-negative integer $a$ such that $I(\Delta,\ell^a)$ holds.   We will see shortly that for each 
fixed $\Delta$ we have $m_{\ell}(\Delta) = 0$ for all sufficiently large primes $\ell$.  It then follows that $I(\Delta_N)$ holds 
if and only if $N \mid \prod_{\ell} \ell^{m_{\ell}(\Delta)}$.  
\\ \\
We have the following ``volcanic'' interpretation of $m_{\ell}(\Delta)$: it is the longest length $a$ of a 
real path in $\mathcal{G}_{K,\ell,\ff_0}$ starting at level $L$ and ending at level $L'$ such that $\Q(\ell^{2L'} \ff_0^2 \Delta_K) \hookrightarrow \Q( \ell^{2L} \ff_0^2 \Delta_K)$ (as we will see, in all but one case the latter condition simplifies to $L' \leq L$).  
\\ \\
$\bullet$ Suppose that $\ell > 2$ and $\left( \frac{\Delta_K}{\ell} \right) \neq 0$.  Then such a path must begin with an upward component --
there are no real horizontal edges and $\Q(\ell \ff) \supsetneq \Q(\ff)$ in this case.  The longest such path ascends all the way to the surface and then descends back down to level $L$, which is possible because every real surface vertex has two real descendants in this case.  Thus we get $m_{\ell}(\Delta) = 2L$.
\\ \\
$\bullet$ Suppose that $\ell > 2$ and $\left( \frac{\Delta_K}{\ell} \right) = 0$.  If $L = 0$ the only nontrivial such path takes 
the unique horizontal edge emanating from the corresponding surface vertex, so $m_{\ell}(\Delta) = 1$.  Otherwise the longest 
such path ascends to the surface, takes the unique horizontal edge, and descends back to level $L$.  Thus in general we get 
$m_{\ell}(\Delta) = 2L+1$.  
\\ \\
$\bullet$ Suppose that $\ell = 2$ and $\left( \frac{\Delta}{2} \right) = 1$.  In this case $L = 0$ and there are no real horizontal edges, so we must descend.  We have $\Q(2 \ff) = \Q(\ff)$ and $[\Q(4\ff):\Q(2\ff)] = 2$ so that we can descend once: $m_2(\Delta) = 1$.
\\ \\
$\bullet$ Suppose that $\ell = 2$ and $\left( \frac{\Delta}{2} \right) = -1$.  In this case $L = 0$, there are no horizontal edges and 
$\Q(2 \ff) \supsetneq \Q(\ff)$, so $m_2(\Delta) = 0$.
\\ \\
$\bullet$ Suppose that $\ell = 2$, $\left( \frac{\Delta}{2} \right) = 0$ and $\left( \frac{\Delta_K}{2} \right) \neq 0$.  If $L = 0$ then we must go downward, but 
$\Q(\ell \ff) \supsetneq \Q(\ff)$, so $m_2(\Delta) = 0$.  If $L = 1$ we can go upward once and then there is no nonbacktracking real edge, so $m_2(\Delta) = 1$.  If $L \geq 2$ then the longest path ascends to level $1$ and then descends back to level $L$ (cf. Lemma 5.5), so we have $m_2(\Delta) = 2L-2$.  
\\ \\
$\bullet$ Suppose that $\ell = 2$, $\left( \frac{\Delta_K}{2} \right) = 0$ and $\ord_2(\Delta_K) = 2$.  In this case we have 
$\Q(2 \ff) \supsetneq \Q(\ff)$, we want the length of the longest real path ending at level $L' \leq L$.   If $L = 0$ we can take 
the horizontal edge, giving $m_2(\Delta) = 1$. If $L \geq 1$ we can ascend to the surface and then descend back down to 
level $L$, giving $m_2(\Delta) = 2L$.
\\ \\
$\bullet$ Suppose that $\ell = 2$, $\left( \frac{\Delta_K}{2} \right) = 0$ and $\ord_2(\Delta_K) = 3$.  Again we have 
$\Q(2 \ff) \supsetneq \Q(\ff)$, so we want the length of the longest real path ending at level $L' \leq L$.  If $L = 0$ we can take the 
horizontal edge, giving $m_2(\Delta) = 1$.  If $L \geq 1$ we can ascend to the surface, take the horizontal edge and then descend 
back down to level $L$, giving $m_2(\Delta) = 2L+1$.  
\\ \\
This agrees with Kwon's result.   To confirm this it is easier to check against \cite[Prop. 6.4]{BCII}, which records 
$m_{\ell}(\Delta)$ in all cases in which $\ell \mid \Delta$.  

\begin{remark}The factorization of a pleasant isogeny given in \S 3.4 is essentially due to Kwon and is a key part of his 
proof of the classification of cyclic $N$-isogenies over $\Q(j(E))$.  He also develops and uses some ideal theory of imaginary quadratic orders.  Thus our approach to Kwon's results is relatively similar to his, the main difference being use of isogeny volcanoes.  Let us also 
recall that in the determination of the action of complex conjugation on the isogeny volcano in some $\ell = 2$ cases we used Kwon's 
Theorem \ref{KWONLEMMA3.1}.
\end{remark}


\subsection{Finiteness of isogenies over a number field} In the previous two sections we proved two results of the the following form: ``For an imaginary 
quadratic discriminant $\Delta = \ff^2 \Delta_K$ with $\Delta_K < -4$ and a number field $F$, we completely determine 
the set of $N \in \Z^+$ such that there is an $F$-rational $\Delta$-CM point on $X_0(N)$.''  Armed with the results of our present work, we could derive a similar classification for any fixed number field $F$. But is perhaps more interesting to think about this classification 
depends on $F$.  In \cite[Thm. 5.3]{BCII} it is shown that Kwon's classification of $\Q(\ff)$-rational cyclic isogenies remains valid if $\Q(\ff)$ 
is replaced by any number field that contains neither $K$ nor a field isomorphic to $\Q(\ell \ff)$ for any prime $\ell$.  The theory presented here makes this extension immediate.  More generally, for any subfield $F$ of $\C$, Corollary \ref{FOMCOR} implies that the set of positive integers $N$ for which 
$X_0(N)$ has an $F$-rational $\Delta$-CM point depends only on whether $F$ contains $K$ and for which $\ff \in \Z^+$ it contains a number field isomorphic to $\Q(\ff)$.  
\\ \\
In fact the classification can be done even for infinite degree extensions of $\Q$.  To stay within our setup, let us consider a subfield 
$F$ of the complex numbers.\footnote{The general case of a field $F$ of characteristic $0$ can be easily reduced to this.} The following result tells us exactly when 
our classification problem has a finite answer:

\begin{thm}
\label{LVCOMPTHM}
Let $F$ be a subfield of $\C$.
\begin{itemize}
\item[a)] Suppose that $F$ contains either the Hilbert class field of some imaginary quadratic field or infinitely many rational ring 
class fields.  Then the set of positive integers $N$ such that $X_0(N)$ has an $F$-rational CM point is infinite.
\item[b)] Suppose that $F$ does not contain the Hilbert class field of any imaginary quadratic field and contains only finitely many 
rational ring class fields.  Then the set of positive integers $N$ such that $X_0(N)$ has an $F$-rational $\Delta$-CM point is finite.
\end{itemize}
\end{thm}
\begin{proof}
a) Suppose first that there is an imaginary quadratic field $K$ such that $F$ contains the Hilbert class field $K(1)$ of $K$.  Since $K(1) = K(j_{\Delta})$, the field $F$ contains $j_{\Delta}$ and $K$ and hence, for each prime $\ell$ that splits in $K$ 
and all $a \in \Z^+$, there is a $K(1)$-rational cyclic $\ell^a$-isogeny $\varphi: E \ra E'$ where $j(E) = j_{\Delta_K}$, establishing the claim in this case.  \\ \indent
Next suppose that $F$ contains infinitely many rational ring class fields.  Since each rational ring class field has finitely many Galois 
conjugates, this holds if and only if there is an infinite set $\{ \Delta_n\}_{n=1}^{\infty}$ of imaginary quadratic discriminants such that for all 
$n \in \Z^+$ $F$ contains a subfield isomorphic to $\Q(j_{\Delta_n})$.  It follows that either the set of primes $\ell$ that 
divide $\Delta_n$ for some $n \in \Z^+$ is infinite or for all $A \in \Z^+$ there is $n_A \in \Z^+$, a prime $\ell_A$ and 
an $a_A \in \Z^+$ such that $\ell_A^{a_A} \mid \Delta_{A}$.  It follows from the results in the previous section (and from Kwon's Theorem, in particular) that when $\Delta_K < -4$, if $\ell^A \mid \Delta$, then there is a $\Q(j_{\Delta})$-rational $\Delta$-CM 
point on $X_0(\ell^A)$.  By \cite[Remark 5.2, Corollary 5.11]{BCII}, the assertions of the previous sentence hold also when $\Delta_K \in \{-3,-4\}$.  This completes the proof of part a).  \\
b) Suppose $F$ contains no Hilbert class field of an imaginary quadratic field and contains only finitely many rational ring class fields.  Then the set of imaginary quadratic discriminants $\Delta$ such that $F$ contains the $j$-invariant of any $\Delta$-CM 
elliptic curve is finite, so it is enough to fix an imaginary quadratic discriminant $\Delta$ such that $F$ contains a subfield isomorphic to $\Q(j_{\Delta})$ and show that the set of positive integers $N$ such that $X_0(N)$ admits an $F$-rational $\Delta$-CM point is finite.  Then $F$ does not contain $K$, for if so it would contain the ring class field of discriminant $\Delta$ and hence the 
Hilbert class field of $K$, contrary to our hypothesis. \\ \indent
We claim that we may further restrict to the case $\Delta = \Delta_K$.  Indeed, if $\Delta = \ff^2 \Delta_K$ and there is a cyclic 
$N$-isogeny $\varphi: E \ra E'$ defined over $F$ with $E$ a $\Delta$-CM elliptic curve, let $C_N$ be the kernel of $\varphi$.  
Let $\iota_{\ff,1}: E \ra \tilde{E}$ be the canonical, $F$-rational cyclic $\ff$-isogeny described in \S 3.4.   Then $\tilde{C}_N \coloneqq \iota_{\ff,1}(C_N)$ is an $F$-rational subgroup scheme of the $\Delta_K$-CM elliptic curve $E_{/F}$ of order a multiple of $\frac{N}{\gcd(N,\ff)}$.   In our situation $\ff$ is bounded, so $N$ becomes arbitrarily large if and only if $\frac{N}{\gcd(N,\ff)}$ does, in which case 
$\tilde{E} \ra \tilde{E}/\tilde{C}_N$ exhibits $F$-rational cyclic $\Delta_K$-CM isogenies of arbitarily large degree.  
\\ \indent
Let $\ell$ be a prime such that $\ell \nmid 2 \Delta_K$, and suppose $\varphi: E \ra E'$ is an $F$-rational $\ell$-isogeny 
with $E$ a $\Delta_K$-CM elliptic curve.  Since $\Delta_K$ is the discriminant of the maximal order $\Z_K$ in $K$, such an isogeny must be horizontal or descending, and since $\ell$ does not ramify in $K$ the only possible horizontal edges occur when $\ell$ splits in $K$, which have field of moduli containing $K$, hence not contained in $F$.  So $E'$ must be an $\ell^2 \Delta_K$-CM elliptic curve, 
and thus $\Q(j(E'))$ is isomorphic to $\Q(\ell)$.  By hypothesis this only holds for finitely many $\ell$.  \\ \indent
It remains to show that for each prime number $\ell$ there is a positive integer $A$ such that no $\Delta_K$-CM elliptic curve 
admits an $F$-rational cyclic $\ell^A$-isogeny.  Fix a prime $\ell$, and let $\varphi: E \ra E'$ be an $F$-rational cyclic $\ell^a$-isogeny, with 
$E$ a $\Delta_K$-CM elliptic curve.  Because of our assumption on $F$, there is a positive integer $B$ such if we factor 
$E$ into $\ell$-isogenies, then the conductor of the endomorphism ring of the source elliptic curve of every factor isogeny divides $\ell^B$.  In particular, if $\ff'$ is the conductor 
of $\End(E')$ and $L' \coloneqq \ord_{\ell}(\ff')$, then $L' \leq B$.  If there were $F$-rational cyclic $\ell^a$-isogenies 
$\varphi_a: E \ra E'$ with $E$ $\Delta_K$-CM for all $a \in \Z^+$, then for all $a \in \Z^+$ there is some $0 \leq L_a \leq B$ and 
an $F$-rational proper cyclic $\ell^a$-isogeny of $\Delta_a$-CM elliptic curves, where $\Delta_a = \ell^{2L_a} \Delta_K$. 
 (This is just because of the Pigeonhole Principle: since 
only finitely many endomorphism rings appear in the factor isogenies of the $\varphi_a$, as $a$ approaches infinity at least one 
endomorphism ring must appear arbitrarily often.)  Since $F$ does not contain $K$, a proper cyclic $\ell^a$ isogeny of $\OO$-CM elliptic curves is the kernel isogeny $E_1 \ra E_1/[\aa]$ attached to a primitive proper real $\OO(\Delta_a)$-ideal $\aa$ of norm $\ell^a$.  By
 \cite[Cor. 3.8]{BCS17} we must have 
\[ \ell^a = |\aa| \mid \Delta_a = \ell^{2L_a} \Delta_K \mid \ell^{2B} \Delta_K, \]
and for sufficiently large $a$, this gives a contradiction.  
\end{proof}
\noindent
The proof of Theorem \ref{LVCOMPTHM} shows that if $F$ contains no Hilbert class field of an imaginary quadratic field and only finitely many rational ring class fields, the set of positive integers $N$ such that $X_0(N)$ has an $F$-rational CM point is not only finite but explicitly bounded in terms of the finite set of discriminants $\Delta$ such that $F$ contains a rational ring class field of discriminant $\Delta$.  (This was clear in more generality when $\Delta_K < -4$, but the proof addresses the cases $\Delta_K \in \{-3,-4\}$ as well.)
As mentioned in \S 2.5, we have that $h_{\Delta} = [\Q(j_{\Delta}):\Q]$ approaches infinity with $|\Delta|$, so not only does every number field $F$ contain finitely many rational ring class fields, but the set of discriminants of such fields can be bounded in terms of $[F:\Q]$ alone.  Moreover, no number field of odd degree contains any imaginary quadratic field, let alone its Hilbert class field.  Accordingly, we deduce the following result.

\begin{cor}
\label{LVCOMPCOR}
\begin{itemize}
\item[a)]
For a number field $F$, the following are equivalent: 
\begin{itemize}
\item[(i)] The set of $N \in \Z^+$ such that $X_0(N)$ has an $F$-rational CM point is infinite.  
\item[(ii)] The field $F$ contains the Hilbert class field of an imaginary quadratic field.
\end{itemize}
\item[b)] Fix $d \in \Z^+$.  As we vary over all degree $d$ number fields $F$ that contain no Hilbert class field of an imaginary 
quadratic field, there are only finitely many positive integers $N$ such that $X_0(N)$ has an $F$-rational CM point.
\item[c)] For each odd $d \in \Z^+$, there are only finitely many $N \in \Z^+$ such that $X_0(N)$ has a closed CM point of degree $d$.  
\end{itemize}
\end{cor}
\noindent
We end this section with a brief comparison to the non-CM case.  A preliminary result:

\begin{lemma}
\label{PRELVLEMMA}
\begin{itemize}
\item[a)]
Let $F$ be a number field.  Suppose that for all sufficiently large prime numbers $\ell$, all the noncuspidal $F$-rational points on 
$X_0(\ell)$ are CM points.  Then the set of $N \in \Z^+$ such that $X_0(N)$ has a noncuspidal, non-CM $F$-rational point is finite.
\item[b)] Let $d \in \Z^+$.  Suppose that for all sufficiently large prime numbers $\ell$, all the noncuspidal degree $d$ closed points 
on $X_0(\ell)$ are CM points.  Then the set of $N \in \Z^+$ such that $X_0(N)$ has a noncuspidal, non-CM closed point of degree $d$ is 
finite.
\end{itemize}
\end{lemma}
\begin{proof}
a) By the assumption and the discussion at the beginning of \S 3.5, it suffices to show that for each fixed prime $\ell$, there 
is $A = A(\ell)$ such that no non-CM elliptic curve $E_{/F}$ admits an $F$-rational cyclic $\ell^A$-isogeny.   By a result of K. Arai \cite{Arai08}, there is $L = L(\ell,F)$ such that for every non-CM elliptic curve $E_{/F}$, the image 
of the $\ell$-adic Galois representation $\rho_{E,\ell^{\infty}}: \gg_F \ra \GL_2(\Z_{\ell})$ contains the kernel of 
the reduction map $\GL_2(\Z_{\ell}) \ra \GL_2(\Z/\ell^L \Z)$.  Thus $E$ admits no $F$-rational cyclic $\ell^{L+1}$-isogeny.   
\\
b) In fact, under the hypotheses of Arai's Theorem, there is such an $L$ that depends only on $\ell$ and $[F:\Q]$ \cite[Thm. 2.3a)]{CP18}, 
and the result follows.  This generalization of Arai's Theorem is essentially due to Cadoret and Tamagawa \cite{CT12} and a simpler 
proof had earlier been sketched by J. Rouse: cf. \cite[Rem. 2.4]{CP18}.  
\end{proof}

\begin{cor}
\label{LASTLVCOR}
Assume the Generalized Riemann Hypothesis (GRH).  For a number field $F$, the following are equivalent:
\begin{itemize}
\item[(i)] The set of $N \in \Z^+$ such that $Y_0(N)(F) \neq \varnothing$ is finite.
\item[(ii)] $F$ \emph{does not} contain the Hilbert class field of any imaginary quadratic field.
\end{itemize}
\end{cor}
\begin{proof}
By \cite[Cor. 6.5]{Larson-Vaintrob14}, there are only finitely many prime numbers $\ell$ such that there is an elliptic curve 
$E_{/F}$ that admits an $F$-rational $\ell$-isogeny.  The result follows from this, Lemma \ref{PRELVLEMMA} and Corollary \ref{LVCOMPCOR}a).
\end{proof}
\noindent
Larson-Vaintrob actually show that over any number field $F$, conditionally 
on GRH, for all sufficiently large primes $\ell$, if there is an $F$-rational $\ell$-isogeny of elliptic curves with 
isogeny character $\chi: \gg_F \ra (\Z/\ell \Z)^{\times}$, there is an $F$-rational $\ell$-isogeny of $K$-CM elliptic curves 
with isogeny character $\psi: \gg_F \ra (\Z/\ell \Z)^{\times}$ such that $F$ contains $K$ (and hence also $K(1)$) and 
$\chi^{12} = \psi^{12}$.  Since isogenies of CM elliptic curves are now well understood, in Corollary \ref{LASTLVCOR} it would be 
very desirable to remove the hypothesis that $F$ contains no Hilbert class field of an imaginary quadratic field and limit the 
conclusion to the the non-CM case.  Using Theorem \ref{INERTNESSTHM} and $\chi^{12} = \psi^{12}$, one can get an upper bound 
on the index of the image of $\chi$.  This is useful for certain applications to the torsion subgroup -- cf. \cite[Thm. 1.8]{CMP18}.  It is unfortunately not so clear how to use it to further study isogenies in the non-CM case.

\subsection{The Projective Torsion Field}
\noindent
Let $F$ be a subfield of $\C$, let $N \geq 2$, let $E_{/F}$ be an elliptic curve, and let 
\[ \rho_N: \gg_F \ra \GL_2(\Z/N\Z) \]
be its modulo $N$ Galois representation.  The center of $\GL_2(\Z/N\Z)$ consists of the subgroup of scalar matrices, isomorphic to 
$(\Z/N\Z)^{\times}$, and we define the \textbf{projective modulo N Galois representation} 
\[ \PP \rho_N: \gg_F \ra \GL_2(\Z/N\Z)/(\Z/N\Z)^{\times} \]
to be the composite of $\rho_N$ with the quotient map $q: \GL_2(\Z/N\Z) \ra \GL_2(\Z/N\Z)/(\Z/N\Z)^{\times}$.  One can view 
this as the action of $\gg_F$ on the lines (i.e., one-dimensional free $\Z/N\Z$-submodules) in $E[N](\overline{F})$.  The 
\textbf{projective $N$-torsion field of E} $F(\PP E[N])$ is the fixed field $\overline{F}^{\Ker \PP \rho_N}$.  It is a finite Galois 
extension of $F$, the unique minimal field extension of $F$ over which Galois acts on $E[N]$ by scalar matrices.  The latter interpretation 
implies that the projective torsion field is also the compositum of all fields of moduli $F(\varphi)$ where $\varphi: E \ra E'$ is a 
cyclic $N$-isogeny.  
\\ \\
Passing from $E_{/F}$ to a quadratic twist $E^D_{/F}$ does not change the projective Galois representation hence also does not 
change the projective torsion field.  This need not be the case for quartic twists when $j = 1728$ or sextic twists when $j = 0$, 
so from now until the end of this section we assume that $j(E) \notin \{0,1728\}$.  In this case the projective Galois representation 
depends only on the closed point $j(E) \in X(1)_{/F}$.  Determining the projective torsion field when $F = \Q(j(E))$ amounts to computing the fiber over $j(E)$ in the Galois covering of $F$-curves $X_0(N,N) \ra X(1)$: if $f = [F( \PP E[N]):F]$ then when $N \geq 3$, the fiber over $j(E)$ consists of 
$\frac{\# \GL_2(\Z/N\Z)}{f \frac{\varphi(N)}{2}}$ closed points, each with residue field $F( \PP E[N])$.  
\\ \\
We would like to compute the projective torsion field in the CM case: again, for now we suppose that $\Delta < -4$.  Thus we choose an 
elliptic curve $E_{/\Q(\ff)}$ with $j(E) = j_{\Delta}$, and for all $N \geq 2$ we wish to determine $\Q(\ff)( \PP E[N])$.  The determination 
of the compositum of this field with $K$ is a result of Parish.

\begin{thm}[Parish]
\label{PARISHTHM}
Let $\Delta < -4$.  Then for all $N \geq 2$, the projective torsion field of any $\Delta$-CM elliptic curve $E_{/K(\ff)}$ is the 
ring class field $K(N\ff)$.
\end{thm}
\noindent
As an application of the present techniques, we will prove the following result.

\begin{thm}
\label{BETTERPARISHTHM}
\label{THM6.10}
Let $\Delta = \ff^2 \Delta_K$ be an imaginary quadratic discriminant with $\Delta_K < -4$.  Let $N \geq 2$, let $F$ be a number field isomorphic to $\Q(\ff)$, and let $E_{/F}$ be a $\Delta$-CM elliptic curve (so $F = \Q(j(E))$).  We put 
\[ P(\Delta,N) \coloneqq F(\PP E[N]). \]
Then:
\begin{itemize}
\item[a)] If $N = 2$ and $\Delta$ is even, then $P(\Delta,2) \cong \Q(2\ff)$.
\item[b)] In all other cases we have $P(\Delta,N) = K(N\ff)$.
\end{itemize}
\end{thm}
\begin{proof}
By Theorem \ref{PARISHTHM} -- and also by Lemma \ref{LEMMA7.1} -- we have $P(\Delta,N) \subseteq K(N\ff)$.  On the other hand, 
there is a cyclic $N$-isogeny $\varphi: E \ra \tilde{E}$ defined over $\C$ such that $\tilde{E}$ is a complex $N^2 \Delta$-CM elliptic curve, so $P(\Delta,N)$ 
contains a subfield isomorphic to $\Q(N\ff)$.  Thus we must show: $P(\Delta,N)$ does not contain $K$ if and only if 
$N = 2$ and $\Delta$ is even. \\ 
$\bullet$ If $N =2$, then $P(\Delta,2) = F(E[2])$, and the result follows from \cite[Thm. 4.2]{BCS17}a).  It also follows from our analysis 
of complex conjugation on the isogeny volcanoes. \\ \indent
$\bullet$ If $N \geq 3$, then it follows from Corollary \ref{REALPROJCOR} that the projective $N$-torsion field cannot be embedded into $\R$, hence the projective $N$-torsion field cannot be $\Q(N\ff)$, so it must be $K(N\ff)$.  
\end{proof}
\noindent
In \cite[\S 6]{CS22b} we will present a generalization of Theorem \ref{BETTERPARISHTHM} that applies to 
any imaginary quadratic discriminant $\Delta$.

\section{CM Points on $X_0(\ell^a)_{/\Q}$}
Let $\ell$ be a prime number, and let $\Delta = \ell^{2A} \ff_0^2 \Delta_K$ be an imaginary quadratic discriminant such that 
$\ff_0^2 \Delta_K < -4$.  In this section we will compute the fiber of $X_0(\ell^a) \ra X(1)$ over $J_{\Delta}$: there is no ramification, 
so we determine which residue fields occur and with what multiplicity.  The residue field of a closed point on a finite-type $\Q$-scheme 
is a number field that is well-determined up to isomorphism; it is not well-defined as a subfield of $\C$.  Thus when we write that 
the residue field is $\Q(\ff)$ for some $\ff \in \Z^+$, we mean that it is isomorphic to this field.  
\\ \\
As described above, without loss of generality we may take our source elliptic curve to have $j$-invariant $j_{\Delta}$ and then our task is to: \\ \\
(i) Enumerate all nonbacktracking length $a$ paths $P$ in $\mathcal{G}_{K,\ell,\ff_0}$. \\
(ii) Sort them into closed point equivalence classes $\mathcal{C}(P)$ as in \S 5.2 and record the field of moduli for each equivalence 
class (again, here we record any number field isomorphic to $\Q(f)$ as $\Q(f)$).  \\
(iii) Record the number of closed point equivalence classes that give rise to each field of moduli.  
\\ \\
One check on the accuracy of the calculation is as follows: let $\psi: \Z^+ \ra \Z^+$ be the multiplicative function such that 
for any prime power $\ell^a$ we have $\psi(\ell^a) = \ell^{a-1}(\ell+1)$.  For all $N \in \Z^+$, we have 
\[ \deg(X_0(N) \ra X(1)) = \psi(N). \]
Since the map $X_0(\ell^a) \ra X(1)$ is unramified over $J_{\Delta}$, we must have 
\[ \sum_{C(\varphi)} d_{\varphi} = \psi(\ell^a) = \ell^a + \ell^{a-1}, \]
where the sum extends over closed point equivalence classes.
\\ \\
As above, we split the path into the concatentation $P_1 \cup P_2$, where $P_1$ is the maximal initial segment of $P$ that 
terminates at a vertex of level $K$ and $P_2$ is the remainder of the path, which consists of $\max(L'-L,0)$ descending edges.  
Let $a_1$ be the length of $P_1$ and $a_2$ be the length of $P_2$.

\subsection{Path type analysis I}
We consider paths of length $a$ in $\mathcal{G}_{K,\ell,\ff_0}$ beginning at $j_{\Delta}$ of level 
$L \geq 0$.  Every such path consists of $b \geq 0$ ascending edges followed by $h \geq 0$ horizontal edges followed by 
$d \geq 0$ descending edges.  For each $L$, value of $\left( \frac{\Delta_K}{\ell} \right)$ and value of $\ord_2(\Delta_K)$ 
when $\ell = 2$, we list all possible 
triples $(b,h,d)$ that occur and for each triple, list the number of closed point equivalence classes and their residue fields.  
\\ \\
For some of the types the classification differs when $\ell = 2$.  We begin in this section with the portion of the analysis that holds 
for all primes $\ell$.  
\\ \\
\textbf{I.} There is always a unique closed point equivalence class $[P_{\downarrow}]$ of paths $(b,h,d) = (0,0,a)$ -- i.e., consisting of $a$ descending edges.  The residue field of this path is $\Q(\ell^a \ff)$.  
\\
\textbf{II.} If $a \leq L$ then there is a unique path $P_{\uparrow}$ with $(b,h,d) = (a,0,0)$ -- i.e., consisting of $a$ ascending edges.  The residue 
field of this path is $\Q(\ff)$. \\
\textbf{III.} In the case $L = 0$ and $\left( \frac{\Delta_K}{\ell} \right) = 0$ there are paths with $(b,h,d) = (0,1,a-1)$. Such paths form one closed point equivalence class, with residue field $\Q(\ell^{a-1} \ff)$.  \\ 
\textbf{IV.}  In the case $L = 0$ and $\left( \frac{ \Delta_K}{\ell} \right) = 1$, for each $1 \leq h \leq a$ there is one complex conjugate pair of paths with $(b,h,d) = (0,h,a-h)$ and residue field $K(\ell^{a-h} \ff)$.  
\\ 
\textbf{X.} If $L \geq 1$ and $a-L \geq 1$ and $\left( \frac{\Delta_K}{\ell} \right) = 1$, there is one closed point equivalence class 
of paths with $(b,h,d) = (L,a-L,0)$ and residue field $K(\ff)$.

\subsection{Path type analysis II: $\ell > 2$} \textbf{} \\ \\ \noindent
\textbf{V.} If $L \geq 2$, for all $1 \leq b \leq \min(a-1,L-1)$, there are paths with $(b,h,d) = (b,0,a-b)$.  They fall into 
$\frac{\ell-1}{2} \ell^{\min(b,a-b)-1}$ closed point equivalence classes, each with residue field $K(\ell^{\max(a-2b,0)} \ff)$.  
\\ 
\textbf{VI.} If $a > L \geq 1$ and $\left( \frac{\Delta_K}{\ell} \right) = -1$ there are paths with $(b,h,d) = (L,0,a-L)$.  There is one closed point equivalence class of paths with residue field $\Q(\ell^{\max(a-2L,0)} \ff)$.  There are $\frac{\ell^{\min(L,a-L)}-1}{2}$ closed point equivalence classes of paths with residue field $K(\ell^{\max(a-2L,0)} \ff)$. 
\\ 
\textbf{VII.} If $a \geq L+1 \geq 2$ and $\left( \frac{\Delta_K}{\ell} \right) = 0$, there are paths with $(b,h,d) = (L,0,a-L)$.  There are 
$\frac{\ell-1}{2} \ell^{\min(L,a-L)-1}$ closed point equivalence classes, each with residue field $K(\ell^{\max(a-2L,0)} \ff)$.
\\ 
\textbf{VIII.} If $a \geq L+1 \geq 2$ and $\left( \frac{\Delta_K}{\ell} \right) = 0$, there are paths with $(b,h,d) = (L,1,a-L-1)$.  One of them 
is real and has residue field $\Q(\ell^{\max(a-2L-1,0)} \ff)$.  There are $\frac{\ell^{\min(L,a-L-1)}-1}{2}$ closed point equivalence classes of paths with residue field $K(\ell^{\max(a-2L-1,0)} \ff)$.\footnote{Thus there are no such classes if and only if $a = L+1$.}
\\
\textbf{IX.} If $a \geq L+1 \geq 2$ and $\left( \frac{\Delta_K}{\ell} \right) = 1$, there are paths with $(b,h,d) = (L,0,a-L)$.  There is one closed 
point equivalence class of such paths with residue field $\Q(\ell^{\max(a-2L,0)} \ff)$.  There are $\frac{(\ell-2)\ell^{\min(L,a-L)-1}-1}{2}$ 
closed point equivalence classes\footnote{Thus when $\ell = 3$ and $\min(L,a-L) = 1$ there are no paths of this kind; otherwise there are.} of such paths with residue field $K(\ell^{\max(a-2L,0)} \ff)$.
\\ 
\textbf{XI.}  If $L \geq 1$, $a-L \geq 2$ and $\left( \frac{\Delta_K}{\ell} \right) = 1$, then for all $1 \leq h \leq a-L-1$ there are \\ $(\ell-1)\ell^{\min(L,a-L-h)-1}$ closed point equivalence classes of paths with $(b,h,d) = (L,h,a-L-h)$ and residue field $K(\ell^{\max(a-2L-h,0)} \ff)$. 

\subsection{Path type analysis III: $\ell = 2$, $\left( \frac{\Delta_K}{2} \right) \neq 0$}
\textbf{} \\ \\ \noindent
In this case, \textbf{Type V.} consists of paths that ascend but not all the way to the surface and then descend at least once.  There are several cases.
\\ \\
V$_1$. If $L \geq 2$, then for all $a \geq 2$ there is one closed point equivalence class of paths with $(b,h,d) = (1,0,a-1)$.  The residue field is $\Q(2^{a-2} \ff)$.
\\
V$_2$. If $L \geq a \geq 3$, there is one closed point equivalence class of paths with $(b,h,d) = (a-1,0,1)$.  The residue field is $\Q(\ff)$.  
\\ 
V$_3$. If $a > L \geq 3$, there are paths with $(b,h,d) = (L-1,0,a-L+1)$.  There are two closed point equivalence classes of such paths with residue field $\Q(2^{\max(a-2L+2,0)} \ff)$ and $2^{\min(a-L+1,L-1)-2}-1$ closed point equivalence classes of such paths with residue 
field $K(2^{\max(a-2L+2,0)} \ff)$.
\\ 
V$_4$.  If $2 \leq b \leq \min(L-2,a-2)$, there are paths with $(b,h,d) = (b,0,a-b)$.  There are $2^{\min(b,a-b)-2}$ closed point equivalence classes of such paths with residue field $K(2^{\max(a-2b,0)} \ff)$.
\\ 
VI.  If $a > L \geq 1$ and $\left( \frac{\Delta_K}{2} \right) = -1$, there are paths with $(b,h,d) = (L,0,a-L)$.  There are 
$2^{\min(L,a-L)-1}$ closed point equivalence classes of such paths with residue field $K(2^{\max(a-2L,0)} \ff)$.
\\ 
(We omit IX. because when $\ell = 2$ and $\left( \frac{\Delta_K}{2} \right) = 1$ there are no paths that ascend to the surface and then descend back down immediately.)  
\\ 
XI. If $a-L \geq 2$ and $\left( \frac{\Delta_K}{2} \right) = 1$, then for all $1 \leq h \leq a-L-1$ there are paths with $(b,h,d) = (L,h,a-L-h)$.  There are $2^{\min(L,a-L-h)-1}$ 
such paths with residue field $K(2^{\max(a-2L-h,0)} \ff)$.  

\subsection{Path type analysis IV: $\ell = 2$, $\ord_2(\Delta_K) = 2$}
\textbf{} \\ \\ \noindent
In this case, \textbf{Type V.} consists of paths that ascend but not all the way to the surface and then descend at least once. 
\\ 
V$_1$. If $L \geq 2$, then for all $a \geq 2$ there is one closed point equivalence class of paths with $(b,h,d) = (1,0,a-1)$.  The residue 
field is $\Q(2^{a-2} \ff)$.
\\ 
V$_2$. If $L \geq a \geq 3$, there is one closed point equivalence class of paths with $(b,h,d) = (a-1,0,1)$.  The residue field is $\Q(\ff)$.
\\ 
V$_3$. If $2 \leq b \leq \min(L-1,a-2)$, there are paths with $(b,h,d) = (b,0,a-b)$.  They fall into $2^{\min(b,a-b)-2}$ closed 
point equivalence classes, each with residue field $K(2^{\max(a-2b,0)} \ff)$.
\\ \\
In this case, \textbf{Type VI.} consists of paths that ascend to the surface and then immediately descend at least once: 
$(b,h,d) = (L,0,a-L)$.
\\ 
VI$_1$.  If $L = 1$, then there is one closed point equivalence class of such paths, with residue field $\Q(2^{a-2} \ff)$.  
\\ 
VI$_2$.  If $a = L+1 \geq 3$, there is one closed point of equivalence classes of such paths, with residue field $\Q(\ff)$.  
\\ 
VI$_3$. If $a \geq L+2 \geq 4$, then there are two closed point equivalence classes of such paths with residue field 
$\Q(2^{\max(a-2L,0)} \ff)$ and $2^{\min(L,a-L)-2}-1$ closed point equivalence classes of such paths with residue field 
$K(2^{\max(a-2L,0)} \ff)$.\footnote{This means that there are no paths of the latter type if and only if $a = L+2$.}
\\  \\
In this case, \textbf{Type VIII.} consists of paths that ascend to the surface, then take the unique surface edge.  The remaining portion 
of such a path, if any, must be a descent.  Thus we have $a \geq L+1 \geq 2$ and $(b,h,d) = (L,1,a-L-1)$.  
\\ 
VIII$_1$.  If $a = L+1$, there is one closed point equivalence class of such paths, with residue field $\Q(\ff)$.
\\ 
VIII$_2$. If $a \geq L+2$, there are $2^{\min(L,a-1-L)-1}$ closed point equivalence classes of such paths, each with residue field 
$K(2^{\max(a-2L-1,0)} \ff)$.  

\subsection{Path type analysis V: $\ell = 2$, $\ord_2(\Delta_K) = 3$} In this case the isogeny graph has the same structure as in the previous section, but with a different action of complex conjugation.  Thus the classification of paths is the samfe as before, but with sometimes different results on which of them are real.
\\ \\
As above \textbf{Type V.} consists of paths ascending but not to the surface and then descend at least once: we need $L \geq 2$.
\\
V$_1$. If $L \geq 2$, then for all $a \geq 2$ there is one closed point equivalence class of paths with $(b,h,d) = (1,0,a-1)$.  The residue 
field is $\Q(2^{a-2} \ff)$.
\\ 
V$_2$. If $L \geq a \geq 3$, there is one closed point equivalence class of paths with $(b,h,d) = (a-1,0,1)$.  The residue field is $\Q(\ff)$.
\\ 
V$_3$. If $2 \leq b \leq \min(L-1,a-2)$, there are paths with $(b,h,d) = (b,0,a-b)$.  They fall into $2^{\min(b,a-b)-2}$ closed 
point equivalence classes, each with residue field $K(2^{\max(a-2b,0)} \ff)$.
\\ \\
As above \textbf{Type VI.} consists of paths that ascend to the surface and then immediately descend at least once: $(b,h,d) = (L,0,a-L)$.
\\ 
VI$_1$.  If $L = 1$, then there is one closed point equivalence class of such paths, with residue field $\Q(2^{a-2} \ff)$.  
\\ 
VI$_2$.  If $a = L+1 \geq 3$, there is one closed point of equivalence classes of such paths, with residue field $\Q(\ff)$.  
\\ 
VI$_3$. If $a \geq L+2 \geq 4$, then there are $2^{\min(L,a-L)-2}$ closed point equivalence classes of such paths, with residue 
field $K(2^{\max(a-2L,0)} \ff)$.  
\\ \\
\textbf{Type VIII.} as above consists of paths that ascend to the surface, then take the unique surface edge.  The remaining portion 
of such a path, if any, must be a descent.  Thus we have $a \geq L+1 \geq 2$ and $(b,h,d) = (L,1,a-L-1)$.  
\\ 
VIII$_1$.  If $a = L+1$, there is one closed point equivalence class of such paths, with residue field $\Q(\ff)$.
\\ 
VIII$_2$. If $a \geq L+2$, there are $2$ closed point equivalence classes of such paths with residue field $\Q(2^{\max(a-2L-1,0)} \ff)$ 
and $2^{\min(L,a-1-L)-1}-1$ closed point equivalence classes of such classes with residue field $K(2^{\max(a-2L-1,0)} \ff)$.

\section{Primitive Residue Fields of CM points on $X_0(\ell^{a'},\ell^a)_{/\Q}$}
\noindent
Let $\Delta = \ff^2 \Delta_K$ be an imaginary quadratic discriminant with $\Delta_K < -4$.  The main goal of this work is the determination of the $\Delta$-CM locus on the curves $X_0(M,N)_{/\Q}$ for $M \mid N$.  In the case of $X_0(\ell^a)$ for a prime power $\ell^a$ 
we record the residue field of every closed $\Delta$-CM point, with multiplicities, in the Appendix of this version of the paper: there are 95 cases.  The work of the present paper allows for a complete enumeration in the $X_0(M,N)_{/\Q}$ case, but the answer would be quite lengthy and the process of recording it would be quite tedious.  
\\ \\
It turns out that for most purposes (including \emph{every} application in the present paper) it is just as useful to record primitive residue fields.  If $X(H)_{/\Q}$ is a modular curve, then we say a closed $\Delta$-CM point $P \in X(H)$ is \textbf{primitive} if there is no other $\Delta$-CM point $Q \in X(H)$ together with an embedding of the residue field $\Q(Q)$ into the residue field of $\Q(P)$ as a proper subfield.  This is a natural definition because for any number field $F$, there is an $F$-rational $\Delta$-CM point on $X(H)$ if and only if some residue field of a primitive $\Delta$-CM point embeds as a subfield of $F$.  
\\ \\
In this section we record all primitive residue fields of $\Delta$-CM points on $X_0(\ell^{a'},\ell^a)_{/\Q}$, where $\ell$ is a prime number 
and $0 \leq a' \leq a$.
\\ \\
Throughout we put 
\[L \coloneqq \ord_{\ell}(\ff). \]

\subsection{$X_0(\ell^a)$}
First we address the case of $a' = 0$, i.e., of $X_0(\ell^a)$.   The results could be deduced from the tables of the Appendix, but it is quicker to derive them directly using the techniques of \S 7. 
\\ \\
Here is the complete description of primitive residue fields of $\Delta$-CM points on $X_0(\ell^a)$:
\\ \\
\textbf{Case 1.1}: Supose $\ell^a = 2$. \\
\textbf{Case 1.1a}: Suppose  $\left(\frac{\Delta}{2} \right) \neq -1$.  The primitive residue field is $\Q(\ff)$ (which equals $\Q(2\ff)$ 
when $\left( \frac{\Delta}{2} \right)  = 1$).  \\
\textbf{Case 1.1b}: Suppose $\left( \frac{\Delta}{2} \right) = -1$.  The primitive residue field is $\Q(2\ff)$.  
\\  
\textbf{Case 1.2}: Suppose $\ell^a > 2$ and $\left( \frac{\Delta}{\ell} \right) = 1$.  The primitive residue fields 
are $\Q(\ell^a \ff)$ and $K(\ff)$.  
\\ 
\textbf{Case 1.3}: Suppose $\ell^a > 2$ and $\left( \frac{\Delta}{\ell} \right) = -1$.  The primitive residue field is $\Q(\ell^a \ff)$.  
\\ 
\textbf{Case 1.4}: Suppose $\ell^a > 2$, $\left(\frac{\Delta}{\ell}\right) = 0$ and $L = 0$.  The primitive residue 
field is $\Q(\ell^{a-1} \ff)$.  
\\  
\textbf{Case 1.5}: Suppose $\ell > 2$, $L \geq 1$ and $\left( \frac{\Delta_K}{\ell} \right) = 1$.  \\
\textbf{Case 1.5a}: Suppose $a \leq 2L$.  In this case there is a $\Q(\ff)$-rational cyclic $\ell^a$-isogeny, so the only primitive 
residue field is $\Q(\ff)$.  
\\
\textbf{Case 1.5b}: Suppose $a > 2L$.  Then the primitive residue fields are $\Q(\ell^{a-2L} \ff)$ and $K(\ff)$.  
\\ 
\textbf{Case 1.6}: Suppose $\ell > 2$, $L \geq 1$, and $\left( \frac{\Delta_K}{\ell} \right) = -1$. \\
\textbf{Case 1.6a}: Suppose $a \leq 2L$.  As in Case 1.5a, there is a $\Q(\ff)$-rational cyclic $\ell^a$-isogeny, so the only primitive 
residue field is $\Q(\ff)$.  
\\
\textbf{Case 1.6b}: Suppose $a > 2L$.  In this case the primitive residue field is $\Q(\ell^{a-2L} \ff)$.  
\\ 
\textbf{Case 1.7}: Suppose $\ell > 2$, $L \geq 1$, $\left( \frac{\Delta_K}{\ell} \right) = 0$. \\
\textbf{Case 1.7a}: Suppose $a \leq 2L+1$. As in Case 1.5a, there is a $\Q(\ff)$-rational cyclic $\ell^a$-isogeny, so the only primitive 
residue field is $\Q(\ff)$.  \\
\textbf{Case 1.7b}: Suppose $a \geq 2L+2$.  In this case the primitive residue field is $\Q(\ell^{a-2L-1} \ff)$.  
\\  
\textbf{Case 1.8}: Suppose $\ell = 2$, $a \geq 2$, $L \geq 1$, and $\left( \frac{\Delta_K}{2} \right) = 1$. \\
\textbf{Case 1.8a}: Suppose $L = 1$.  The primitive residue fields are $\Q(2^a \ff)$ and $K(\ff)$.  
\\
\textbf{Case 1.8b}: Suppose $L \geq 2$ and $a \leq 2L-2$.  The primitive residue field is $\Q(\ff)$.  
\\
\textbf{Case 1.8c}: Suppose $L \geq 2$ and $a \geq 2L-1$.  The primitive residue fields are $\Q(2^{a-2L+2} \ff)$ and $K(\ff)$.  
\\  
\textbf{Case 1.9}: Suppose $\ell = 2$, $a \geq 2$, $L \geq 1$, and $\left( \frac{\Delta_K}{2} \right) = -1$. \\
\textbf{Case 1.9a}: Suppose $L = 1$.  The primitive residue fields are $\Q(2^a \ff)$ and $K(2^{a-2} \ff)$.  
\\
\textbf{Case 1.9b}: Suppose $L \geq 2$ and $a \leq 2L-2$.  The primitive residue field is $\Q(\ff)$.  
\\
\textbf{Case 1.9c}: Suppose $L \geq 2$ and $a \geq 2L-1$.  The primitive residue fields are $\Q(2^{a-2L+2} \ff)$ and $K(2^{\max(a-2L,0)}\ff)$.  
\\ 
\textbf{Case 1.10}: Suppose $\ell = 2$, $a \geq 2$, $L \geq 1$, $\left( \frac{\Delta_K}{2} \right) = 0$, and $\ord_2(\Delta_K) = 2$.  
\\
\textbf{Case 1.10a}: Suppose $a \leq 2L$.  The primitive residue field is $\Q(\ff)$.  
\\
\textbf{Case 1.10b}: Suppose $a \geq 2L+1$.  The primitive residue fields are $\Q(2^{a-2L} \ff)$ and $K(2^{a-2L-1} \ff)$.  
\\  
\textbf{Case 1.11}: Suppose $\ell = 2$, $a \geq 2$, $L \geq 1$,  $\left( \frac{\Delta_K}{2} \right) = 0$, and $\ord_2(\Delta_K) = 3$.  
\\
\textbf{Case 1.11a}: Suppose $a \leq 2L+1$.  The primitive residue field is $\Q(\ff)$.  
\\
\textbf{Case 1.11b}: Suppose $a \geq 2L+2$.  The primitive residue field is $\Q(2^{a-2L-1} \ff)$.  

\subsection{$X_0(2,2^a)$, $\Delta$ even}
Next we suppose that $a' \geq 1$.  Every closed $\Delta$-CM point $\tilde{P} \in X_0(\ell^{a'},\ell^a)$ is induced by a $\Delta$-CM 
elliptic curve $E_{/\Q(\tilde{P})}$ that admits a $\Q(\tilde{P})$-rational cyclic $\ell^a$-isogeny and also has Galois acting on $E[\ell^{a'}]$ by scalar matrices, and (since $\Delta < -4$) the point $\tilde{P}$ determines $E$ up to quadratic twist.  When
$\ell^{a'} = \ell^a = 2$ there are two other ways to describe the same level structure on $E_{/\Q(\tilde{P})}$: first that $E$ has 
full $2$-torsion defined over $\Q(\tilde{P})$ and second that $E$ has two independent $2$-isogenies defined over $\Q(\tilde{P})$.  
\\ \\
Let 
\[ \pi: X_0(\ell^{a'},\ell^a) \ra X_0(\ell^a) \]
be the natural map, a $\Q$-morphism of degree $(\ell-1)\ell^{2a'-1}$, and let $P \coloneqq \pi(\tilde{P})$.  Then
\[ \Q(\tilde{P}) = \Q(P)(\mathbb{P} E[\ell^{a'}]). \]
By adjusting $E_{/\Q(\tilde{P})}$ by a quadratic twist, we may assume that $E$ arises by base change from an elliptic curve 
defined over $\Q(j(E))$, putting us in a position to apply Theorem \ref{BETTERPARISHTHM}.  
\\ \\
Suppose first that $\ell^{a'} = 2$ and $\Delta$ is even.  Then by Theorem \ref{BETTERPARISHTHM}a) and the above discussion, 
we have that $\Q(\tilde{P})$ is the compositum of $\Q(P)$ and a number field isomorphic to the rational ring class field $\Q(2\ff)$.  We may assume without loss of generality that $j(E) = j_{\Delta} \coloneqq j(C/\OO(\Delta))$.  In this case, the projective torsion 
field $\Q(\ff)(E[2])$ is a quadratic extension of $\Q(\ff)$, hence Galois.  Indeed it is equal to $\Q(2\ff) = \Q(j_{4\Delta})$ and not 
just isomorphic to it, e.g. because this field is also the field of moduli of the isogeny $\C/\OO(4\Delta) \ra \C/\OO(\Delta)$.  
\\ \\
This observation allows us to easily write down primitive residue fields of $\Delta$-CM points on $X_0(2,2^a)$ in terms of primitive 
residue fields of $\Delta$-CM points on $X_0(2^a)$.  
\\ \\
$\bullet$ If the primitive residue field of a $\Delta$-CM point on $X_0(2^a)$ is $\Q(2^c \ff)$ for some $c \geq 0$ (depending on $\Delta$ and $a$), then the primitive residue field of a $\Delta$-CM point on $X_0(2^a)$ is 
\[ \Q(2^c \ff)\Q(2\ff) = \begin{cases} \Q(2\ff) & \text{if } c = 0 \\ \Q(2^c \ff) & \text{if } c \geq 1 \end{cases}. \]
$\bullet$ If the primitive residue fields of a $\Delta$-CM point on $X_0(2^a)$ are $\Q(2^b \ff)$ and $K(2^c \ff)$ with $0 \leq c < b$, 
then the primitive residue fields of $\Delta$-CM points on $X_0(2^a)$ lie among $\Q(2^{\max(b,1)} \ff)$ and $K(2^{\max(b,1)} \ff)$: 
both occur if $b > 1$; if $b \leq 1$, then $\Q(2\ff)$ is the primitive residue field.  
\\ \\
Here is the complete description of primitive residue fields of $\Delta$-CM points on $X_0(2,2^a)$:
\\ \\
\textbf{Case 2.0:} $a = 1$.  The primitive residue field is $\Q(2\ff)$.
\\ \\
Throughout the other cases we shall assume $a \geq 2$.  
\\ \\
\textbf{Case 2.1:} $L = 0$ and $\ord_2(\Delta_K) = 2$.  The primitive residue fields are $\Q(2^a \ff)$ and $K(2^{a-1} \ff)$.
\\ 
\textbf{Case 2.2:} $L = 0$ and $\ord_2(\Delta_K) = 3$.  The primitive residue field is $\Q(2^{a-1} \ff)$.
\\ 
\textbf{Case 2.3:} $\left( \frac{\Delta_K}{2} \right) = 1$ and $L = 1$.  The primitive residue fields are $\Q(2^a \ff)$ and $K(2\ff)$.
\\ 
\textbf{Case 2.4:} $\left( \frac{\Delta_K}{2} \right) = 1$, $L \geq 2$ and $a \leq 2L-1$.  The primitive residue field is $\Q(2\ff)$.
\\ 
\textbf{Case 2.5:} $\left( \frac{\Delta_K}{2} \right) = 1$, $L \geq 2$ and $a \geq 2L$.  The primitive residue fields 
are $\Q(2^{a-2L+2} \ff)$ and $K(2\ff)$. 
\\ 
\textbf{Case 2.6:} $\left( \frac{\Delta_K}{2} \right) = -1$, $L = 1$ and $a = 2$.  The primitive residue fields are $\Q(2^2 \ff)$ 
and $K(2 \ff)$.
\\ 
\textbf{Case 2.7:} $\left( \frac{\Delta_K}{2} \right) = -1$, $L = 1$ and $a \geq 3$.  The primitive residue 
fields are $\Q(2^a \ff)$ and $K(2^{a-2} \ff)$.
\\ 
\textbf{Case 2.8}: $\left( \frac{\Delta_K}{2} \right) = -1$, $L \geq 2$ and $a \leq 2L-1$.  The primitive residue field is $\Q(2\ff)$.
\\ 
\textbf{Case 2.9}: $\left( \frac{\Delta_K}{2} \right) = -1$, $L \geq 2$ and $a = 2L$.  The primitive residue fields are 
$\Q(2^2 \ff)$ and $K(2\ff)$.
\\ 
\textbf{Case 2.10}: $\left( \frac{\Delta_K}{2} \right) = -1$, $L \geq 2$ and $a \geq 2L+1$.  The primitive residue fields are 
$\Q(2^{a-2L+2} \ff)$ and $K(2^{a-2L} \ff)$.
\\ 
\textbf{Case 2.11}: $\ord_2(\Delta_K) = 2$, $L \geq 1$ and $a \leq 2L+1$.  The primitive residue field is $\Q(2\ff)$.
\\ 
\textbf{Case 2.12}: $\ord_2(\Delta_K) = 2$, $L \geq 1$ and $a \geq 2L+2$.  The primitive residue fields are
$\Q(2^{a-2L} \ff)$ and $K(2^{a-2L-1} \ff)$.  
\\
\textbf{Case 2.13}: $\ord_2(\Delta_K) = 3$, $L \geq 1$ and $a \leq 2L+1$.  The primitive residue field is $\Q(2\ff)$.
\\ 
\textbf{Case 2.14}: $\ord_2(\Delta_K) = 3$, $L \geq 1$ and $a \geq 2L+2$.  The primitive residue field is $\Q(2^{a-2L-1}\ff)$.  

\subsection{$X_0(2,2^a)$, $\Delta$ odd}
Suppose that $\ell^{a'} > 2$ or $\Delta$ is odd.  Then Theorem \ref{THM6.10} shows that 
\[ \Q(\tilde{P}) = \Q(P)K(\ell^{a'} \ff). \]
As we have seen in \S 8.2, the possibilities for $\Q(P)$ are either:
\\ 
(I.) $\Q(\ell^b \ff)$ for some $b \geq 0$, or \\
(II.) $\Q(\ell^b \ff)$ and $K(\ell^c \ff)$ for some $0 \leq c < b$.  
\\ \\
It follows that there is always a unique primitive residue field: 
\\ \\
In Case (I.), the primitive residue field is $K(\ell^{\max(a',b)} \ff)$. \\
In Case (II.), the primitive residue field is $K(\ell^{\max(a',c)} \ff)$.  
\\ \\
Here is the complete description of primitive residue fields of $\Delta$-CM points on $X_0(2,2^a)$ when $\Delta$ is odd:
\\ \\
\textbf{Case 3.1}: If $a = 1$, the primitive residue field is $K(2\ff)$. \\
\textbf{Case 3.2}: If $a \geq 2$ and $\left( \frac{\Delta}{2} \right) = 1$, the primitive residue field is $K(2\ff) = K(\ff)$.  \\
\textbf{Case 3.3}: If $a \geq 2$ and $\left( \frac{\Delta}{2} \right) = -1$, the primitive residue field is $K(2^a \ff)$.  

\begin{remark}
\label{NAJMANREMARK}
F. Najman pointed out to me that for all $N \in \Z^+$, the modular curve $X_0(4N)_{/\Q}$ is isomorphic to the modular curve $X_0(2,2N)$.  This comes 
from an isomorphism of moduli problems:
\[ [E,C_{4N}] \mapsto [E/C_2,C_4/C_2,C_{4N}/C_2], \]
where $C_{4N}$ is a cyclic subgroup of $E$ of order $4N$ and for $n \mid 4N$, we denote by $C_n$ the unique order $n$ cyclic subgroup of $C_{4N}$.\footnote{This isomorphism is somewhat curious, as the function fields 
$\Q(X_0(2,2N))$ and $\Q(X_0(4N))$ have the same degree over $\Q(X(1))$ but are not conjugate as subfields of $\Q(X(4N))$.}  This provides an alternate approach to CM points on $X_0(2,2^a)$.  
\end{remark}

\subsection{$X_0(\ell^{a'},\ell^a)$ with $\ell^{a'} \geq 3$}
Here is the complete description of primitive residue fields of $\Delta$-CM points on $X_0(\ell^{a'},\ell^a)$ when $\ell^{a'} \geq 3$:
\\ \\
\textbf{Case 4.1}: If $\left( \frac{\Delta_K}{\ell} \right) = 1$, the primitive residue field is $K(\ell^{a'} \ff)$.  \\
\textbf{Case 4.2}: If $\left( \frac{\Delta_K}{\ell} \right) = -1$, the primitive residue field is $K(\ell^{\max(a',a-2L)} \ff)$.  \\
\textbf{Case 4.3}: If $\left( \frac{\Delta_K}{\ell} \right) = 0$, the primitive residue field is $K(\ell^{\max(a',a-2L-1)} \ff)$.

\subsection{The Dreaded Case 1.5b)} Let us look back at Case 1.5b): we have a prime power $\ell^a$ with $\ell > 2$, $\left( \frac{\Delta_K}{\ell} \right) = 1$, $L \geq 1$ and 
$a > 2L$.  In this case we have not just two primitive residue fields of $\Delta$-CM points on $X_0(\ell^a)$ but also two primitive degrees: $2$ and $\ell^{a-2L}$.  
\\ \\
Suppose $M \mid N$ are positive integers and $\Delta < -4$.   By Theorem \ref{INERTNESSTHM} the multiset of degrees of closed $\Delta$-CM points on $X_1(M,N)$ 
is obtained from the multiset of degrees of closed $\Delta$-CM points on $X_0(M,N)$ by multiplying by $\max(\frac{\varphi(N)}{2},1)$.   It follows that there is more than one primitive degree of a $\Delta$-CM closed point on 
$X_0(M,N)$ if and only if there is more than one primitive degree of a $\Delta$-CM closed point on $X_1(M,N)$, which by the work of 
\cite{BCII} can only happen if $M \leq 2$.  \\ \indent
Coming back to the case $M = 1$, $N = \ell^a$ with $\ell > 2$, $\left(\frac{\Delta_K}{\ell} \right) = 1$, $L \geq 1$ and $a > 2L$, we find that the primitive degrees of closed $\Delta$-CM points on $X_1(\ell^a)$ are
\[ \frac{\varphi(\ell^a)}{2} \cdot \ell^{a-2L} = \frac{\ell^{2a-2L-1}(\ell-1)}{2} \text{ and } \frac{\varphi(\ell^a)}{2} \cdot 2 = \varphi(\ell^a) = \ell^{a-1}(\ell-1). \] 
The existence of multiple primitive degrees in this case was previously shown in \cite[Example 6.7]{BCII}.  In that work the least 
degree of a $\Delta$-CM point on $X_1(M,N)$ was determined (for all $\Delta$ and $M \mid N$); by \cite[Thm. 6.6]{BCII} the 
least degree in this case is indeed $d_c = \varphi(\ell^a) = \ell^{a-1}(\ell-1)$.  In \cite[Example 6.7]{BCII} the other primitive degree $d_b$ was not determined precisely (nor was it shown there was exactly one other primitive degree), but rather a Galois-theoretic argument was used to show that $\ord_2(d_b) < \ord_2(d_c)$.  We see now that $d_b = \frac{\ell^{2a-2L-1}(\ell-1)}{2}$ and 
thus $\ord_2(d_b) = \ord_2(d_c) - 1$.

\section{CM points on $X_0(M,N)_{/\Q}$}
\noindent
Throughout this section $\Delta = \ff^2 \Delta_K$ with $\Delta_K < -4$ and $M \mid N$ are positive integers.  Here we put the pieces together to determine the $\Delta$-CM locus on $X_0(M,N)_{/\Q}$.  
\\ \\
In \S 9.1 we explain how the compiling across prime powers process works in general.  Just as in the previous section, we do not attempt to exhaustively record the answers.  In \S 9.2 give an explicit (albeit somewhat complicated) description of all primitive residue fields and primitive degrees.

\subsection{Compiling Across Prime Powers}
For a prime number $\ell$ and integers $0 \leq a' \leq a$, the fiber $F$ of 
$X_0(\ell^{a'},\ell^a) \ra X(1)$ over the closed point $J_{\Delta}$ is a finite \'etale $\Q(J_{\Delta})$-scheme: that is, it is isomorphic to a product of finite degree field extensions of $\Q(\ff)$.  More precisely, there are non-negative integers $b_0,\ldots,b_a,c_1,\ldots,c_{a}$ 
such that $F \cong \Spec A$, where 
\begin{equation}
\label{COMPILEQ1}
 A = \prod_{j=0}^a \Q(\ell^j \ff)^{b_j} \times \prod_{k=0}^{a} K(\ell^k \ff)^{c_k}.
\end{equation}
When $a' = 0$, the explicit values of the $b_j$'s and $c_k$'s are given by the tables in the Appendix.  On the other hand, 
when either $\ell^{a'} \geq 3$ or ($\ell^{a'} = 2$ and $\Delta$ is odd), by Theorem \ref{BETTERPARISHTHM} we have $b_j = 0$ for all $0 \leq j \leq a$.  
\\ \\
We now explain how the previous results allow us to compute the fiber $F = \Spec A$ of $X_0(M,N) \ra X(1)$ over $J_{\Delta}$ for any positive integers $M \mid N$, where $M = \ell_1^{a_1'} \cdots \ell_r^{a_r'}$ and $N = \ell_1^{a_1} \cdots \ell_r^{a_r}$.  For $1 \leq i \leq r$, let $F_i \cong \Spec A_i$ be the fiber of $X_0(\ell_i^{a_i'},\ell_i^{a_i}) \ra X(1)$ over $J_{\Delta}$.  By Proposition \ref{PROP3.7} we have 
\begin{equation}
\label{COMPILEQ2}
 A \cong A_1 \otimes_{\Q(J_{\Delta})} \cdots \otimes_{\Q(J_{\Delta})} A_r. 
\end{equation}
It follows that $A$ is isomorphic to a direct sum of terms of the form 
\[ B \coloneqq B_1 \otimes_{\Q(\ff)} \cdots \otimes_{\Q(\ff)} B_r, \]
where for $1 \leq i \leq r$ we have that $B_i$ is isomorphic to either $\Q(\ell_i^{j_i} \ff)$ for some $0 \leq j_i \leq a$ or to $K(\ell^{j_i} \ff)$ for some $0 \leq j_i \leq a$.  In the former case, it follows from Proposition \ref{TEDIOUSALGEBRAPROP2}a) that
\[ \Q(\ell_1^{j_1} \ff) \otimes_{\Q(\ff)} \cdots \otimes \Q(\ell_r^{j_r} \ff) \cong \Q(\ell_1^{j_1} \cdots \ell_r^{j_r} \ff). \]
Suppose now that some $B_i$ contains $K$, and let the number of such $1 \leq i \leq r$ be $s$.   In this case, it follows from parts b) and c) of Proposition \ref{TEDIOUSALGEBRAPROP2} that 
\[ B \cong K(\ell_1^{j_1} \cdots \ell_r^{j_r})^{2^{s-1}}. \]
From this we deduce:

\begin{thm}
\label{COR6.1}
\label{9.1}
Let $\Delta = \ff^2 \Delta_K$ be an imaginary quadratic discriminant with $\Delta_K < -4$.  Let $M \mid N \in \Z^+$.  Let $P$ be a 
$\Delta$-CM closed point on $X_0(M,N)$.  
\begin{itemize}
\item[a)] The residue field $\Q(P)$ is isomorphic to either $\Q(M\ff)$ or $K(M \ff)$ for some $M \mid N$.
\item[b)] Let $M = \ell_1^{a_1} \cdots \ell_r^{a_r}$, $N = \ell_1^{a_1} \cdots \ell_r^{a_r}$ be the prime power decompositions of $M$ and $N$.  For $1 \leq i \leq r$, let 
$\pi_i: X_0(M,N) \ra X_0(\ell_i^{a_i'},\ell_i^{a_i})$ be the natural map and put $P_i \coloneqq \pi_i(P)$.  The following are equivalent:
\begin{itemize}
\item[(i)] The field $\Q(P)$ is formally real.
\item[(ii)] The field $\Q(P)$ does not contain $K$.
\item[(iii)] For all $1 \leq i \leq r$, the field $\Q(P_i)$ is formally real.
\item[iv)] For all $1 \leq i \leq r$, the field $\Q(P_i)$ does not contain $K$. 
\end{itemize}
\end{itemize}
\end{thm}


\subsection{Primitive Residue Fields and Primitive Degrees I}
In this section we treat the (harder) case in which ($M = 1$) or ($M = 2$ and $\Delta$ is even).  
\\ \\
In these cases there is always a closed $\Delta$-CM point on $X_0(M,N)$ with residue field isomorphic to $\Q(N \ff)$, and from this it 
follows that there is exactly one $B \mid N$ for which there is a primitive residue field isomorphic to $\Q(B \ff)$: this residue field 
is obtained by taking, for each $1 \leq i \leq r$, the natural number $b_i$ to be the least integer $B_i$ such that $\Q(\ell_i^{B_i} \ff)$ 
is isomorphic to the residue field of 
some $\Delta$-CM point on $X_0(\ell_i^{a_i'},\ell_i^{a_i})$, and we have 
\[ B = \ell_1^{b_1} \cdots \ell_r^{b_r}. \]
 There is at most 
one other primitive residue field of a $\Delta$-CM point on $X_0(M,N)$, necessarily of the form $K(C \ff)$ for some $C \mid N$: 
this additional primitive field occurs if and only if there are two primitive residue fields for $\Delta$-CM points on $X_0(\ell_i^{a_i'},\ell_i^{a_i})$ for some 
$1 \leq i \leq r$, in which case for $1 \leq i \leq r$ we let $c_i$ be the least natural number $C_i$ for which there is
a $\Delta$-CM point on $X_0(\ell_i^{a_i'},\ell_i^{a_i})$ with residue field isomorphic to either $\Q(\ell_i^{C_i} \ff)$ or to $K(\ell_i^{C_i} \ff)$; then
we have 
 \[ C = \ell_1^{c_1} \cdots \ell_r^{c_r}. \]
We are also interested in when there is a unique primitive degree $[\Q(P):\Q]$ of $\Delta$-CM closed points $P \in X_0(M,N)$.  
This occurs if there is a unique primitive residue field $\Q(B\ff)$.  In the case of two primitive residue fields $\Q(b\ff)$ and $K(C\ff)$, we 
put 
\[  \bb \coloneqq [\Q(B\ff):\Q], \ \cc \coloneqq [K(C\ff):\Q] ,\]
and then we still have a unique primitive degree if and only if one of $\bb$ and $\cc$ divides the other.  As we will soon see, we always have 
$\cc \leq \bb$ (our analysis shows that this holds on $X_0(\ell^{a_i'},\ell^{a_i})$ in every case, and the general case follows easily from this) 
so the question comes down to determining whether $\cc \mid \bb$.  It is easy to see that for this divisibility to hold it is sufficient for the analogous divisibility to hold 
at every prime power.  However it turns out that a divisibility that is ``lost'' at one prime power can be ``regained'' at a different prime power.  This makes the form 
of the final answer a bit complicated, even though the proof amounts to keeping track of a bunch of factors of $2$.  
\\ \\
The following table records information about primitive degrees in the prime power case.

{
\begin{center}
    \begin{tabular}{|c|c|c|c|c|c}
     Case & $b$ & $d_b = [\Q(\ell^b \ff):\Q(\ff)]$ & $c$ & $d_c = [K(\ell^c \ff):\Q(\ff)]$  & $d_c \mid d_b$? \\  \hline
1.1a) & $0$ & $1$ & - & -  & - \\
1.1b) & $1$ & $3$  & - & -  & - \\
1.2 & $a$ & $\ell^{a-1}(\ell-1)$ & $0$ & $2$  & Y \\
1.3 & $a$ & $\ell^{a-1}(\ell+1)$ & - & -  & - \\
1.4 & $a-1$ & $\ell^{a-1}$  & - & -  & - \\
1.5a) & $0$ & $1$  & - & -  & - \\
1.5b) & $a-2L$ & $\ell^{a-2L}$ & $0$ & $2$  & N \\
1.6a) & $0$ & $1$  & - & -  & - \\
1.6b) & $a-2L$ & $\ell^{a-2L}$  & - & -  & - \\
1.7a) & $0$ & $1$  & - & -  & - \\
1.7b) & $a-2L-1$ & $\ell^{a-2L-1}$  & - & -  & - \\
1.8a) & $a$ & $2^a$ & $0$ & $2$  & Y \\
1.8b) & $0$ & $1$  & - & - & - \\
1.8c) & $a-2L+2$ & $2^{a-2L+2}$ & $0$ & $2$  & Y \\
1.9a) & $a$ & $2^a$ & $a-2$ & $2^{a-1}$  & Y \\
1.9b) & $0$ & $1$  & - & -  & - \\
1.9c) & $a-2L+2$ & $2^{a-2L+2}$ & $a-2L$ & $2^{a-2L+1}$ & Y \\
1.10a) & $0$ & $1$  & - & -  & - \\
1.10b) & $a-2L$ & $2^{a-2L}$ & $a-2L-1$ & $2^{a-2L}$  & Y \\
1.11a) & $0$ & $1$  & - & -  & - \\
1.11b) & $a-2L-1$ & $2^{a-2L-1}$  & - & -  & - \\
2.0 & $1$ & $2$ & - & - & - \\
2.1 & $a$ & $2^a$ & $a-1$ & $2^a$ & Y \\
2.2 & $a-1$ & $2^{a-1}$ & - & - & - \\
2.3 & $a$ & $2^a$ & $1$ & $4$ & Y \\
2.4 & $1$ & $2$ & - & - & - \\
2.5 & $a-2L+2$ & $2^{a-2L+2}$ & $1$ & $4$ & Y \\
2.6 & $2$ & $4$ & $1$ & $4$ & Y \\
2.7 & $a$ & $2^a$ & $a-2$ & $2^{a-1}$ & Y \\
2.8 & $1$ & $2$ & - & - & - \\
2.9 & $2$ & $4$ & $1$ & $4$ & Y \\
2.10 & $a-2L+2$ & $2^{a-2L+2}$ & $a-2L$ & $2^{a-2L+1}$ & Y \\
2.11 & $1$ & $2$ & - & - & - \\
2.12 & $a-2L$ & $2^{a-2L}$ & $a-2L-1$ & $2^{a-2L}$ & Y \\
2.13 & $1$ & $2$ & - & - & - \\
2.14 & $a-2L-1$ & $2^{a-2L-1}$ & - & - & -
     \end{tabular}
\end{center}
}
\noindent
The key takeaway here is that $d_c \mid d_b$ in every case except Case 1.5b).  

\newcommand{\ibullet}{i_{\bullet}}
\newcommand{\jbullet}{j_{\bullet}}
\begin{thm}
\label{COOLTECHTHM}
Let $\Delta = \ff^2 \Delta_K$ be an imaginary quadratic discriminant with $\Delta_K < -4$, and let $M = \ell_1^{a_1'} \cdots \ell_r^{a_r'} \mid N = \ell_1^{a_1} \cdots \ell_r^{a_r}$.  We suppose that $M \in \{1,2\}$.  For $1 \leq i \leq r$, let 
$b_i \geq 0$ be the unique natural number such that $\Q(\ell_i^{b_i} \ff)$ occurs up to isomorphism as a primitive residue field of a closed $\Delta$-CM point on $X_0(\ell_i^{a_i'},\ell_i^{a_i})$.  Let $c_i$ be 
equal to $b_i$ if there is a unique primitive residue field of $\Delta$-CM points on $X_0(\ell_i^{a_i'},\ell_i^{a_i})$ and otherwise let it be such that the unique non-real primitive residue field of a closed $\Delta$-CM point 
on $X_0(\ell_i^{a_i'},\ell_i^{a_i})$ is $K(\ell_i^{c_i} \ff)$.  Put $B \coloneqq \ell_1^{b_1} \cdots \ell_r^{b_r}$ and $C \coloneqq \ell_1^{c_1} \cdots \ell_r^{c_r}$.  Let $s$ be the number of $1 \leq i \leq r$ such that 
there is a non-real primitive residue field of a closed $\Delta$-CM point on $X_0(\ell_i^{a_i'},\ell_i^{a_i})$.   
\begin{itemize}
\item[a)] If $s = 0$, the unique primitive residue field of a $\Delta$-CM point on $X_0(M,N)$ is $\Q(B\ff)$, so the unique primitive degree of a $\Delta$-CM point on $X_0(M,N)$ is $[\Q(B \ff):\Q]$.
\item[b)] If $s \geq 1$ and there is some $1 \leq i \leq r$ such that there are two primitive residue fields of closed $\Delta$-CM points on $X_0(\ell_i^{a_i'},\ell_i^{a_i})$ and we are not in Case 1.5b) with respect to $\Delta$ and $\ell_i^{a_i}$,
then:
\begin{itemize}
\item[(i)] There are two primitive residue fields of $\Delta$-CM points on $X_0(M,N)$: $\Q(B \ff)$ and $K(C \ff)$. 
\item[(ii)] The unique primitive degree of $\Delta$-CM points on $X_0(M,N)$ is $[K(C\ff):\Q]$.  
\end{itemize}
\item[c)] If $s \geq 1$ and for all $1 \leq i \leq r$ such that there are two primitive residue fields of closed $\Delta$-CM points on $X_0(\ell_i^{a_i'},\ell_i^{a_i})$ we are in Case 1.5b), 
then there are two primitive degrees of $\Delta$-CM points on $X_0(M,N)$: $[\Q(B \ff):\Q]$ and $[K(C \ff):\Q]$.  
\end{itemize}
\end{thm}
\begin{proof}
Step 1: The case $s = 0$ follows from the discussion above.  Henceforth we suppose $s \geq 1$.   In this case there are (up to isomorphism) two primitive 
residue fields of $\Delta$-CM closed points on $X_0(M,N)$: $\Q(B \ff)$ and $K(C \ff)$.  Put 
\[ \mathbf{b}  \coloneqq [\Q(B \ff):\Q(\ff)], \ \mathbf{c} \coloneqq [K(C \ff):\Q(\ff)]. \]
For each $1 \leq i \leq r$, let $F_i$ be a primitive residue field of a closed point of a $\Delta$-CM elliptic curve on $X_0(\ell_i^{a_i'},\ell_i^{a_i})$; if there is any non-real such field, we take $F_i$ to be nonreal.  Since for each $i$ such that there are two primitive residue fields $\Q(\ell_i^{b_i} \ff)$ and $K(\ell_i^{c_i} \ff)$ we have $[K(\ell_i^{c_i} \ff):\Q] \leq [\Q(\ell_i^{b_i} \ff):\Q$], it follows that 
\[ \dim_{\Q(\ff)} F_1 \otimes_{\Q(\ff)} \otimes \cdots \otimes_{\Q(\ff)} F_r \leq \dim_{\Q(\ff)} \Q(\ell_1^{b_1} \ff) \otimes_{\Q(\ff)} \cdots \otimes_{\Q(\ff)} \Q(\ell_r^{b_r} \ff) = [\Q(B \ff):\Q(\ff)]. \]
Since also 
\begin{equation}
\label{COOLTECHEQ1}
F_1 \otimes_{\Q(\ff)} \cdots \otimes_{\Q(\ff)} F_r \cong K(C \ff)^{2^{s-1}}, 
\end{equation}
it follows that $\mathbf{c} \leq \mathbf{b}$.  Thus there is a unique primitive degree if and only if $\mathbf{c} \mid \mathbf{b}$.  \\
Step 2: Since $K(C \ff) \subset K(B \ff) = K \Q(B \ff)$, we have $\mathbf{c} \mid 2 \mathbf{b}$.  In particular, we have $\ord_p(\mathbf{c}) \leq \ord_p(\mathbf{b})$ 
for every odd prime $p$.  Moreover we have 
\[ \ord_2(\mathbf{c})= 1 + \ord_2([\Q(C\ff):\Q(\ff)]) = 1 + \sum_{i=1}^r [\Q(\ell^{c_i} \ff):\Q(\ff)] \]
\[ \ord_2(\mathbf{b}) = \ord_2([\Q(B \ff):\Q(\ff)]) = \sum_{i=1}^r [\Q(\ell^{b_i} \ff_:\Q(\ff)], \]
so in the $s \geq 1$ case it follows that $\mathbf{c} \mid \mathbf{b}$ iff there is some $1 \leq i \leq r$ such that there are two primitive residue fields of $\Delta$-CM closed points on $X_0(\ell_i^{a_i'},\ell_i^{a_i})$ 
for which we have 
\[\ord_2([\Q(\ell^{c_i} \ff):\Q(\ff)]) < \ord_2([\Q(\ell^{b_i} \ff):\Q(\ff)], \]
which holds iff
\[\ord_2([K(\ell^{c_i} \ff):\Q(\ff)]) \leq \ord_2([\Q(\ell^{b_i} \ff):\Q(\ff)]. \]
Consulting the Table above, we see that this holds in every case in which there are two primitive residue fields \emph{except} Case 1.5b).
\end{proof}

\subsection{Primitive Residue Fields and Primitive Degrees II} In this section we treat the case in which ($M = 2$ and $\Delta$ is odd) or $M \geq 3$.  In this case it follows from \S 8.3, 8.4 and 9.1 that there is always a unique primitive residue 
field, which is a ring class field $K(C\ff)$.  \\ \indent Here $C$ is computed as follows: let $M = \ell_1^{a_1'} \cdots \ \ell_r^{a_r'}$ and $N = \ell_1^{a_1} \cdots \ell_r^{a_r}$.  Fix $1 \leq i \leq r$.  If the only primitive residue field of a $\Delta$-CM point on 
$X_0(\ell_i^{a_i'},\ell_i^{a_i})$ is $\Q(\ell^c \ff)$, put $c_i \coloneqq c$.  If the primitive residue fields of a $\Delta$-CM point on $X_0(\ell_i^{a_i'},\ell_i^{a_i})$ are $\Q(\ell^b \ff)$ and $K(\ell^c \ff)$, put $c_i \coloneqq c$.  If the only primitive 
residue field of a $\Delta$-CM point on $X_0(\ell_i^{a_i'},\ell_i^{a_i})$ is $K(\ell^c \ff)$, put $c_i \coloneqq c$.  Finally, put 
\[ C \coloneqq \ell_1^{c_i} \cdots \ell_r^{c_r}. \]

\section{CM Points of Odd Degree}
\noindent
In this section we will determine all triples $(\Delta,M,N)$ for which there is an odd degree $\Delta$-CM point on $X_0(M,N)$.  For each 
such triple there is a unique primitive (i.e., not a proper multiple of any other degree of a $\Delta$-CM point on $X_0(M,N)$) odd degree and we determine that.  From this it follows that for every $(M,N)$ such that $X_0(M,N)$ has odd degree CM points, all such 
points have degree multiples of the least such degree $d_{\oddCM}(X_0(N))$, which we compute.  Since all CM elliptic curves over number fields are $\Q$-curves, these results complement recent work of Cremona and Najman, especially \cite[Thm. 1.1]{Cremona-Najman21}.  
\\ \\
Using the close relationship between degrees of CM points on $X_0(M,N)$ and degrees of CM points on $X_1(M,N)$ given by 
Theorem \ref{INERTNESSTHM} for $\Delta < -4$ and \cite[Thm. 3.7]{CGPS} for $\Delta \in \{-4,-3\}$, we deduce analogous results for odd degree CM points on $X_1(M,N)$.  Thus we recover variants of results of Bourdon-Clark \cite{BCS17} and Bourdon-Pollack \cite{BP17}.  



\subsection{Some Preliminaries}
If $\pi: X \ra Y$ is a finite $\Q$-morphism of nice curves and $P$ is a closed point of $X$, then we have an inclusion of residue fields $\Q(\pi(P)) \subset \Q(P)$, and thus the degree of $\pi(P)$ divides the degree of $P$.  Applying this observation to the morphism $\pi: X_0(N) \ra X(1)$ we see that in order for a $\Delta$-CM point $P$ on $X_0(N)$ to have odd degree it is necessary that 
$[\Q(\pi(P)):\Q] = h_{\Delta}$ be odd.  It follows from Lemma \ref{GENUSLEMMA} that \[h_{\Delta} \text{ is odd } \iff \Delta \in \{-4,-8,-12,-16\} \text{ or } \Delta = - 2^{\epsilon} \ell^{2L+1}, \]
where $\epsilon \in \{0,2\}$, $\ell \equiv 3 \pmod{4}$ is a prime number and $L \in \Z^+$.  

\begin{lemma}
\label{7.1}
\begin{itemize}
\item[(a)] Let $P$ be an odd degree $\Delta$-CM point on $X_0(4)$.  Then $\Delta \in \{-4,-16\}$.
\item[b)] For $\Delta \in \{-4,-16\}$, the unique primitive degree of a $\Delta$-CM point on $X_0(4)$ is $1$.  
\end{itemize}
\end{lemma}
\begin{proof}
a) Let $\Delta = \ff^2 \Delta_K$.  Because $X_0(4) = X_1(4)$, it follows from a result of Aoki \cite{Aoki95} (see also \cite[Thm. 5.1]{BCS17}) 
that $\Delta_K = -4$.  Since $h_{\Delta}$ must be odd, as above it follows that $\Delta \in \{-4,-16\}$.   \\
b) Consider the elliptic curves 
\[ \ (E_1)_{/\Q}: y^2 = x^3 +4x, \ (E_2)_{/\Q}: y^2 = x^3-11x+14. \]
Then $E_1$ has $-4$-CM, $E_2$ has $-16$-CM and $E_1(\Q)[\tors] \cong E_2(\Q)[\tors] \cong \Z/4\Z$.
\end{proof}
\noindent
It turns out that there are very few odd degree CM points on $X_0(M,N)$ with $M \geq 2$.  

\begin{prop}
\label{7.2}
\begin{itemize}
\item[a)] Let $2 \leq M \mid N$, and let $P$ be an odd degree $\Delta$-CM point on $X_0(M,N)$.  Then $M = N = 2$ and $\Delta = -4$.
\item[b)] The unique primitive degree of a $-4$-CM point on $X_0(2,2)$ is $1$.
\end{itemize}
\end{prop}
\begin{proof}
a) Put $\Delta = \ff^2 \Delta_K$.  Let $P$ be an odd degree $\Delta$-CM point on $X_0(M,N)$ with $M \geq 2$.  \\ \indent
$\bullet$ Suppose first that $M$ is divisible by a prime $\ell \geq 3$.  Then there is an odd degree $\Delta$-CM point on $X_0(\ell,\ell)$.  
If $\Delta_K < -4$ this is ruled out by Theorem \ref{BETTERPARISHTHM}.    If $\Delta_K = -3$ and $\ell \geq 5$ we will 
show in Proposition  \ref{DELTAK3PROP} below that there is not even an odd degree $\Delta$-CM point on $X_0(1,\ell)$.  Finally, if $\Delta_K = -3$ 
then we have $X_0(3,3) = X_1(3,3)$, so the constant subfield of $\Q(X_0(3,3))$ is $\Q(\zeta_3)$ and every closed point on $X_0(3,3)$ 
has even degree.  \\
$\bullet$ Thus $M$ must be a power of $2$, and there is an odd degree $\Delta$-CM point on $X_0(2,2) = X_1(2,2)$.  By work of Aoki \cite{Aoki95} and Bourdon-Clark-Stankewicz \cite[Thm. 5.1b)]{BCS17} we must have $\Delta = -4$.  This same work 
shows that there is no odd degree CM point on $X_1(2,4) = X_0(2,4)$.  \\
b) The elliptic curve $E_{/\Q}: y^2 = x^3-x$ has $-4$-CM and $E(\Q)[\tors] \cong \Z/2\Z \times \Z/2\Z$.  
\end{proof}
\noindent
In view of Proposition \ref{7.2} we restrict our attention to $X_0(N)$.

\subsection{$\Delta = -8$}

\begin{prop}
\label{6.11}
For $N \in \Z^+$, the curve $X_0(N)_{/\Q}$ has an odd degree $-8$-CM point iff $N \in \{1,2\}$, in which case the unique primitive 
degree of $-8$-CM points is $1$.
\end{prop}
\begin{proof}
Because $h_{-8} = 1$, there is a $\Q$-rational $-8$-CM point on $X(1)$.  By \S 8.2, Case 1a) there is a $\Q$-rational $-8$-CM point on 
$X_0(2)$.  By Lemma \ref{7.1} there is no $\Q$-rational $-8$-CM point on $X_0(4)$.  It remains to rule out the odd degree $-8$-CM 
points on $X_0(\ell)$ for an odd prime $\ell$.  It follows from \S 10.2, Cases 2 and 3 that a primitive residue 
field of a $-8$-CM point on $X_0(\ell)$ is isomorphic to a rational ring class field of discriminant $- 8 \ell^2$ or to the ring class field of discriminant $-8$, both of which have even degree.  
\end{proof}

\subsection{$\Delta \in \{-4,-16\}$}


\begin{prop}
\label{DELTAK4PROP}
Let $N, \ff \in \Z^+$, let $\Delta = - 4\ff^2$ and let $P \in X_0(N)$ be a $\Delta$-CM point of odd degree.  Then:
\begin{itemize}
\item[a)] We have $N \in \{1,2,4\}$ and $\Delta \in \{-4,-16\}$.  
\item[b)] Conversely, for $N \in \{1,2,4\}$ and $\Delta \in \{-4,-16\}$, the unique primitive degree of a $\Delta$-CM point on $X_0(N)$ 
is $1$.  
\end{itemize}
\end{prop}
\begin{proof}
As above we need $h_{\Delta}$ to be odd, and thus $\Delta \in \{-4,-16\}$.  \\
a) We claim that we cannot have $N = \ell$ for an odd prime $\ell$.  Indeed, let $E$ be a $\Delta$-CM elliptic curve 
and $\varphi: E \ra E'$ be an $\ell$-isogeny.  Then $\varphi$ can be proper iff there is an ideal $I$ of $\OO(\Delta)$ of norm $\ell$, 
which occurs iff $\ell \equiv 1 \pmod{4}$.  In this case there are two such ideals that are interchanged by complex conjugation, so 
the field of moduli of $\varphi$ has even degree.  Otherwise $\varphi$ is descending, so $\Q(\varphi)$ contains the 
rational ring class field of discriminant $\ell^2 \Delta$, which by Proposition \ref{LCLASSCOR}b) has even degree.  
Thus $N$ must be a power of $2$.  \\
$\bullet$ We claim that we cannot have $N = 8$. \\ \indent  Let $\varphi: E \ra E'$ be a 
cyclic $8$-isogeny of odd degree such that $E$ has $-4$-CM and $E'$ has $(-2^{2L} \cdot 4)$-CM.  Because the ring class fields of conductor $-2^{2L} \cdot 4$ have even degree for all $L \geq 2$, we must have $L \in \{0,1\}$.  If $L = 0$ then the isogeny is 
proper and thus is isomorphic to $E \ra E/E[I]$ for a primitive $8$-ideal $I$ of $\OO(-4)$.  Since every ideal of $2$-power order in 
the Dedekind domain is a power of the unique prime ideal $\pp_2$ of norm $2$ and $\pp_2^2 = (2)$, there is no primitive $8$-ideal.  
If $L = 1$ then the path in the isogeny graph, since must consist of vertices of level $0$ and $1$, must include a downward edge 
followed by an upward edge, and by uniqueness of the upward edge from the unique vertex at level $1$, this must be an isogeny 
followed by its dual isogeny, hence $\varphi$ is not cyclic.  \\ \indent
Now let $\varphi: E \ra E'$ be a cyclic $8$-isogeny of odd degree such that $E$ has $-16$-CM.  As above, $E'$ must have either $-4$-CM or $-16$-CM.  If $E'$ has $-4$-CM, then $\varphi^{\vee}$ is a cyclic $8$-isogeny of the type that was ruled out above.  
If $E'$ has $-16$-CM, then $\varphi$ is a proper $8$-isogeny of odd degree, hence of the form $E \ra E/[I]$ for a primitive, proper 
real ideal $I$.  By \cite[Thm. 3.7]{BCS17} there are no such ideals in the quadratic order $\OO(-16)$. 
Thus $N \mid 4$.  \\
b) This follows from Lemma \ref{7.1}.
\end{proof}

\subsection{$\Delta = -\ell^{2L+1}$, $3 < \ell \equiv 3 \pmod{4}$, $L \in \N$}

\begin{prop}
\label{6.12}
\label{7.5}
Let $\ell$ be a prime number such that $\ell \equiv 3 \pmod{4}$ and $\ell > 3$.  Let $L \in \N$, let $N \in \Z^+$, and put 
 $\Delta \coloneqq -\ell^{2L+1}$.
\begin{itemize}
\item[a)] If there is an odd degree $\Delta$-CM point on $X_0(N)$, then $N \in \{\ell^a,2 \ell^a\}$ for some $a \in \Z^+$.  
\item[b)] Let $a \in \Z^+$.  There is an odd degree $\Delta$-CM point on $X_0(\ell^a)$, and the unique primitive degree of 
such a point is $h_{-\ell} \cdot  \ell^{L+ \max(a-2L-1,0)}$. 
\item[c)] Let $a \in \Z^+$.  There is an odd degree $\Delta$-CM point on $X_0(2\ell^a)$, and the unique primitive degree of such a 
point is $\begin{cases} 3h_{-\ell} \cdot  \ell^{L+ \max(a-2L-1,0)} & \text{if } \ell \equiv 3 \pmod{8} \\
h_{-\ell} \cdot  \ell^{L+ \max(a-2L-1,0)} & \text{if }\ell \equiv 7 \pmod{8} \end{cases}$.

\end{itemize}
\end{prop}
\begin{proof}
a) As in the proof of Proposition \ref{6.11} above, it suffices to show that there is no odd degree $\Delta$-CM point on $X_0(N)$ 
when $N = 4$ or $N = p$ for a prime $p \nmid 2 \ell$.   Both are ruled out using Cases 2 and 3 of \S 9.2, as above.  \\
b) This follows from Case 7 of \S 9.2 and the identity $h_{-\ell^{2L+1}} = h_{-\ell} \cdot \ell^L$ (cf. Corollary \ref{LCLASSCOR}c)). 
c) By Case 1 of \S 9.2, the unique primitive residue field of a $\Delta$-CM point on $X_0(2)$ is $\Q(2\ff)$.  It follows that the unique 
odd degree primitive residue field of a $\Delta$-CM point on $X_0(2 \ell^a)$ is the rational ring class field of discriminant 
$-4 \cdot \ell^{2L+2\max(a-2L-1,0)+1}$, which has the degree claimed in the statement.
\end{proof}

\subsection{$\Delta = -4 \cdot \ell^{2L+1}$, $3 < \ell \equiv 3 \pmod{4}$, $L \in \N$}

\begin{prop}
\label{6.13}
\label{7.6}
Let $3< \ell \equiv 3 \pmod{4}$ be a prime, let $L,N \in \Z^+$, and put $\Delta \coloneqq -4 \cdot \ell^{2L+1}$.
\begin{itemize}
\item[a)] If there is an odd degree $\Delta$-CM point on $X_0(N)$, then $N \in \{\ell^a,2 \ell^a\}$ for some $a \in \Z^+$.
\item[b)] Let $a \in \Z^+$.  There is an odd degree $\Delta$-CM point on $X_0(\ell^a)$, and the unique primitive degree of 
such a point is $h_{-\ell} \cdot \ell^{L+\max(a-2L-1,0)}$.
\item[c)] Let $a \in \Z^+$.  There is an odd degree $\Delta$-CM point on $X_0(2\ell^a)$, and the unique primitive degree of 
such a point is $(2-\left( \frac{\ell}{2} \right)h_{-\ell} \cdot \ell^{L+\max(a-2L-1,0)}$.
\end{itemize}
\end{prop}
\begin{proof}
The proof is very similar to that of Proposition \ref{6.12} and may be left to the reader.
\end{proof}


\subsection{A Useful Lemma}

\begin{lemma}
\label{CHEATLEMMA}
\label{7.7}
Let $\Delta$ be an imaginary quadratic discriminant, let $K = \Q(\sqrt{\Delta})$, and let $M \mid N$ with $N \geq 3$.  Let $d$ be the least degree of a $\Delta$-CM point on 
$X_0(M,N)$.  Suppose that:
\begin{itemize}
\item[(i)] The least degree $\tilde{d}$ of a $\Delta$-CM point on $X_1(M,N)_{/\Q}$ is equal to the least degree $\tilde{d}_K$of a $\Delta$-CM point on $X_1(M,N)_{/K}$.
\item[(ii)] We have $\tilde{d} = \frac{\varphi(N)}{2} \cdot d$.  
\end{itemize}
Then $d$ is the unique primitive degree of $\Delta$-CM points on $X_0(M,N)$.  
\end{lemma}
\begin{proof}
By \cite[Thm. 4.1]{BCII}, we have that $\tilde{d}$ is the unique primitive degree of a $\Delta$-CM point on $X_1(M,N)_{/K}$.  
Under the natural unramified degree $2$ morphism $\nu: X_1(M,N)_{/K} \ra X_1(M,N)_{/\Q}$ if $x \in X_1(M,N)_{/K}$ is a closed 
point, then $\deg(\nu(x)) \in \{ \deg(x),2\deg(x)\}$, so it follows that $\tilde{d}$ is also the unique primitive degree of a $\Delta$-CM 
closed point on $X_0(M,N)_{/\Q}$.  Now consider the natural covering $\pi: X_1(M,N) \ra X_0(M,N)$, let $P \in X_0(M,N)$ be a 
$\Delta$-CM closed point, and let $\tilde{P} \in X_1(M,N)$ be any closed point such that $\pi(\tilde{P}) = P$.  Since the corresponding extension of function fields is Galois of degree $\frac{\varphi(N)}{2}$, there is a positive integer $f \mid \frac{\varphi(N)}{2}$ such that 
\[ \deg(\tilde{P}) = f \deg(P). \]
It follows that 
\[ d= \frac{\tilde{d}}{\varphi(N)/2} \mid \frac{\deg(\tilde{P})}{f} = \deg(P).  \qedhere\]
\end{proof}
\noindent
For all triples $(\Delta,M,N)$ the quantity $\tilde{d}_K$ is computed in \cite[Thm. 4.1]{BCII}, the quantity $\tilde{d}$ is 
computed in \cite[\S 8]{BCII}.  When $M = 1$ the quantity $d$ is computed in \cite[Thm. 3.7]{CGPS}.  Thus when $M = 1$ (the case of 
interest to us here) we are always in a position to know whether Lemma \ref{CHEATLEMMA} applies.   \\ \indent
Let us give a more explicit desription of the triples $(\Delta,1,N)$ for which the conditions (i) and (ii) of Lemma \ref{CHEATLEMMA} 
applies.  \\ \indent Whether (i) holds depends on the prime power factorization $N = \ell_1^{a_1} \cdots \ell_r^{a_r}$; the condition (i) holds
for $(\Delta,1,N)$ iff it holds for $(\Delta,1,\ell_i^{a_i})$ for all $1 \leq i \leq r$.  Condition (i) holds for all $(\Delta,1,2)$.  
If $\left( \frac{\Delta}{\ell} \right) = 1$ and $\ell^a > 2$, then condition (i) does not hold for $(\Delta,1,\ell^a)$.  If $\left( \frac{\Delta}{\ell} \right) = -1$ then condition (i) holds for $(\Delta,1,\ell^a)$.  When $\ell^a > 2$ and $\ell \mid \Delta$ the conditions are a bit complicated, and we refer to \cite[Prop. 6.4 and Thm. 6.5]{BCII} for the complete answer but only record the following case for 
future use: for all $a,\ff \in \Z^+$, condition (i) holds for $(-3 \ff^2,1,3^a)$.
\\ \indent
By Theorem \ref{INERTNESSTHM}, condition (ii) of Lemma \ref{CHEATLEMMA} always holds when $\Delta < -4$.  By Theorem \cite[Thm. 3.7]{CGPS}, for $(-4,1,N)$ this condition holds iff $N$ is divisible either by $4$ or by some prime $\ell \equiv 3 \pmod{4}$, while for $(-3,1,N)$ this condition holds 
iff $N = 3$ or $N$ is divisible either by $9$ or by some prime $\ell \equiv 2 \pmod{3}$.

\subsection{$\Delta = - 3^{2L+1}$, $L \in \N$}

\begin{prop}
\label{DELTAK3PROP}
\label{7.8}
Let $\Delta = -3 \ff^2$ for some $\ff \in \Z^+$.
\begin{itemize}
\item[a)] Let $N \in \Z^+$ be such that there is an odd degree $\Delta$-CM point on $X_0(N)$.  Then 
$\ff = 2^{\epsilon} 3^L$ with $\epsilon \in \{0,1\}$ and $L \geq 0$.  Moreover $N$ is $1$, $2$, $3^a$ for some $a \in \Z^+$ or $2 \cdot 3^a$ for some $a \in \Z^+$.  
\item[b)] Suppose $\Delta = -3$. 
\begin{itemize}
\item[(i)] For $N = 3^a$ with $a \geq 1$, the unique primitive degree of a $(-3)$-CM point on $X_0(N)$ is $3^{\max(a-2,0)}$.  
\item[(ii)] For $N = 2 \cdot 3^a$ with $a \geq 1$, the unique primitive degree of a $(-3)$-CM point on $X_0(N)$ is $3^{a-1}$.
\end{itemize}
\item[c)] Suppose $\Delta = -3^{2L+1}$ with $L \geq 1$.
\begin{itemize}
\item[(i)] For $N= 3^a$ with $a \geq 1$, the unique primitive degree of a $(-3^{2L+1})$-CM point on $X_0(N)$ 
is $3^{L-1 + \max(a-2L-1,0)}$.  
\item[(ii)] For $N = 2 \cdot 3^a$ with $a \geq 1$, the unique primitive degree of a $(-3^{2L+1})$-CM point on $X_0(N)$ 
is $3^{L + \max(a-2L-1,0)}$.  
\end{itemize}
\item[d)] Suppose $\Delta = -4 \cdot 3^{2L+1}$ with $L \geq 0$.  For $N \in \{3^a,2 \cdot 3^a\}$, the unique primitive degree of a $\Delta$-CM point 
on $X_0(N)$ is $3^{L+\max(a-2L-1,0)}$.
\end{itemize}
\end{prop}
\begin{proof}
a) $\bullet$ By Lemma \ref{7.1}, there is no odd degree $\Delta$-CM point on $X_0(4)$.  \\
$\bullet$ We claim that for any prime $\ell \nmid 6$, there is no odd degree $\ell$-isogeny $\varphi: E \ra E'$ with $E$ a $\Delta$-CM 
elliptic curve.  Since $\left( \frac{\Delta}{\ell} \right) \neq 0$, every proper such isogeny would have its field moduli containing 
the CM field $K$, hence of even degree.  Otherwise the isogeny is decending and the field of moduli would contain the rational 
ring class field of discriminant $\ell^2 \Delta$, which always has even degree. \\
b), c), d) In all of these cases the conditions of Lemma \ref{CHEATLEMMA} apply and we get that there is a unique primitive degree 
$d$ of a $\Delta$-CM point on $X_0(N)$.  Moreover \cite[Thm. 3.7]{CGPS} and \cite[Thm. 7.2]{BCI} compute this degree in 
terms of $h_{-\Delta}$, which is in turn easily computed using $h_{-3} = 1$ and Corollary \ref{LCLASSCOR}.
\end{proof}


\subsection{Compiling Across $\Delta$}
For a modular curve $X(H)_{/\Q}$ we say that $d \in \Z^+$ is a \textbf{primitive CM degree} 
if there is a closed CM point on $X(H)$ of degree $d$ and there is no closed CM point on $X(H)$ of degree properly dividing $d$.  
\\ \\
It follows from our previous results that the pairs $(M,N)$ for which the modular curve $X_0(M,N)$ admits an odd degree CM point are: 
$(1,1)$, $(1,2)$, $(1,4)$, $(1,\ell^a)$ for a prime $\ell \equiv 3 \pmod{4}$ and $a \in \Z^+$, $(1, 2 \cdot \ell^a)$ for a prime $\ell \equiv 3 \pmod{4}$ and $a \in \Z^+$ and $(2,2)$.   In this section, for each of these pairs $(M,N)$ we will show that there is a unique odd 
primitive CM degree $d_{\oddCM}(X_0(M,N))$, compute $d_{\oddCM}(X_0(M,N))$, and list all \textbf{corresponding} imaginary quadratic discriminants: i.e., those $\Delta$ for which there is a $\Delta$-CM point on $X_0(M,N)$ of degree $d_{\oddCM}(X_0(M,N))$.
\\ \\
By Proposition 7.2, the only pair $(M,N)$ with $M \geq 2$ for which there is an odd degree CM point on $X_0(M,N)$ is $(2,2)$.  
This result also shows that $d_{\oddCM}(X_0(2,2)) = 1$ and the corresponding discriminant is $-4$.  Henceforth we take $M = 1$.

\begin{prop}We treat $N \in \Z^+$ for which there is a $\Q$-rational CM point on $X_0(N)$.
\begin{itemize}
\item[a)] We have $d_{\oddCM}(X_0(N)) = 1$ iff $N \in \{1,2,3,4,6,7,9,11,14,19,27,43,67,163\}$.  
\item[b)] We have $d_{\oddCM}(X_0(1)) = 1$.  The corresponding discriminants are $-3$, $-4$, $-7$, $-8$, $-11$, $-12$, $-16$, 
$-19$, $-27$, $-28$, $-43$, $-67$, and $-163$.  
\item[c)] We have $d_{\oddCM}(X_0(2)) = 1$.  The corresponding discriminants are $-3$, $-4$, $-7$, $-8$, $-12$, and $-16$.  
\item[d)] We have $d_{\oddCM}(X_0(3)) = 1$.  The corresponding discriminants are $-3$, $-12$, and $-27$.  
\item[e)] We have $d_{\oddCM}(X_0(4)) = 1$.  The coresponding discriminants are $-4$ and $-16$.
\item[f)] We have $d_{\oddCM}(X_0(6)) = 1$.  The corresponding discriminants are $-3$ and $-12$.  
\item[g)] For $N \in \{7,14\}$, We have $d_{\oddCM}(X_0(N)) = 1$.  The corresponding discriminants are $-7$ and $-28$.  
\item[h)] We have $d_{\oddCM}(X_0(9)) = 1$.  The corresponding discriminants are $-3$ and $-27$.  
\item[i)] For $\ell \in \{11,19,43,67\}$, we have $d_{\oddCM}(X_0(\ell)) = 1$.  The corresponding discriminant is $-\ell$.   
\item[j)] We have $d_{\oddCM}(X_0(27)) = 1$.  The corresponding discriminant is $-27$.  
\end{itemize}
\end{prop}
\begin{proof}
We have $d_{\oddCM}(X_0(N)) = 1$ iff the least degree of a CM point on $X_0(N)$ is $1$.  Since the degree of every 
$\Delta$-CM point on $X_0(N)$ is a multiple of the class number $h_{\Delta}$, we must have $h_{\Delta} = 1$, and the list of such discriminants is the one given in part b): this is the solution of the class number one problem due to Heegner, Baker and Stark.  
Conversely, for each of the $13$ class number one discriminants $\Delta$, it is straightforward to use our previous results on the primitive odd degrees of $\Delta$-CM points on $X_0(N)$ to find the set of all $N \in \Z^+$ such that there is a $\Q$-rational $\Delta$-CM point on 
$X_0(N)$.
\end{proof}

\begin{thm}
\label{7.10}
Let $\ell \equiv 3 \pmod{4}$ be a prime number, and let $a \in \Z^+$.
\begin{itemize}
\item[a)] Suppose that $\ell > 3$ and $a$ is odd.  
\begin{itemize}
\item[(i)] The unique primitive odd degree of a CM point on $X_0(\ell^a)$ is 
$h_{-\ell} \cdot \ell^{\frac{a-1}{2}}$.  The corresponding discriminants are $\begin{cases} -\ell^a & \text{if } \ell \equiv 3 \pmod{8} \\
-\ell^a, \ -4 \cdot \ell^a & \text{if } \ell \equiv 7 \pmod{8} \end{cases}$.
\item[(ii)] The unique primitive odd degree of a CM point on $X_0(2 \ell^a)$ is $(2-\left(\frac{\ell}{2}\right)) h_{-\ell} \cdot \ell^{\frac{a-1}{2}}$.  The corresponding discriminants are $-\ell^a$ and $-4 \cdot \ell^a$. 
\end{itemize}
\item[b)] Suppose $\ell > 3$ and $a$ is even.  
\begin{itemize}
\item[(i)] The unique primitive odd degree of a CM point on $X_0(\ell^a)$ is 
$h_{-\ell} \ell^{a/2}$.  The corresponding discriminants are $\begin{cases} -\ell^{a-1}, -\ell^{a+1} & \text{if } \ell \equiv 3 \pmod{8} \\
-\ell^{a-1}, -4 \cdot \ell^{a-1}, -\ell^{a+1}, -4 \cdot \ell^{a+1} & \text{if } \ell \equiv 7 \pmod{8} \end{cases}$. 
\item[(ii)] The unique primitive odd degree of a CM point on $X_0(2 \cdot \ell^a)$ is $(2-\left( \frac{\ell}{2} \right)) h_{-\ell}  \ell^{a/2}$.  
The corresponding discriminants are $-\ell^{a-1}, -4 \cdot \ell^{a-1}, -\ell^{a+1}, -4 \cdot \ell^{a+1}$.
\end{itemize}
\item[c)] Suppose that $\ell = 3$ and $a \geq 3$ is odd.
\begin{itemize}
\item[(i)] The unique primitive odd degree of a CM point on $X_0(3^a)$ is $3^{\frac{a-3}{2}}$.  The corresponding discriminant is 
$-3^a$.  
\item[(ii)] The unique primitive odd degree of a CM point on $X_0(2 \cdot 3^a)$ is $3^{\frac{a-1}{2}}$.  The corresponding discrminants are $-3^a$ and $-4 \cdot 3^a$.    
\end{itemize}
\item[d)] Suppose that $\ell = 3$ and $a \geq 4$ is even.
\begin{itemize}
\item[(i)] The unique primitive odd degree of a CM point on $X_0(3^a)$ is $3^{\frac{a-2}{2}}$.  The corresponding discriminants are 
$-3^{a-1}$ and $-3^{a+1}$.
\item[(ii)] The unique primitive odd degree of a CM point on $X_0(2 \cdot 3^a)$ is $3^{a/2}$.  The corresponding discriminants are 
$-3^{a-1}$, $-4 \cdot 3^{a-1}$, $-3 \cdot^{a+1}$ and $-4 \cdot 3^{a+1}$.   
\end{itemize}
\end{itemize}
\end{thm}
\begin{proof}
a),b) It follows from our previous results that for a prime $\ell$ with $3 < \ell \equiv \pmod{4}$ and $a \geq 1$, if there is an odd degree $\Delta$-CM point on $X_0(\ell^a)$, then $\Delta = -\ell^{2L+1}$ or $\Delta = -4 \ell^{2L+1}$ for some $L \geq 0$.  Conversely, when this holds, for each $L \geq 0$ the unique primitive odd degree $-\ell^{2L+1}$-CM point on $X_0(\ell^a)$ is computed by Proposition \ref{7.5}b): it is always $h_{-\ell}$ times a power of $\ell$ that depends on $a$ and $L$.  Similarly, Proposition \ref{7.6} computes for each $M \geq 0$ the unique primitive odd degree $-4 \cdot \ell^{2M+1}$-CM point on $X_0(\ell^a)$: it is alwys $(2-\left(\frac{2}{\ell}\right)h_{-\ell}$ 
times a power of $\ell$ that depends on $a$ and $M$.  Using these formulas it is straightforward to find the values of $L$ and/or $M$ 
that minimize the degree and see that this minimal degree divides the degree for each choice of $L$ and $M$.  The same holds 
for $X_0(2 \ell^a)$ in place of $X_0(\ell^a)$.  We leave the details to the reader. \\
c), d) This proved the same way parts a) and b), using Proposition \ref{7.8} in place of Proposition \ref{7.5}.  Again we may leave the details to the reader.
\end{proof}

\subsection{Odd Degree CM Points on $X_1(M,N)$}

\begin{prop}
\label{7.11}
For $\Delta$ and $M \mid N$, the following are equivalent:
\begin{itemize}
\item[(i)] There is an odd degree $\Delta$-CM point on $X_1(M,N)$.
\item[(ii)] There is an odd degree $\Delta$-CM point on $X_0(M,N)$.
\end{itemize}
\end{prop}
\begin{proof}
(i) $\implies$ (ii) Let $\pi: X_1(M,N) \ra X_0(M,N)$ be the natural map.  If $P \in X_1(M,N)$ is a $\Delta$-CM closed point of odd degree,
 then $\pi(P) \in X_0(M,N)$ is a $\Delta$-CM closed point of odd degree. \\
(ii) $\implies$ (ii) Let $P \in X_0(M,N)$ be a $\Delta$-CM closed point of odd degree.  By our previous results we have 
\begin{equation}
\label{AOKIEQ}
 (M,N) \in \{ (1,1), (1,2), (1,4), (1,\ell^a), (1,2\ell^a), (2,2)\} 
\end{equation}
for a prime $\ell \equiv 3 \pmod{4}$ and $a \in \Z^+$.  For $N \in \{1,2,4\}$ the map $\pi$ has degree $1$, so the result is 
clear.  Otherwise $N = 2^{\epsilon} \ell^a$ for some $\epsilon \in \{0,1\}$.  In this case we have 
\[ \deg(\pi) = \frac{\varphi(2^{\epsilon} \ell^a)}{2} = \ell^{a-1} \cdot \frac{\ell-1}{2}, \]
which (since $\ell \equiv 3 \pmod{4}$) is odd.  It follows that some lift of $P$ to $\tilde{P} \in X_1(M,N)$ is an odd degree $\Delta$-CM point.
\end{proof}

\begin{cor}
\label{7.12}
\begin{itemize}
Let $\Delta$ and $M \mid N$ with $N \geq 3$ be such that there is an odd degree $\Delta-$CM closed point on $X_0(M,N)$ (equivalently, on $X_1(M,N)$).
\item[a)]  Let $d$ be the unique 
primitive odd degree of a $\Delta$-CM point on $X_0(M,N)$.  Then the unique primitive odd degree of a $\Delta$-CM point on 
$X_1(M,N)$ is $\frac{\varphi(N)}{2} \cdot d$. 
\item[b)] We have \[d_{\oddCM}(X_1(M,N)) = \frac{\varphi(N)}{2} \cdot d_{\oddCM}(X_0(M,N)). \]
\item[c)] Thus Theorem \ref{7.10} computes $d_{\oddCM}(X_1(M,N))$ in all cases and also determines the corresponding discriminants: they are precisely the same discriminants as for $X_0(M,N)$.  
\end{itemize}
\end{cor}
\begin{proof}
Part a) holds for $\Delta < -4$ by Theorem \ref{INERTNESSTHM}.  It holds for $\Delta \in \{-3,-4\}$ because the method of 
determination of the unique primitive odd degree CM point on $X_0(M,N)$ in these cases used Lemma \ref{CHEATLEMMA}, which includes 
this as a hypothesis.  (In all cases in which this result was applied, \cite[Cor. 3.7]{CGPS} was also applied, showing that the minimal degree of a $\Delta$-CM point on $X_0(M,N)$ is odd.)  The remaining parts follow immediately.  
\end{proof}

\subsection{Comments}
The classification of pairs $(M,N)$ for which there is an odd degree CM point on $X_1(M,N)$ attained in Proposition \ref{7.11} is essentially\footnote{The latter result is a bit more precise in that it determines which groups $\Z/M\Z \times \Z/N\Z$ arise as the \emph{full} torsion subgroup $E(F)[\tors]$ of a CM elliptic curve over some odd degree number field $F$, whereas the 
present result determines which groups $\Z/M\Z \times \Z/N\Z$ can be embedded into $E(F)[\tors]$.  All discrepancies between the two 
results are accounted for by the fact \cite[Thm. 4.2]{BCS17} that if $\Delta < -4$ and $\Delta \equiv 1 \pmod{8}$ then every $\Delta$-CM elliptic curve defined over a number field must have an $F$-rational point of order $2$.}  the Odd Degree Theorem of Bourdon-Clark-Stankewicz \cite[Thm. 5.3]{BCS17}.  That the set of such pairs $(M,N)$ is contained in the list of (\ref{AOKIEQ}) had been obtained earlier by Aoki \cite{Aoki95}.
\\ \indent
The computation of $d_{\oddCM}(X_1(M,N))$ for all the $(M,N)$ listed in (\ref{AOKIEQ}) is essentially\footnote{Same proviso as above.} 
the Strong Odd Degree Theorem of Bourdon-Pollack \cite[Thm. 2.12]{BP17}.  Our computation of the corresponding discriminants in each case -- i.e., the complete list of $\Delta$ for which the minimum odd degree CM point on $X_1(M,N)$ is a $\Delta$-CM point -- 
is a further refinement, a ``Stronger Odd Degree Theorem.''  However there is also \cite[Remark 2.7]{BP17}, which gives 
most (but not all) of the restrictions on the corresponding discriminants in the case $N = \ell^a$.  That this is a Remark rather than a result seems to indicate that Bourdon and Pollack were not as interested in this refinement; it may well be that 
they could have proven the Stronger Odd Degree Theorem for $X_1(M,N)$ using  their methods.
\\ \\
The relationship between odd degree CM points on $X_0(M,N)$ versus $X_1(M,N)$ is interesting.  The implication (i) $\implies$ (ii) 
in Proposition \ref{7.11} is completely routine: of course if there is an odd degree CM point on $X_1(M,N)$ then its image on $X_0(M,N)$ 
is also an odd degree CM point.  But the converse implication (ii) $\implies$ (i) seems a bit surprising, even in retrospect.   One way to view this is: the fact that every $N \geq 3$ that appears in (\ref{AOKIEQ}) has $\frac{\varphi(N)}{2}$ odd seemed to play an important 
role in the Odd Degree Theorem on $X_1(M,N)$: indeed the proof of the Odd Degree Theorem used the fact that if $P \in X_1(N)$ 
is a CM point of odd degree $d$, then $\frac{\varphi(N)}{2} h_{\Delta} \mid d$ \cite[Thm. 4.12d)]{BCS17}.  One might have 
expected there to be odd degree CM points on $X_0(N)$ for values of $N \geq 3$ for which $\frac{\varphi(N)}{2}$ is even, which 
therefore would not give rise to odd degree CM points on $X_1(N)$.  This turns out not to be the case, but our argument does not directly explain why.
\\ \\
Although the structure of odd degree CM points on $X_0(M,N)$ is exactly the same as that of odd degree CM points on $X_1(M,N)$,
nevertheless we can deduce the results on $X_1(M,N)$ from those on $X_0(M,N)$, but not conversely, which seems to indicate once again that the curves $X_0(M,N)$ are fundamental.  There is however an exception to this: whereas for $\Delta_K < -4$ our discussion on 
$X_0(M,N)$ uses the description of primitive degrees of CM points that lies at the heart of the present work, for $\Delta_K \in \{-4,-3\}$ we use some observations about isogeny graphs but also Lemma \ref{CHEATLEMMA}.  Our application of Lemma \ref{CHEATLEMMA} in turn uses results of degrees of CM points on $X_1(M,N)_{/\Q}$ and $X_1(M,N)_{/K}$ due to Bourdon and Clark \cite{BCI}, \cite{BCII}.  Thus from 
the perspective of treating $X_0(M,N)$ as fundamental it is ``a cheat.''  This again points to the need to complete the present 
results with a treatment of $\Delta_K \in \{-4,-3\}$, but we did not want to allow this incompleteness to yield incomplete information
in the odd degree case.

\section*{Appendix}

\subsection{Cases 1-3: $L = 0$}

\subsubsection{Case 1: $L =0$, $\left (\frac{\Delta_K}{\ell} \right) = -1$}

{
\begin{center}
    \begin{tabular}{|c|c|}
     Residue Field & Multiplicity \\  \hline
$\Q(\ell^a \ff)$ & $1$
     \end{tabular}
\end{center}
}
\noindent
{\textbf{\sc Primitive Residue Fields:}} $\Q(\ell^a \ff)$
\\ \\
The path type that occurs in this case is I.


\subsubsection{Case 2: $L = 0$, $\left( \frac{\Delta_K}{\ell} \right) = 0$}

{
\begin{center}
    \begin{tabular}{|c|c|}
     Residue Field & Multiplicity \\  \hline
$\Q(\ff \ell^{a-1})$ & $1$ \\
$\Q(\ff \ell^a)$ & $1$
     \end{tabular}
\end{center}
}
\noindent
{\textbf{\sc Primitive Residue Fields:}} $\Q(\ell^{a-1} \ff)$
\\ \\
The path types that occur in this case are I. and III.


\subsubsection{Case 3: $L = 0$, $\left( \frac{\Delta_K}{\ell} \right) = 1$}

{
\begin{center}
    \begin{tabular}{|c|c|}
     Residue Field & Multiplicity \\  \hline
$K(\ff)$ & $1$ \\
$K( \ell \ff)$ & $1$ \\
$\vdots$ & $\vdots$ \\
$K(\ell^{a-1} \ff)$ & $1$ \\
$\Q( \ell^a \ff)$ & $1$
     \end{tabular}

\end{center}
}
\noindent
{\textbf{\sc Primitive Residue Fields:}} $\begin{cases} \Q(\ell^a \ff), \ K(\ff) & \ell^a > 2 \\ \Q(2\ff) = \Q(\ff) & \ell^a = 2 \end{cases}$
\\ \\
The path types that occur in this case are I. and IV.

\subsection{Case 4: $L \geq a = 1$}

{
\begin{center}
    \begin{tabular}{|c|c|}
     Residue Field & Multiplicity \\  \hline
$\Q(\ff)$ & $1$ \\
$\Q(\ff \ell)$ & $1$ 
     \end{tabular}
\end{center}
}
\noindent
{\textbf{\sc Primitive Residue Fields:}} $\Q(\ff)$
\\ \\
The path types that occur in this case are I. and II.


\subsection{Case 5: $L \geq a \geq 2$, $\ell > 2$}
Put $\gamma(a) \coloneqq \begin{cases} 1 & a \text{ odd} \\ 2 & a \text{ even} \end{cases}$.  

{
\begin{center}
    \begin{tabular}{|c|c|}
     Residue Field & Multiplicity \\  \hline
$\Q(\ff)$ & $1$ \\
$K(\ff)$ & $\frac{\ell^{\lfloor \frac{a}{2} \rfloor}-1}{2}$\\
$K(\ell^{\gamma(a)}\ff)$ & $\frac{(\ell-1) \ell^{\lfloor \frac{a-1}{2}\rfloor-1}}{2}$ \\
$K( \ell^{\gamma(a)+2}\ff) $ & $\frac{(\ell-1) \ell^{\lfloor \frac{a-1}{2}\rfloor-2}}{2}$ \\
$\vdots$ & $\vdots$ \\
$K(\ell^{a-2} \ff)$ & $\frac{\ell-1}{2}$ \\
$\Q(\ell^a \ff)$ & $1$
     \end{tabular}
\end{center}
}
\noindent
{\textbf{\sc Primitive Residue Fields:}} $\Q(\ff)$
\\ \\
The path types that occur in this case are I., II. and V.  


\subsection{Cases 6-9: $L \geq a \geq 2$, $\ell =2$}

\subsubsection{Case 6: $L \geq a = 2, \ \ell = 2$}

{
\begin{center}
    \begin{tabular}{|c|c|}
     Residue Field & Multiplicity \\  \hline
$\Q(\ff)$ & $2$ \\
$\Q(4\ff )$ & $1$ 
     \end{tabular}
\end{center}
}
\noindent
{\textbf{\sc Primitive Residue Fields:}} $\Q(\ff)$
\\ \\
The path types that occur in this case are I., II. and V$_1$.

\subsubsection{Case 7: $L \geq a = 3, \ \ell = 2$}

{
\begin{center}
    \begin{tabular}{|c|c|}
     Residue Field & Multiplicity \\  \hline
$\Q(\ff)$ & $2$ \\
$\Q(2 \ff)$ & $1$ \\
$\Q(2^3 \ff)$ & $1$
     \end{tabular}
\end{center}
}
\noindent
{\textbf{\sc Primitive Residue Fields:}} $\Q(\ff)$
\\ \\
The path types that occur in this case are I. II. V$_1$ and $V_2$.

\subsubsection{Case 8: $L \geq a \geq 4$, $\ell = 2$, $A = 2a$ is even.}

{
\begin{center}
    \begin{tabular}{|c|c|}
     Residue Field & Multiplicity \\  \hline
$\Q(\ff)$ & $2$ \\
$K(\ff)$ & $2^{A-1}-1$ \\
$K(2^2 \ff)$ & $2^{A-3}$ \\
$K(2^4 \ff)$ & $2^{A-4}$ \\
$\vdots$ & $\vdots$ \\
$K(2^{a-6} \ff)$ & $2$ \\
$K(2^{a-4} \ff)$ & $1$ \\
$\Q(2^{a-2} \ff)$ & $1$ \\
$\Q(2^a \ff)$ & 1 
     \end{tabular}
\end{center}
}
\noindent
{\textbf{\sc Primitive Residue Fields:}} $\Q(\ff)$
\\ \\
Again $b = 0$ yields the closed point equivalence class 
$[P_{\downarrow}]$ with residue field $\Q(2^a \ff)$, while $b = a$ yields the path $P_{\uparrow}$, with residue field $\Q(\ff)$.  For all paths with $b=  1$, the initial subpath $P_1$ consists of one ascent followed by one (unique) descent, so $P_1$ is real, and we get one closed point equivalence class of paths with residue field $\Q(2^{a-2} \ff)$.  It follows from Lemmas \ref{LEMMA5.5}, \ref{LEMMA5.7}, 
\ref{LEMMA5.8} and \ref{LEMMA5.9} that for $2 \leq b \leq a-1$ the segment $P_1$ is not real and thus 
the residue field contains $K$.  For $2 \leq b < A$ the path $P_1$ consists of $b$ ascents followed by $b$ descents.  There are 
$2^{b-1}$ such paths, which fall into $2^{b-2}$ closed point equivalence classes.  There are $a-2b$ further descents, so all in 
all there are $2^{b-1}$ closed point equivalence classes of paths, each with residue field $K(2^{a-2b} \ff)$.  

\subsubsection{Case 9: $L \geq a \geq 5$, $\ell = 2$, $a = 2A+1$.}  

{
\begin{center}
    \begin{tabular}{|c|c|}
     Residue Field & Multiplicity \\  \hline
$\Q(\ff)$ & $2$ \\
$K(\ff)$ & $2^{A-1}-1$ \\
$K(2 \ff)$ & $2^{A-2}$ \\
$K(2^3 \ff)$ & $2^{A-3}$ \\
$\vdots$ & $\vdots$ \\
$K(2^{a-6} \ff)$ & $2$ \\
$K(2^{a-4} \ff)$ & $1$ \\
$\Q(2^{a-2} \ff)$ & $1$ \\
$\Q(2^a \ff)$ & 1 
     \end{tabular}
\end{center}
}
\noindent
{\textbf{\sc Primitive Residue Fields:}} $\Q(\ff)$
\\ \\
Again $b = 0$ yields the closed point equivalence class $[P_{\downarrow}]$ with residue field $\Q(2^a \ff)$.  The case 
$b = 1$ yields a unique closed point equivalence class with residue field $\Q(2^{a-2} \ff)$.  For $2 \leq b \leq A$, by the same 
results cited above the second descent means that the residue field contains $K$, and there are $2^{b-2}$ closed point equivalence 
classes of such paths, each with residue field $K(2^{a-2b} \ff)$.  For $A+1 \leq b \leq a-2$ the residue field again contains $K$ and 
there are $2^{a-b-2}$ closed point equivalence classes of such paths, each with residue field $K(\ff)$.

\subsection{Cases 10-12: $L \geq 1$, $a > L$,  $\left( \frac{\Delta_K}{\ell} \right) = -1$, $\ell > 2$}

\subsubsection{Case 10: $L = 1$, $a \geq 2$, $\left(\frac{\Delta_K}{\ell} \right) = -1$, $\ell > 2$}

{
\begin{center}
    \begin{tabular}{|c|c|}
     Residue Field & Multiplicity \\  \hline
$\Q(\ell^{a-2} \ff)$ & $1$ \\
$K(\ell^{a-2} \ff)$ & $\frac{\ell-1}{2}$ \\
$\Q(\ell^a \ff)$ & $1$
     \end{tabular}
\end{center}
}
\noindent
{\textbf{\sc Primitive Residue Fields:}} $\Q(\ell^{a-2} \ff)$
\\ \\
The path types that occur in this case are I. and VI.

\subsubsection{Case 11: $2 \leq L \leq \frac{a}{2}$, $\left( \frac{\Delta_K}{\ell} \right) = -1$, $\ell > 2$}


{
\begin{center}
    \begin{tabular}{|c|c|}
     Residue Field & Multiplicity \\  \hline
$\Q( \ell^{a-2L} \ff)$ & $1$ \\
$K(\ell^{a-2L} \ff)$ & $\frac{\ell^L-1}{2}$ \\
$K( \ell^{a-2L+2} \ff)$ & $\frac{(\ell-1) \ell^{L-2}}{2}$ \\
$K( \ell^{a-2L+4} \ff)$ & $\frac{(\ell-1) \ell^{L-3}}{2}$ \\
$\vdots$ & $\vdots$ \\
$K(\ell^{a-4} \ff)$ & $\frac{(\ell-1)\ell}{2}$ \\
$K( \ell^{a-2} \ff)$ & $\frac{\ell-1}{2}$ \\
$\Q( \ell^a \ff)$ & $1$
     \end{tabular}
\end{center}
}
\noindent
{\textbf{\sc Primitive Residue Fields:}} $\Q(\ell^{a-2L} \ff)$
\\ \\
The path types that occur in this case are I., V. and VI.

\subsubsection{Case 12: $\frac{a}{2} < L < a$, $\left( \frac{\Delta_K}{\ell} \right) = -1$, $\ell > 2$}
Put $\gamma(a) \coloneqq \begin{cases} 1 & a \text{ odd} \\ 2 & a \text{ even} \end{cases}$.  


{
\begin{center}
    \begin{tabular}{|c|c|}
     Residue Field & Multiplicity \\  \hline
$\Q(\ff)$ & $1$ \\
$K(\ff)$ & $\frac{\ell^{\lfloor \frac{a}{2} \rfloor}-1}{2}$\\
$K(\ff \ell^{\epsilon})$ & $\frac{(\ell-1) \ell^{\lfloor \frac{a-1}{2}\rfloor-1}}{2}$ \\
$\vdots$ & $\vdots$ \\
$K(\ff \ell^{a-2})$ & $\frac{\ell-1}{2}$ \\
$\Q(\ff \ell^a)$ & $1$
     \end{tabular}
\end{center}
}
\noindent 
{\textbf{\sc Primitive Residue Fields:}} $\Q(\ff)$
\\ \\
The path types that occur in this case are I., V. and VI.

\subsection{Cases 13-20: $L \geq 1$, $a > L$, $\left(\frac{\Delta_K}{\ell} \right) = 0$, $\ell > 2$}

\subsubsection{Case 13: $L = 1$, $a = 2$, $\left(\frac{\Delta_K}{\ell} \right) = 0$, $\ell > 2$}

{
\begin{center}
    \begin{tabular}{|c|c|}
     Residue Field & Multiplicity \\  \hline
$\Q(\ff)$ & $1$ \\
$K(\ff)$ & $\frac{\ell-1}{2}$ \\
$\Q(\ell^a \ff)$ & $1$ \\
     \end{tabular}
\end{center}
}
\noindent 
{\textbf{\sc Primitive Residue Fields:}} $\Q(\ff)$
\\ \\
The path types that occur in this case are I., VII. and VIII.

\subsubsection{Case 14: $L = 1$, $a = 3$, $\left(\frac{\Delta_K}{\ell} \right) = 0$, $\ell > 2$}

{
\begin{center}
    \begin{tabular}{|c|c|}
     Residue Field & Multiplicity \\  \hline
$\Q(\ff)$ & $1$ \\
$K(\ff)$ & $\frac{\ell-1}{2}$ \\
$K(\ell \ff)$ & $\frac{\ell-1}{2}$ \\
$\Q(\ell^a \ff)$ & $1$ \\
     \end{tabular}
\end{center}
}
\noindent 
{\textbf{\sc Primitive Residue Fields:}} $\Q(\ff)$
\\ \\
The path types that occur in this case are I., VII. and VIII.

\subsubsection{Case 15: $L = 1$, $a \geq 4$, $\left(\frac{\Delta_K}{\ell} \right) = 0$, $\ell > 2$}

{
\begin{center}
    \begin{tabular}{|c|c|}
     Residue Field & Multiplicity \\  \hline
$\Q(\ell^{a-3} \ff)$ & $1$ \\
$K(\ell^{a-3} \ff)$ & $\frac{\ell-1}{2}$ \\
$K(\ell^{a-2} \ff)$ & $\frac{\ell-1}{2}$ \\
$\Q(\ell^a \ff)$ & $1$ \\
     \end{tabular}
\end{center}
}
\noindent 
{\textbf{\sc Primitive Residue Fields:}} $\Q(\ell^{a-3} \ff)$
\\ \\
The path types that occur in this case are I., VII. and VIII.

\subsubsection{Case 16: $L \geq 2$, $a \geq 2L+1$, $\left(\frac{\Delta_K}{\ell} \right) = 0$, $\ell > 2$}

{
\begin{center}
    \begin{tabular}{|c|c|}
     Residue Field & Multiplicity \\  \hline
$\Q(\ell^{a-2L-1} \ff)$ & $1$ \\
$K(\ell^{a-2L-1} \ff)$ & $\frac{\ell^L-1}{2}$ \\
$K(\ell^{a-2L} \ff)$ & $\frac{(\ell-1)\ell^{L-1}}{2}$ \\
$K(\ell^{a-2L+2} \ff)$ & $\frac{(\ell-1)\ell^{L-2}}{2}$ \\
$K(\ell^{a-2L+4} \ff)$ & $\frac{(\ell-1)\ell^{L-3}}{2}$ \\
$\vdots$ & $\vdots$ \\
$K(\ell^{a-2} \ff)$ & $\frac{\ell-1}{2}$ \\
$\Q(\ell^a \ff)$ & $1$ \\
     \end{tabular}
\end{center}
}
\noindent 
{\textbf{\sc Primitive Residue Fields:}} $\Q(\ell^{a-2L-1} \ff)$
\\ \\
The path types that occur in this case are I., V. VII. and VIII.

\subsubsection{Case 17: $L \geq 2$, $a = 2L$, $\left(\frac{\Delta_K}{\ell} \right) = 0$, $\ell > 2$}

{
\begin{center}
    \begin{tabular}{|c|c|}
     Residue Field & Multiplicity \\  \hline
$\Q(\ff)$ & $1$ \\
$K(\ff)$ & $\frac{(\ell-1)\ell^{a/2-1}}{2} + \frac{\ell^{a/2-1}-1}{2}$ \\
$K(\ell^2 \ff)$ & $\frac{(\ell-1)\ell^{a/2-2}}{2}$ \\
$K(\ell^4 \ff)$ & $\frac{(\ell-1)\ell^{a/2-3}}{2}$ \\
$\vdots$ & $\vdots$ \\
$K(\ell^{a-2} \ff)$ & $\frac{\ell-1}{2}$ \\
$\Q(\ell^a \ff)$ & $1$
     \end{tabular}
\end{center}
}
\noindent 
{\textbf{\sc Primitive Residue Fields:}} $\Q(\ff)$
\\ \\
The path types that occur in this case are I., V. VII. and VIII.

\subsubsection{Case 18: $L \geq 2$, $a = 2L-1$, $\left(\frac{\Delta_K}{\ell} \right) = 0$, $\ell > 2$}

{
\begin{center}
    \begin{tabular}{|c|c|}
     Residue Field & Multiplicity \\  \hline
$\Q(\ff)$ & $1$ \\
$K(\ff)$ & $\frac{(\ell-1)\ell^{L-2}}{2}$ \\
$K(\ell \ff)$ & $\frac{(\ell-1)\ell^{L-2}}{2}$ \\
$K(\ell^3 \ff)$ & $\frac{(\ell-1)\ell^{L-3}}{2}$ \\
$\vdots$ & $\vdots$ \\
$K(\ell^{a-2} \ff)$ & $\frac{\ell-1}{2}$ \\
$\Q(\ell^a \ff)$ & $1$  
     \end{tabular}
\end{center}
}
\noindent 
{\textbf{\sc Primitive Residue Fields:}} $\Q(\ff)$
\\ \\
The path types that occur in this case are I., V. VII. and VIII.

\subsubsection{Case 19: $L \geq 2$, $L+1 \leq a \leq 2L-2$, $a$ even, $\left(\frac{\Delta_K}{\ell} \right) = 0$, $\ell > 2$}

{
\begin{center}
    \begin{tabular}{|c|c|}
     Residue Field & Multiplicity \\  \hline
$\Q(\ff)$ & $1$ \\
$K(\ff)$ & $\frac{(\ell-1)(\ell^{a-L-1} + \ell^{a-L} + \ldots + \ell^{a/2-1}}{2} + \frac{\ell^{a-L-1}-1}){2}$ \\
$K(\ell^2 \ff)$ & $\frac{(\ell-1) \ell^{a/2-2}}{2}$ \\
$K(\ell^4 \ff)$ & $\frac{(\ell-1) \ell^{a/2-3}}{2}$ \\
$\vdots$ & $\vdots$ \\
$K(\ell^{a-2} \ff)$ & $\frac{\ell-1}{2}$ \\
$\Q(\ell^a \ff)$ & $1$ 
     \end{tabular}
\end{center}
}
\noindent 
{\textbf{\sc Primitive Residue Fields:}} $\Q(\ff)$
\\ \\
The path types that occur in this case are I., V. VII. and VIII.

\subsubsection{Case 20: $L \geq 2$, $L+1 \leq a \leq 2L-3$, $a$ odd, $\left(\frac{\Delta_K}{\ell} \right) = 0$, $\ell > 2$}

{
\begin{center}
    \begin{tabular}{|c|c|}
     Residue Field & Multiplicity \\  \hline
$\Q(\ff)$ & $1$ \\
$K(\ff)$ & $\frac{(\ell-1)(\ell^{a-L-1} + \ell^{a-L} + \ldots + \ell^{\frac{a-1}{2}-1})}{2} + \frac{\ell^{a-L-1}-1}{2}$ \\
$K(\ell \ff)$ & $\frac{(\ell-1)\ell^{\frac{a-1}{2}-1}}{2}$ \\
$K(\ell^3 \ff)$ & $\frac{(\ell-1)\ell^{\frac{a-1}{2}-2}}{2}$ \\
$\vdots$ & $\vdots$ \\
$K(\ell^{a-2} \ff)$ & $\frac{\ell-1}{2}$ \\
$\Q(\ell^a \ff)$ & $1$
     \end{tabular}
\end{center}
}
\noindent 
{\textbf{\sc Primitive Residue Fields:}} $\Q(\ff)$
\\ \\
The path types that occur in this case are I., V. VII. and VIII.

\subsection{Cases 21-33: $L \geq 1$, $a > L$, $\left(\frac{\Delta_K}{\ell} \right) = 1$, $\ell > 2$}

\subsubsection{Case 21: $L = 1$, $a = 2$, $\left( \frac{\Delta_K}{\ell} \right) = 1$, $\ell > 2$}

{
\begin{center}
    \begin{tabular}{|c|c|}
     Residue Field & Multiplicity \\  \hline
$\Q(\ff)$ & $1$ \\
$K(\ff)$ & $\frac{\ell-3}{2} + 1$ \\
$\Q(\ell^2 \ff)$ & $1$
     \end{tabular}
\end{center}
}
\noindent 
{\textbf{\sc Primitive Residue Fields:}} $\Q(\ff)$
\\ \\
The path types that occur in this case are I., IX. and X.

\subsubsection{Case 22: $L = 1$, $a \geq 3$, $\left( \frac{\Delta_K}{\ell} \right) = 1$, $\ell > 2$}

{
\begin{center}
    \begin{tabular}{|c|c|}
     Residue Field & Multiplicity \\  \hline
$K(\ff)$ & $\ell$ \\
$K(\ell \ff)$ & $\ell-1$ \\
$K(\ell^2 \ff)$ & $\ell-1$ \\
$\vdots$ & $\vdots$ \\
$K(\ell^{a-3} \ff)$ & $\ell-1$ \\
$\Q(\ell^{a-2} \ff)$ & $1$ \\
$K(\ell^{a-2} \ff)$ & $\frac{\ell-3}{2}$ \\
$\Q(\ell^a \ff)$ & $1$
     \end{tabular}
\end{center}
}
\noindent 
{\textbf{\sc Primitive Residue Fields:}} $\Q(\ell^{a-2} \ff)$, $K(\ff)$
\\ \\
The path types that occur in this case are I., IX., X., XI.

\subsubsection{Case 23: $L = 2$, $a = 3$, $\left( \frac{\Delta_K}{\ell} \right) = 1$, $\ell > 2$}

{
\begin{center}
    \begin{tabular}{|c|c|}
     Residue Field & Multiplicity \\  \hline
$\Q(\ff)$ & $1$ \\
$K(\ff)$ & $1 + \frac{\ell-3}{2}$ \\
$K(\ell \ff)$ & $\frac{\ell-1}{2}$ \\
$\Q(\ell^3 \ff)$ & $1$
     \end{tabular}
\end{center}
}
\noindent 
{\textbf{\sc Primitive Residue Fields:}} $\Q(\ff)$
\\ \\
The path types that occur in this case are I., V., IX. and X.

\subsubsection{Case 24: $L = 3$, $a = 4$, $\left( \frac{\Delta_K}{\ell} \right) = 1$, $\ell > 2$}

{
\begin{center}
    \begin{tabular}{|c|c|}
     Residue Field & Multiplicity \\  \hline
$\Q(\ff)$ & $1$ \\
$K(\ff)$ & $1 + \frac{\ell-3}{2} + \frac{(\ell-1)\ell}{2}$ \\
$K(\ell^2 \ff)$ & $\frac{\ell-1}{2}$ \\
$\Q(\ell^4 \ff)$ & $1$
     \end{tabular}
\end{center}
}
\noindent 
{\textbf{\sc Primitive Residue Fields:}} $\Q(\ff)$
\\ \\
The path types that occur in this case are I., V., IX. and X.

\subsubsection{Case 25: $L = 4$, $a = 5$, $\left( \frac{\Delta_K}{\ell} \right) = 1$, $\ell > 2$}

{
\begin{center}
    \begin{tabular}{|c|c|}
     Residue Field & Multiplicity \\  \hline
$\Q(\ff)$ & $1$ \\
$K(\ff)$ & $1 + \frac{\ell-3}{2} + \frac{(\ell-1)\ell}{2}$ \\
$K(\ell \ff)$ & $\frac{\ell-1}{2}$ \\
$K(\ell^3 \ff)$ & $\frac{\ell-1}{2}$ \\
$\Q(\ell^4 \ff)$ & $1$
     \end{tabular}
\end{center}
}
\noindent 
{\textbf{\sc Primitive Residue Fields:}} $\Q(\ff)$
\\ \\
The path types that occur in this case are I., V., IX. and X.

\subsubsection{Case 26: $L \geq 5$ odd, $a = L+1$, $\left( \frac{\Delta_K}{\ell} \right) = 1$, $\ell > 2$}

{
\begin{center}
    \begin{tabular}{|c|c|}
     Residue Field & Multiplicity \\  \hline
$\Q(\ff)$ & $1$ \\
$K(\ff)$ & $1 + \frac{\ell-3}{2} + \frac{(\ell-1)(\ell + \ldots + \ell^{a/2-1})}{2}$ \\
$K(\ell^2 \ff)$ & $\frac{(\ell-1)\ell^{a/2-2}}{2}$ \\
$K(\ell^4 \ff)$ & $\frac{(\ell-1)(\ell^{a/2-3}}{2}$ \\
$\vdots$ & $\vdots$ \\
$K(\ell^{a-2} \ff)$ & $\frac{\ell-1}{2}$  \\
$\Q(\ell^a \ff)$ & $1$
     \end{tabular}
\end{center}
}
\noindent 
{\textbf{\sc Primitive Residue Fields:}} $\Q(\ff)$
\\ \\
The path types that occur in this case are I., V., IX. and X.

\subsubsection{Case 27: $L \geq 6$ even, $a = L+1$, $\left( \frac{\Delta_K}{\ell} \right) = 1$, $\ell > 2$}

{
\begin{center}
    \begin{tabular}{|c|c|}
     Residue Field & Multiplicity \\  \hline
$\Q(\ff)$ & $1$ \\
$K(\ff)$ & $1 + \frac{\ell-3}{2} + \frac{(\ell-1)(\ell + \ldots + \ell^{L/2-1})}{2}$ \\
$K(\ell \ff)$ & $\frac{(\ell-1)\ell^{L/2-1}}{2}$ \\
$K(\ell^3 \ff)$ & $\frac{(\ell-1)\ell^{L/2-2}}{2}$ \\
$\vdots$ & $\vdots$ \\
$K(\ell^{a-2} \ff)$ & $\frac{\ell-1}{2}$ \\
$\Q(\ell^a \ff)$ & $1$
     \end{tabular}
\end{center}
}
\noindent 
{\textbf{\sc Primitive Residue Fields:}} $\Q(\ff)$
\\ \\
The path types that occur in this case are I., V., IX. and X.

\subsubsection{Case 28: $L \geq 2$, $a \geq 2L+2$, $\left( \frac{\Delta_K}{\ell} \right) = 1$, $\ell > 2$}

{
\begin{center}
    \begin{tabular}{|c|c|}
     Residue Field & Multiplicity \\  \hline
$K(\ff)$ & $\ell^L$ \\
$K(\ell \ff)$ & $(\ell-1)\ell^{L-1}$ \\
$K(\ell^2 \ff)$ & $(\ell-1) \ell^{L-1}$ \\
$\vdots$ & $\vdots$ \\
$K(\ell^{a-2L-1} \ff)$ & $(\ell-1)\ell^{L-1}$ \\
$\Q(\ell^{a-2L} \ff)$ & $1$ \\
$K(\ell^{a-2L} \ff)$ & $\frac{(\ell-2)\ell^{L-1}-1}{2}$ \\
$K(\ell^{a-2L+2} \ff)$ & $\frac{(\ell-1)\ell^{L-2}}{2}$ \\
$\vdots$ & $\vdots$ \\
$K(\ell^{a-4} \ff)$ & $\frac{(\ell-1)\ell}{2}$ \\
$K(\ell^{a-2} \ff)$ & $\frac{\ell-1}{2}$ \\
$\Q(\ell^a \ff)$ & $1$
     \end{tabular}
\end{center}
}
\noindent 
{\textbf{\sc Primitive Residue Fields:}} $\Q(\ell^{a-2L} \ff)$, $K(\ff)$
\\ \\
The path types that occur in this case are I., V., IX., X. and XI.

\subsubsection{Case 29: $L \geq 2$, $a = 2L+1$, $\left( \frac{\Delta_K}{\ell} \right) = 1$, $\ell > 2$}

{
\begin{center}
    \begin{tabular}{|c|c|}
     Residue Field & Multiplicity \\  \hline
$K(\ff)$ & $\ell^L$ \\
$\Q(\ell \ff)$ & $1$ \\
$K(\ell \ff)$ & $\frac{(\ell-2)\ell^{L-1}-1}{2}$ \\
$K(\ell^{3} \ff)$ & $\frac{(\ell-1)\ell^{L-2}}{2}$ \\
$\vdots$ & $\vdots$ \\
$K(\ell^{a-4} \ff)$ & $\frac{(\ell-1)\ell}{2}$ \\
$K(\ell^{a-2} \ff)$ & $\frac{\ell-1}{2}$ \\
$\Q(\ell^a \ff)$ & $1$
     \end{tabular}
\end{center}
}
\noindent 
{\textbf{\sc Primitive Residue Fields:}} $\Q(\ell^{a-2L} \ff)$, $K(\ff)$
\\ \\
The path types that occur in this case are I., V., IX., X. and XI.

\subsubsection{Case 30: $L \geq 2$, $a = 2L$, $\left( \frac{\Delta_K}{\ell} \right) = 1$, $\ell > 2$}

{
\begin{center}
    \begin{tabular}{|c|c|}
     Residue Field & Multiplicity \\  \hline
$\Q(\ff)$ & $1$ \\
$K(\ff)$ & $\ell^{L-1} + \frac{(\ell-2)\ell^{L-1}-1}{2}$ \\
$K(\ell^2 \ff)$ & $\frac{(\ell-1)\ell^{L-2}}{2}$ \\
$\vdots$ & $\vdots$ \\
$K(\ell^{a-2}$ & $\frac{\ell-1}{2}$ \\
$\Q(\ell^{a} \ff)$ & $1$
     \end{tabular}
\end{center}
}
\noindent 
{\textbf{\sc Primitive Residue Fields:}} $\Q(\ff)$
\\ \\
The path types that occur in this case are I., V., IX., X. and XI.

\subsubsection{Case 31: $L \geq 2$, $a = 2L-1$, $\left( \frac{\Delta_K}{\ell} \right) = 1$, $\ell > 2$}

{
\begin{center}
    \begin{tabular}{|c|c|}
     Residue Field & Multiplicity \\  \hline
$\Q(\ff)$ & $1$ \\
$K(\ff)$ & $\ell^{L-2}$ + $\frac{(\ell-2)\ell^{L-2}-1}{2}$ \\
$K(\ell \ff)$ & $\frac{(\ell-1)\ell^{L-2}}{2}$ \\
$K(\ell^3 \ff)$ & $\frac{(\ell-1)\ell^{L-3}}{2}$ \\
$\vdots$ & $\vdots$ \\
$K(\ell^{a-2} \ff)$ & $\frac{\ell-1}{2}$ \\
$\Q(\ell^a \ff)$ & $1$
     \end{tabular}
\end{center}
}
\noindent 
{\textbf{\sc Primitive Residue Fields:}} $\Q(\ff)$
\\ \\
The path types that occur in this case are I., V., IX., X. and XI.

\subsubsection{Case 32: $L \geq 2$, $L+2 \leq a \leq 2L-2$, $a$ even, $\left( \frac{\Delta_K}{\ell} \right) = 1$, $\ell > 2$}

{
\begin{center}
    \begin{tabular}{|c|c|}
     Residue Field & Multiplicity \\  \hline
$\Q(\ff)$ & $1$ \\
$K(\ff)$ & $\frac{\ell^{a/2}-1}{2}$ \\
$K(\ell^2 \ff)$ & $\frac{(\ell-1)\ell^{a/2-2}}{2}$ \\
$K(\ell^4 \ff)$ & $\frac{(\ell-1)\ell^{a/2-3}}{2}$ \\
$\vdots$ & $\vdots$ \\
$K(\ell^{a-2} \ff)$ & $\frac{\ell-1}{2}$ \\
$\Q(\ell^a \ff)$ & $1$
     \end{tabular}
\end{center}
}
\noindent 
{\textbf{\sc Primitive Residue Fields:}} $\Q(\ff)$
\\ \\
The path types that occur in this case are I., V., IX., X. and XI.

\subsubsection{Case 33: $L \geq 2$, $L+2 \leq a \leq 2L-3$, $a$ odd, $\left( \frac{\Delta_K}{\ell} \right) = 1$, $\ell > 2$}

{
\begin{center}
    \begin{tabular}{|c|c|}
     Residue Field & Multiplicity \\  \hline
$\Q(\ff)$ & $1$ \\
$K(\ff)$ & $\frac{\ell^{\frac{a-1}{2}}-1}{2}$ \\
$K(\ell \ff)$ & $\frac{(\ell-1)\ell^{\frac{a-1}{2}-1}}{2}$ \\
$K(\ell^3 \ff)$ & $\frac{(\ell-1)\ell^{\frac{a-1}{2}-2}}{2}$ \\
$\vdots$ & $\vdots$ \\
$K(\ell^{a-4} \ff)$ & $\frac{(\ell-1)\ell}{2}$ \\
$K(\ell^{a-2} \ff)$ & $\frac{\ell-1}{2}$ \\
$\Q(\ell^a \ff)$ & $1$
     \end{tabular}
\end{center}
}
\noindent 
{\textbf{\sc Primitive Residue Fields:}} $\Q(\ff)$
\\ \\
The path types that occur in this case are I., V., IX., X. and XI. 

\subsection{Cases 34-43: $L \geq 1$, $a > L$, $\ell = 2$, $\left( \frac{\Delta_K}{2} \right) = -1$}

\subsubsection{Case 34: $L = 1$, $a \geq 2$}

{
\begin{center}
    \begin{tabular}{|c|c|}
     Residue Field & Multiplicity \\  \hline
$K(2^{a-2} \ff)$ & $1$ \\
$\Q(2^a\ff)$ & $1$ 
     \end{tabular}
\end{center}
}
\noindent 
{\textbf{\sc Primitive Residue Fields:}} $\Q(2^a \ff)$, $K(2^{a-2} \ff)$
\\ \\
The path types that occur in this case are I. and VI.

\subsubsection{Case 35: $L = 2$, $a = 3$}

{
\begin{center}
    \begin{tabular}{|c|c|}
     Residue Field & Multiplicity \\  \hline
$K(\ff)$ & $1$ \\
$\Q(2\ff)$ & $1$ \\
$\Q(8 \ff)$ & $1$
     \end{tabular}
\end{center}
}
\noindent 
{\textbf{\sc Primitive Residue Fields:}} $\Q(2 \ff)$, $K(\ff)$
\\ \\
The path types that occur in this case are I., V$_1$. and VI.

\subsubsection{Case 36: $L = 2$, $a \geq 4$}

{
\begin{center}
    \begin{tabular}{|c|c|}
     Residue Field & Multiplicity \\  \hline
$K(2^{a-4} \ff)$ & $2$ \\
$\Q(2^{a-2} \ff)$ & $2$ \\
$\Q(2^a \ff)$ & $1$ 
     \end{tabular}
\end{center}
}
\noindent 
{\textbf{\sc Primitive Residue Fields:}} $\Q(2^{a-2} \ff)$, $K(2^{a-4} \ff)$
\\ \\
The path types that occur in this case are I., V$_1$ and VI.

\subsubsection{Case 37: $L = 3$, $a = 4$}

{
\begin{center}
    \begin{tabular}{|c|c|}
     Residue Field & Multiplicity \\  \hline
$\Q(\ff)$ & $2$ \\
$K(\ff)$ & $1$ \\
$\Q(2^2 \ff)$ & $1$ \\
$\Q(2^4 \ff)$ & $1$
     \end{tabular}
\end{center}
}
\noindent 
{\textbf{\sc Primitive Residue Fields:}} $\Q(\ff)$
\\ \\
The path types that occur in this case are I., V$_1$., V$_3$. and VI.

\subsubsection{Case 38: $L = 3$, $a = 5$}

{
\begin{center}
    \begin{tabular}{|c|c|}
     Residue Field & Multiplicity \\  \hline
$K(\ff)$ & $2$ \\
$\Q(2\ff)$ & $2$ \\
$\Q(2^3 \ff)$ & $1$ \\
$\Q(2^5 \ff)$ & $1$
     \end{tabular}
\end{center}
}
\noindent 
{\textbf{\sc Primitive Residue Fields:}} $\Q(2 \ff)$, $K(\ff)$ 
\\ \\
The path types that occur in this case are I., V$_1$., V$_3$. and VI.

\subsubsection{Case 39: $L = 3$, $a \geq 6$}

{
\begin{center}
    \begin{tabular}{|c|c|}
     Residue Field & Multiplicity \\  \hline
$K(2^{a-6} \ff)$ & $4$ \\
$\Q(2^{a-4} \ff)$ & $2$ \\
$\Q(2^{a-2} \ff)$ & $1$ \\
$\Q(2^a \ff)$ & $1$ 
     \end{tabular}
\end{center}
}
\noindent 
{\textbf{\sc Primitive Residue Fields:}} $\Q(2^{a-4} \ff)$, $K(2^{a-6} \ff)$
\\ \\
The path types that occur in this case are I., V$_1$., V$_3$. and VI.

\subsubsection{Case 40: $4 \leq L < a \leq 2L-6$, $a$ even}

{
\begin{center}
    \begin{tabular}{|c|c|}
     Residue Field & Multiplicity \\  \hline
$\Q(\ff)$ & $2$ \\
$K(\ff)$ & $2^{a/2-1}$ \\
$K(2^2 \ff)$ & $2^{a/2-3}$ \\
$K(2^4 \ff)$ & $2^{a/2-4}$ \\
$\vdots$ & $\vdots$ \\
$K(2^{a-4} \ff)$ & $1$ \\
$\Q(2^{a-2} \ff)$ & $1$ \\
$\Q(2^a \ff)$ & $1$
     \end{tabular}
\end{center}
}
\noindent 
{\textbf{\sc Primitive Residue Fields:}} $\Q(\ff)$
\\ \\
The path types that occur in this case are I., V$_1$., V$_3$. and VI.

\subsubsection{Case 41: $4 \leq L < a \leq 2L-5$, $a$ odd}

{
\begin{center}
    \begin{tabular}{|c|c|}
     Residue Field & Multiplicity \\  \hline
$\Q(\ff)$ & $2$ \\
$K(\ff)$ & $2^{\frac{a-3}{2}}-1$ \\
$K(2\ff)$ & $2^{\frac{a-5}{2}}$ \\
$K(2^3 \ff)$ & $2^{\frac{a-7}{2}}$ \\
$\vdots$ & $\vdots$ \\
$K(2^{a-4} \ff)$ & $1$ \\
$\Q(2^{a-2} \ff)$ & $1$ \\
$\Q(2^a \ff)$ & $1$
     \end{tabular}
\end{center}
}
\noindent 
{\textbf{\sc Primitive Residue Fields:}} $\Q(\ff)$
\\ \\
The path types that occur in this case are I., V$_1$., V$_3$. and VI.

\subsubsection{Case 42: $5 \leq L$, $a = 2L-4$}

{
\begin{center}
    \begin{tabular}{|c|c|}
     Residue Field & Multiplicity \\  \hline
$\Q(\ff)$ & $2$ \\
$K(\ff)$ & $2^{a-L+1}-1$ \\
$K(2^2 \ff)$ & $2^{a-L-1}$ \\
$K(2^4 \ff)$ & $2^{a-L-2}$ \\
$\vdots$ & $\vdots$ \\
$K(2^{a-4} \ff)$ & $1$ \\
$\Q(2^{a-2} \ff)$ & $1$ \\
$\Q(2^a \ff)$ & $1$
     \end{tabular}
\end{center}
}
\noindent 
{\textbf{\sc Primitive Residue Fields:}} $\Q(\ff)$
\\ \\
The path types that occur in this case are I., V$_1$., V$_3$. and VI.

\subsubsection{Case 43: $L \geq 4$, $a \geq 2L$}

{
\begin{center}
    \begin{tabular}{|c|c|}
     Residue Field & Multiplicity \\  \hline
$K(2^{a-2L}\ff)$ & $2^{L-1}$ \\
$\Q(2^{a-2L+2} \ff)$ & $2$ \\
$K(2^{a-2L+2} \ff)$ & $2^{L-3}-1$ \\
$K(2^{a-2L+4} \ff)$ & $2^{L-4}$ \\
$\vdots$ & $\vdots$ \\
$K(2^{a-4} \ff)$ & $1$ \\
$\Q(2^{a-2} \ff)$ & $1$ \\
$\Q(2^a \ff)$ & $1$
     \end{tabular}
\end{center}
}
\noindent 
{\textbf{\sc Primitive Residue Fields:}} $\Q(2^{a-2L+2} \ff)$, $K(2^{a-2L} \ff)$
\\ \\
The path types that occur in this case are I., V$_1$., V$_3$., V$_4$. and VI.

\subsection{Cases 44-64: $L \geq 1$, $a > L$, $\ell = 2$, $\left( \frac{\Delta_K}{2} \right) = 1$}

\subsubsection{Case 44: $L = 1$, $a = 2$}

{
\begin{center}
    \begin{tabular}{|c|c|}
     Residue Field & Multiplicity \\  \hline
$K(\ff)$ & $2$ \\
$\Q(4\ff)$ & $1$ 
     \end{tabular}
\end{center}
}
\noindent 
{\textbf{\sc Primitive Residue Fields:}} $\Q(4\ff)$, $K(\ff)$
\\ \\
The path types that occur in this case are I. and X.

\subsubsection{Case 45: $L = 1$, $a = 3$}

{
\begin{center}
    \begin{tabular}{|c|c|}
     Residue Field & Multiplicity \\  \hline
$K(\ff)$ & $2$ \\
$\Q(8 \ff)$ & $1$
     \end{tabular}
\end{center}
}
\noindent 
{\textbf{\sc Primitive Residue Fields:}} $\Q(8\ff)$, $K(\ff)$
\\ \\
The path types that occur in this case are I. and X. and XI.

\subsubsection{Case 46: $L = 1$, $a \geq 4$}

{
\begin{center}
    \begin{tabular}{|c|c|}
     Residue Field & Multiplicity \\  \hline
$K(\ff)$ & $2$ \\
$K(2 \ff)$ & $1$ \\
$\vdots$ & $\vdots$ \\
$K(2^{a-3} \ff)$ & $1$ \\
$\Q(2^a \ff)$ & $1$
     \end{tabular}
\end{center}
}
\noindent 
{\textbf{\sc Primitive Residue Fields:}} $\Q(2^a \ff)$, $K(\ff)$
\\ \\
The path types that occur in this case are I. and X. and XI.  

\subsubsection{Case 47: $L = 2$, $a = 3$}

{
\begin{center}
    \begin{tabular}{|c|c|}
     Residue Field & Multiplicity \\  \hline
$K(\ff)$ & $1$ \\
$\Q(2 \ff)$ & $1$ \\
$\Q(2^3 \ff)$ & $1$
     \end{tabular}
\end{center}
}
\noindent 
{\textbf{\sc Primitive Residue Fields:}} $\Q(2\ff)$, $K(\ff)$
\\ \\
The path types that occur in this case are I., V$_1$. and X.

\subsubsection{Case 47: $L = 2$, $a = 4$}

{
\begin{center}
    \begin{tabular}{|c|c|}
     Residue Field & Multiplicity \\  \hline
$K(\ff)$ & $2$ \\
$\Q(2^2 \ff)$ & $1$ \\
$\Q(2^4 \ff)$ & $1$
     \end{tabular}
\end{center}
}
\noindent 
{\textbf{\sc Primitive Residue Fields:}} $\Q(2^2\ff)$, $K(\ff)$
\\ \\
The path types that occur in this case are I., V$_1$, X. and XI.  

\subsubsection{Case 48: $L = 2$, $a = 5$}

{
\begin{center}
    \begin{tabular}{|c|c|}
     Residue Field & Multiplicity \\  \hline
$K(\ff)$ & $4$ \\
$\Q(2^3 \ff)$ & $1$ \\
$\Q(2^5 \ff)$ & $1$
     \end{tabular}
\end{center}
}
\noindent 
{\textbf{\sc Primitive Residue Fields:}} $\Q(2^3\ff)$, $K(\ff)$
\\ \\
The path types that occur in this case are I., V$_1$, X. and XI.  

\subsubsection{Case 49: $L = 2$, $a \geq 6$}

{
\begin{center}
    \begin{tabular}{|c|c|}
     Residue Field & Multiplicity \\  \hline
$K(\ff)$ & $4$ \\
$K(2 \ff)$ & $2$ \\
$\vdots$ & $\vdots$ \\
$K(2^{a-5} \ff)$ & $2$ \\
$\Q(2^{a-2} \ff)$ & $1$ \\
$\Q(2^a \ff)$ & $1$
     \end{tabular}
\end{center}
}
\noindent 
{\textbf{\sc Primitive Residue Fields:}} $\Q(2^{a-2}\ff)$, $K(\ff)$
\\ \\
The path types that occur in this case are I., V$_1$, X. and XI.  

\subsubsection{Case 50: $L = 3$, $a =4$}

{
\begin{center}
    \begin{tabular}{|c|c|}
     Residue Field & Multiplicity \\  \hline
$\Q(\ff)$ & $2$ \\
$K(\ff)$ & $1$ \\
$\Q(2^2 \ff)$ & $1$ \\
$\Q(2^4 \ff)$ & $1$
     \end{tabular}
\end{center}
}
\noindent 
{\textbf{\sc Primitive Residue Fields:}} $\Q(\ff)$
\\ \\
The path types that occur in this case are I., V$_1$, V$_3$., X. and XI.

\subsubsection{Case 51: $L = 3$, $a =5$}

{
\begin{center}
    \begin{tabular}{|c|c|}
     Residue Field & Multiplicity \\  \hline
$K(\ff)$ & $2$ \\
$\Q(2 \ff)$ & $2$ \\
$\Q(2^3 \ff)$ & $1$ \\
$\Q(2^5 \ff)$ & $1$
     \end{tabular}
\end{center}
}
\noindent 
{\textbf{\sc Primitive Residue Fields:}} $\Q(2\ff)$, $K(\ff)$
\\ \\
The path types that occur in this case are I., V$_1$, V$_3$., X. and XI.

\subsubsection{Case 52: $L = 3$, $a =6$}

{
\begin{center}
    \begin{tabular}{|c|c|}
     Residue Field & Multiplicity \\  \hline
$K(\ff)$ & $4$ \\
$\Q(2^2 \ff)$ & $2$ \\
$\Q(2^4 \ff)$ & $1$ \\
$\Q(2^6 \ff)$ & $1$
     \end{tabular}
\end{center}
}
\noindent 
{\textbf{\sc Primitive Residue Fields:}} $\Q(2^2\ff)$, $K(\ff)$
\\ \\
The path types that occur in this case are I., V$_1$, V$_3$., X. and XI.

\subsubsection{Case 53: $L = 3$, $a =7$}

{
\begin{center}
    \begin{tabular}{|c|c|}
     Residue Field & Multiplicity \\  \hline
$K(\ff)$ & $8$ \\
$\Q(2^3 \ff)$ & $2$ \\
$\Q(2^5 \ff)$ & $1$ \\
$\Q(2^7 \ff)$ & $1$
     \end{tabular}
\end{center}
}
\noindent 
{\textbf{\sc Primitive Residue Fields:}} $\Q(2^3\ff)$, $K(\ff)$
\\ \\
The path types that occur in this case are I., V$_1$, V$_3$., X. and XI.

\subsubsection{Case 54: $L = 3$, $a \geq 8$}

{
\begin{center}
    \begin{tabular}{|c|c|}
     Residue Field & Multiplicity \\  \hline
$K(\ff)$ & $8$ \\
$K(2 \ff)$ & $4$ \\
$\vdots$ & $\vdots$ \\
$K(2^{a-7} \ff)$ & $4$ \\
$\Q(2^{a-4} \ff)$ & $2$ \\
$\Q(2^{a-2} \ff)$ & $1$ \\
$\Q(2^a \ff)$ & $1$
     \end{tabular}
\end{center}
}
\noindent 
{\textbf{\sc Primitive Residue Fields:}} $\Q(2^{a-4}\ff)$, $K(\ff)$
\\ \\
The path types that occur in this case are I., V$_1$, V$_3$., X. and XI.

\subsubsection{Case 55: $4 \leq L < a \leq 2L-6$, $a$ even}

{
\begin{center}
    \begin{tabular}{|c|c|}
     Residue Field & Multiplicity \\  \hline
$\Q(\ff)$ & $2$ \\
$K(\ff)$ & $2^{\frac{a}{2}-1}-1$ \\
$K(2^2 \ff)$ & $2^{\frac{a}{2}-3}$ \\
$K(2^4 \ff)$ & $2^{\frac{a}{2}-4}$ \\
$\vdots$ & $\vdots$ \\
$K(2^{a-4} \ff)$ & $1$ \\
$\Q(2^{a-2} \ff)$ & $1$ \\
$\Q(2^a \ff)$ & $1$
     \end{tabular}
\end{center}
}
\noindent 
{\textbf{\sc Primitive Residue Fields:}} $\Q(\ff)$
\\ \\
The path types that occur in this case are I., V$_1$, V$_3$., V$_4$., X. and XI.

\subsubsection{Case 56: $4 \leq L < a \leq 2L-5$, $a$ odd}

{
\begin{center}
    \begin{tabular}{|c|c|}
     Residue Field & Multiplicity \\  \hline
$\Q(\ff)$ & $2$ \\
$K(\ff)$ & $2^{\frac{a-1}{2}-1}-1$ \\
$K(2 \ff)$ & $2^{\frac{a-1}{2}-2}$ \\
$K(2^3 \ff)$ & $2^{\frac{a-1}{2}-3}$ \\
$\vdots$ & $\vdots$ \\
$K(2^{a-4} \ff)$ & $1$ \\
$\Q(2^{a-2} \ff)$ & $1$ \\
$\Q(2^a \ff)$ & $1$
     \end{tabular}
\end{center}
}
\noindent 
{\textbf{\sc Primitive Residue Fields:}} $\Q(\ff)$
\\ \\
The path types that occur in this case are I., V$_1$, V$_3$., V$_4$., X. and XI.

\subsubsection{Case 57: $L \geq 4$, $a = 2L-4$}

{
\begin{center}
    \begin{tabular}{|c|c|}
     Residue Field & Multiplicity \\  \hline
$\Q(\ff)$ & $2$ \\
$K(\ff)$ & $2^{L-3}-1$ \\
$K(2^2 \ff)$ & $2^{L-5}$ \\
$K(2^4 \ff)$ & $2^{L-6}$ \\
$\vdots$ & $\vdots$ \\
$K(2^{a-4} \ff)$ & $1$ \\
$\Q(2^{a-2} \ff)$ & $1$ \\
$\Q(2^a \ff)$ & $1$
     \end{tabular}
\end{center}
}
\noindent 
{\textbf{\sc Primitive Residue Fields:}} $\Q(\ff)$
\\ \\
The path types that occur in this case are I., V$_1$, V$_3$., V$_4$., X. and XI.

\subsubsection{Case 58: $L \geq 4$, $a = 2L-3$}

{
\begin{center}
    \begin{tabular}{|c|c|}
     Residue Field & Multiplicity \\  \hline
$\Q(\ff)$ & $2$ \\
$K(\ff)$ & $2^{L-3}-1$ \\
$K(2 \ff)$ & $2^{L-4}$ \\
$\vdots$ & $\vdots$ \\
$K(2^{a-6} \ff)$ & $2$ \\
$K(2^{a-4} \ff)$ & $1$ \\
$\Q(2^{a-2} \ff)$ & $1$ \\
$\Q(2^a \ff)$ & $1$
     \end{tabular}
\end{center}
}
\noindent 
{\textbf{\sc Primitive Residue Fields:}} $\Q(\ff)$
\\ \\
The path types that occur in this case are I., V$_1$, V$_3$., V$_4$., X. and XI.  

\subsubsection{Case 59: $L \geq 4$, $a = 2L-2$}

{
\begin{center}
    \begin{tabular}{|c|c|}
     Residue Field & Multiplicity \\  \hline
$\Q(\ff)$ & $2$ \\
$K(\ff)$ & $2^{L-2}-1$ \\
$K(2^2 \ff)$ & $2^{L-4}$ \\
$\vdots$ & $\vdots$ \\
$K(2^{a-6} \ff)$ & $2$ \\
$K(2^{a-4} \ff)$ & $1$ \\
$\Q(2^{a-2} \ff)$ & $1$ \\
$\Q(2^a \ff)$ & $1$  
     \end{tabular}
\end{center}
}
\noindent 
{\textbf{\sc Primitive Residue Fields:}} $\Q(\ff)$
\\ \\
The path types that occur in this case are I., V$_1$, V$_3$., V$_4$., X. and XI.  

\subsubsection{Case 60: $L \geq 4$, $a = 2L-2$}

{
\begin{center}
    \begin{tabular}{|c|c|}
     Residue Field & Multiplicity \\  \hline
$\Q(\ff)$ & $2$ \\
$K(\ff)$ & $2^{L-2}-1$ \\
$K(2^2 \ff)$ & $2^{L-4}$ \\
$K(2^4 \ff)$ & $2^{L-5}$ \\
$\vdots$ & $\vdots$ \\
$K(2^{a-6} \ff)$ & $2$ \\
$K(2^{a-4} \ff)$ & $1$ \\
$\Q(2^{a-2} \ff)$ & $1$ \\
$\Q(2^a \ff)$ & $1$
     \end{tabular}
\end{center}
}
\noindent 
{\textbf{\sc Primitive Residue Fields:}} $\Q(\ff)$
\\ \\
The path types that occur in this case are I., V$_1$, V$_3$., V$_4$., X. and XI.

\subsubsection{Case 61: $L \geq 4$, $a = 2L-1$}

{
\begin{center}
    \begin{tabular}{|c|c|}
     Residue Field & Multiplicity \\  \hline
$K(\ff)$ & $2^{L-2}$ \\
$\Q(2\ff)$ & $2$ \\
$K(2\ff)$ & $2^{L-3}-1$ \\
$K(2^3 \ff)$ & $2^{L-4}$ \\
$K(2^5 \ff)$ & $2^{L-5}$ \\
$\vdots$ & $\vdots$ \\
$K(2^{a-4} \ff)$ & $1$ \\
$\Q(2^{a-2} \ff)$ & $1$ \\
$\Q(2^a \ff)$ & $1$
     \end{tabular}
\end{center}
}
\noindent 
{\textbf{\sc Primitive Residue Fields:}} $\Q(2 \ff)$, $K(\ff)$
\\ \\
The path types that occur in this case are I., V$_1$, V$_3$., V$_4$., X. and XI.

\subsubsection{Case 62: $L \geq 4$, $a = 2L$}

{
\begin{center}
    \begin{tabular}{|c|c|}
     Residue Field & Multiplicity \\  \hline
$K(\ff)$ & $2^{L-1}$ \\
$\Q(2^2 \ff)$ & $2$ \\
$K(2^2 \ff)$ & $2^{L-3}-1$ \\
$K(2^4 \ff)$ & $2^{L-4}$ \\
$K(2^6 \ff)$ & $2^{L-5}$ \\
$\vdots$ & $\vdots$ \\
$K(2^{a-4} \ff)$ & $1$ \\
$\Q(2^{a-2} \ff)$ & $1$ \\
$\Q(2^a \ff)$ & $1$
     \end{tabular}
\end{center}
}
\noindent 
{\textbf{\sc Primitive Residue Fields:}} $\Q(2^2 \ff)$, $K(\ff)$
\\ \\
The path types that occur in this case are I., V$_1$, V$_3$., V$_4$., X. and XI.

\subsubsection{Case 63: $L \geq 4$, $a = 2L+1$}

{
\begin{center}
    \begin{tabular}{|c|c|}
     Residue Field & Multiplicity \\  \hline
$K(\ff)$ & $2^L$ \\
$\Q(2^3 \ff)$ & $2$ \\
$K(2^3 \ff)$ & $2^{L-3}-1$ \\
$K(2^5 \ff)$ & $2^{L-4}$ \\
$K(2^7 \ff)$ & $2^{L-5}$ \\
$\vdots$ & $\vdots$ \\
$K(2^{a-4} \ff)$ & $1$ \\
$\Q(2^{a-2} \ff)$ & $1$ \\
$\Q(2^a \ff)$ & $1$
     \end{tabular}
\end{center}
}
\noindent 
{\textbf{\sc Primitive Residue Fields:}} $\Q(2^3 \ff)$, $K(\ff)$
\\ \\
The path types that occur in this case are I., V$_1$, V$_3$., V$_4$., X. and XI.

\subsubsection{Case 64: $L \geq 4$, $a \geq 2L+2$}

{
\begin{center}
    \begin{tabular}{|c|c|}
     Residue Field & Multiplicity \\  \hline
$K(\ff)$ & $2^L$ \\
$K(2\ff)$ & $2^{L-1}$ \\
$K(2^2 \ff)$ & $2^{L-1}$ \\
$\vdots$ & $\vdots$ \\
$K(2^{2-2L-1} \ff)$ & $2^{L-1}$ \\
$\Q(2^{a-2L+2} \ff)$ & $2$ \\
$K(2^{a-2L+2} \ff)$ & $2^{L-3}-1$ \\
$K(2^{a-2L+4} \ff)$ & $2^{L-4}$ \\
$K(2^{a-2L+6} \ff)$ & $2^{L-5}$ \\
$\vdots$ & $\vdots$ \\
$K(2^{a-4} \ff)$ & $1$ \\
$\Q(2^{a-2} \ff)$ & $1$ \\
$\Q(2^a \ff)$ & $1$
     \end{tabular}
\end{center}
}
\noindent 
{\textbf{\sc Primitive Residue Fields:}} $\Q(2^{a-2L+2} \ff)$, $K(\ff)$
\\ \\
The path types that occur in this case are I., V$_1$., V$_3$., V$_4$., X. and XI.

\subsection{Cases 65-79: $L \geq 1$, $a > L$, $\ell = 2$, $\ord_2(\Delta_K) = 2$}

\subsubsection{Case 65: $L = 1$, $a = 2$}

{
\begin{center}
    \begin{tabular}{|c|c|}
     Residue Field & Multiplicity \\  \hline
$\Q(\ff)$ & $2$ \\
$\Q(2^2 \ff)$ & $1$
     \end{tabular}
\end{center}
}
\noindent 
{\textbf{\sc Primitive Residue Fields:}} $\Q(\ff)$
\\ \\
The path types that occur in this case are I., VI$_1$. and VIII$_1$.

\subsubsection{Case 66: $L = 1$, $a \geq 3$}

{
\begin{center}
    \begin{tabular}{|c|c|}
     Residue Field & Multiplicity \\  \hline
$K(2^{a-3} \ff)$ & $1$ \\
$\Q(2^{a-2} \ff)$ & $1$ \\
$\Q(2^a \ff)$ & $1$

     \end{tabular}
\end{center}
}
\noindent 
{\textbf{\sc Primitive Residue Fields:}} $\Q(2^{a-2} \ff)$, $K(2^{a-3} \ff)$
\\ \\
The path types that occur in this case are I., VI$_1$. and VIII$_2$.

\subsubsection{Case 67: $L = 2$, $a = 3$}

{
\begin{center}
    \begin{tabular}{|c|c|}
     Residue Field & Multiplicity \\  \hline
$\Q(\ff)$ & $2$ \\
$\Q(2 \ff)$ & $1$ \\
$\Q(2^3 \ff)$ & $1$
     \end{tabular}
\end{center}
}
\noindent 
{\textbf{\sc Primitive Residue Fields:}} $\Q(\ff)$
\\ \\
The path types that occur in this case are I., V$_1$., VI$_2$. and VIII$_1$.

\subsubsection{Case 68: $L = 2$, $a = 4$}

{
\begin{center}
    \begin{tabular}{|c|c|}
     Residue Field & Multiplicity \\  \hline
$\Q(\ff)$ & $2$ \\
$K(\ff)$  & $1$ \\
$\Q(2^2 \ff)$ & $1$ \\
$\Q(2^4 \ff)$ & $1$
     \end{tabular}
\end{center}
}
\noindent 
{\textbf{\sc Primitive Residue Fields:}} $\Q(\ff)$
\\ \\
The path types that occur in this case are I., V$_1$., VI$_3$. and VIII$_2$.

\subsubsection{Case 69: $L = 2$, $a \geq 5$}

{
\begin{center}
    \begin{tabular}{|c|c|}
     Residue Field & Multiplicity \\  \hline
$K(2^{a-5} \ff)$ & $2$ \\
$\Q(2^{a-4} \ff)$ & $2$ \\
$\Q(2^{a-2} \ff)$ & $1$ \\
$\Q(2^a \ff)$ & $1$
     \end{tabular}
\end{center}
}
\noindent 
{\textbf{\sc Primitive Residue Fields:}} $\Q(2^{a-4} \ff)$, $K(2^{a-5} \ff)$
\\ \\
The path types that occur in this case are I., V$_1$., VI$_3$. and VIII$_2$.

\subsubsection{Case 70: $L = 3$, $a =4$}

{
\begin{center}
    \begin{tabular}{|c|c|}
     Residue Field & Multiplicity \\  \hline
$\Q(\ff)$ & $2$ \\
$K(\ff)$ & $1$ \\
$\Q(2^2 \ff)$ & $1$ \\
$\Q(2^4 \ff)$ & $1$
     \end{tabular}
\end{center}
}
\noindent 
{\textbf{\sc Primitive Residue Fields:}} $\Q(\ff)$
\\ \\
The path types that occur in this case are I., V$_1$., V$_3$., VI$_2$., VIII$_1$.

\subsubsection{Case 71: $L = 3$, $a =5$}

{
\begin{center}
    \begin{tabular}{|c|c|}
     Residue Field & Multiplicity \\  \hline
$\Q(\ff)$ & $2$ \\
$K(\ff)$ & $1$ \\
$K(2 \ff)$ & $1$ \\
$\Q(2^3 \ff)$ & $1$ \\
$\Q(2^5 \ff)$ & $1$
     \end{tabular}
\end{center}
}
\noindent 
{\textbf{\sc Primitive Residue Fields:}} $\Q(\ff)$
\\ \\
The path types that occur in this case are I., V$_1$., V$_3$., VI$_3$, VIII$_2$.

\subsubsection{Case 72: $L = 3$, $a =6$}

{
\begin{center}
    \begin{tabular}{|c|c|}
     Residue Field & Multiplicity \\  \hline
$\Q(\ff)$ & $2$ \\
$K(\ff)$ & $3$ \\
$K(2^2 \ff)$ & $1$ \\
$\Q(2^4 \ff)$ & $1$ \\
$\Q(2^6 \ff)$ & $1$
     \end{tabular}
\end{center}
}
\noindent 
{\textbf{\sc Primitive Residue Fields:}} $\Q(\ff)$
\\ \\
The path types that occur in this case are I., V$_1$., V$_3$., VI$_3$, VIII$_2$.

\subsubsection{Case 73: $L \geq 3$, $a \geq 2L+1$}

{
\begin{center}
    \begin{tabular}{|c|c|}
     Residue Field & Multiplicity \\  \hline
$K(2^{a-2L-1} \ff)$ & $2^{L-1}$ \\
$\Q(2^{a-2L} \ff)$ & $2$ \\
$K(2^{a-2L} \ff)$ & $2^{L-2}-1$ \\
$K(2^{a-2L+2} \ff)$ & $2^{L-3}$ \\
$K(2^{a-2L+4} \ff)$ & $2^{L-2}$ \\
$\vdots$ & $\vdots$ \\
$K(2^{a-6} \ff)$ & $2$ \\
$K(2^{a-4} \ff)$ & $1$ \\
$\Q(2^{a-2} \ff)$ & $1$ \\
$\Q(2^a \ff)$ & $1$
     \end{tabular}
\end{center}
}
\noindent 
{\textbf{\sc Primitive Residue Fields:}} $\Q(2^{a-2L} \ff)$, $K(2^{a-2L-1} \ff)$ 
\\ \\
The path types that occur in this case are I., V$_1$., V$_3$., VI$_3$, VIII$_2$.

\subsubsection{Case 74: $L \geq 4$, $a = L+1$, $a$ odd}

{
\begin{center}
    \begin{tabular}{|c|c|}
     Residue Field & Multiplicity \\  \hline
$\Q(\ff)$ & $2$ \\
$K(\ff)$ & $2^{\frac{a-3}{2}}-1$ \\
$K(2 \ff)$ & $2^{\frac{a-5}{2}}$ \\
$K(2^3 \ff)$ & $2^{\frac{a-7}{2}}$ \\
$\vdots$ & $\vdots$ \\
$\vdots$ & $\vdots$ \\
$K(2^{a-4} \ff)$ & $1$ \\
$\Q(2^{a-2} \ff)$ & $1$ \\
$\Q(2^{a} \ff)$ & $1$
     \end{tabular}
\end{center}
}
\noindent 
{\textbf{\sc Primitive Residue Fields:}} $\Q(\ff)$
\\ \\
The path types that occur in this case are I., V$_1$., V$_3$., VI$_2$, VIII$_1$.

\subsubsection{Case 75: $L \geq 5$, $a = L+1$, $a$ even}

{
\begin{center}
    \begin{tabular}{|c|c|}
     Residue Field & Multiplicity \\  \hline
$\Q(\ff)$ & $2$ \\
$K(\ff)$ & $2^{\frac{a}{2}-1}-1$ \\
$K(2^2 \ff)$ & $2^{\frac{a}{2}-3}$ \\
$\vdots$ & $\vdots$ \\
$K(2^{a-4} \ff)$ & $1$ \\
$\Q(2^{a-2} \ff)$ & $1$ \\
$\Q(2^{a} \ff)$ & $1$
     \end{tabular}
\end{center}
}
\noindent 
{\textbf{\sc Primitive Residue Fields:}} $\Q(\ff)$
\\ \\
The path types that occur in this case are I., V$_1$., V$_3$., VI$_2$, VIII$_1$.

\subsubsection{Case 76: $L \geq 4$, $a = 2L$}

{
\begin{center}
    \begin{tabular}{|c|c|}
     Residue Field & Multiplicity \\  \hline
$\Q(\ff)$ & $2$ \\
$K(\ff)$ & $2^{L-1}-1$ \\
$K(2^2 \ff)$ & $2^{L-3}$ \\
$K(2^4 \ff)$ & $2^{L-4}$ \\
$\vdots$ & $\vdots$ \\
$K(2^{a-4} \ff)$ & $1$ \\
$\Q(2^{a-2} \ff)$ & $1$ \\
$\Q(2^a \ff)$ & $1$
     \end{tabular}
\end{center}
}
\noindent 
{\textbf{\sc Primitive Residue Fields:}} $\Q(\ff)$
\\ \\
The path types that occur in this case are I., V$_1$., V$_3$., VI$_3$, VIII$_2$.

\subsubsection{Case 77: $L \geq 4$, $a = 2L-1$}

{
\begin{center}
    \begin{tabular}{|c|c|}
     Residue Field & Multiplicity \\  \hline
$\Q(\ff)$ & $2$ \\
$K(\ff)$ & $2^{L-2}-1$ \\
$K(2 \ff)$ & $2^{L-3}$ \\
$K(2^3 \ff)$ & $2^{L-4}$ \\
$\vdots$ & $\vdots$ \\
$K(2^{a-4} \ff)$ & $1$ \\
$\Q(2^{a-2} \ff)$ & $1$ \\
$\Q(2^a \ff)$ & $1$
     \end{tabular}
\end{center}
}
\noindent 
{\textbf{\sc Primitive Residue Fields:}} $\Q(\ff)$
\\ \\
The path types that occur in this case are I., V$_1$., V$_3$., VI$_3$, VIII$_2$.

\subsubsection{Case 77: $L \geq 4$, $a = 2L-2$}

{
\begin{center}
    \begin{tabular}{|c|c|}
     Residue Field & Multiplicity \\  \hline
$\Q(\ff)$ & $2$ \\
$K(\ff)$ & $2^{L-2}-1$ \\
$K(2^2 \ff)$ & $2^{L-4}$ \\
$K(2^4 \ff)$ & $2^{L-5}$ \\
$\vdots$ & $\vdots$ \\
$K(2^{a-4} \ff)$ & $1$ \\
$\Q(2^{a-2} \ff)$ & $1$ \\
$\Q(2^a \ff)$ & $1$
     \end{tabular}
\end{center}
}
\noindent 
{\textbf{\sc Primitive Residue Fields:}} $\Q(\ff)$
\\ \\
The path types that occur in this case are I., V$_1$., V$_3$., VI$_3$, VIII$_2$.

\subsubsection{Case 78: $L \geq 4$, $L+2 \leq a \leq 2L-3$, $a$ odd}

{
\begin{center}
    \begin{tabular}{|c|c|}
     Residue Field & Multiplicity \\  \hline
$\Q(\ff)$ & $2$ \\
$K(\ff)$ & $2^{\frac{a-1}{2}-1}-1$ \\
$K(2 \ff)$ & $2^{\frac{a-1}{2}-2}$ \\
$\vdots$ & $\vdots$ \\
$K(2^{a-4} \ff)$ & $1$ \\
$\Q(2^{a-2} \ff)$ & $1$ \\
$\Q(2^a \ff)$ & $1$
     \end{tabular}
\end{center}
}
\noindent 
{\textbf{\sc Primitive Residue Fields:}} $\Q(\ff)$
\\ \\
The path types that occur in this case are I., V$_1$., V$_3$., VI$_3$, VIII$_2$.

\subsubsection{Case 79: $L \geq 6$, $L+2 \leq a \leq 2L-4$, $a$ even}

{
\begin{center}
    \begin{tabular}{|c|c|}
     Residue Field & Multiplicity \\  \hline
$\Q(\ff)$ & $2$ \\
$K(\ff)$ & $2^{\frac{a}{2}-1}-1$ \\
$K(2^2 \ff)$ & $2^{\frac{a}{2}-3}$ \\
$\vdots$ & $\vdots$ \\
$K(2^{a-4} \ff)$ & $1$ \\
$\Q(2^{a-2} \ff)$ & $1$ \\
$\Q(2^a \ff)$ & $1$
     \end{tabular}
\end{center}
}
\noindent 
{\textbf{\sc Primitive Residue Fields:}} $\Q(\ff)$
\\ \\
The path types that occur in this case are I., V$_1$., V$_3$., VI$_3$, VIII$_2$.

\subsection{Cases 80-95: $L \geq 1$, $a > L$, $\ell = 2$, $\ord_2(\Delta_K) = 3$}

\subsubsection{Case 80: $L = 1$, $a = 2$}

{
\begin{center}
    \begin{tabular}{|c|c|}
     Residue Field & Multiplicity \\  \hline
$\Q(\ff)$ & $2$ \\
$\Q(2^2 \ff)$ & $1$
     \end{tabular}
\end{center}
}
\noindent 
{\textbf{\sc Primitive Residue Fields:}} $\Q(\ff)$
\\ \\
The path types that occur in this case are I., VI$_1$. and VIII$_1$.

\subsubsection{Case 81: $L = 1$, $a \geq 3$}

{
\begin{center}
    \begin{tabular}{|c|c|}
     Residue Field & Multiplicity \\  \hline
$\Q(2^{a-3} \ff)$ & $2$ \\
$\Q(2^{a-2} \ff)$ & $1$ \\
$\Q(2^a \ff)$ & $1$

     \end{tabular}
\end{center}
}
\noindent 
{\textbf{\sc Primitive Residue Fields:}} $\Q(2^{a-3} \ff)$
\\ \\
The path types that occur in this case are I., VI$_1$. and VIII$_2$.

\subsubsection{Case 82: $L = 2$, $a = 3$}

{
\begin{center}
    \begin{tabular}{|c|c|}
     Residue Field & Multiplicity \\  \hline
$\Q(\ff)$ & $2$ \\
$\Q(2 \ff)$ & $1$ \\
$\Q(2^3 \ff)$ & $1$
     \end{tabular}
\end{center}
}
\noindent 
{\textbf{\sc Primitive Residue Fields:}} $\Q(\ff)$
\\ \\
The path types that occur in this case are I., V$_1$., VI$_2$. and VIII$_1$.

\subsubsection{Case 83: $L = 2$, $a = 4$}

{
\begin{center}
    \begin{tabular}{|c|c|}
     Residue Field & Multiplicity \\  \hline
$\Q(\ff)$ & $2$ \\
$K(\ff)$  & $1$ \\
$\Q(2^2 \ff)$ & $1$ \\
$\Q(2^4 \ff)$ & $1$
     \end{tabular}
\end{center}
}
\noindent 
{\textbf{\sc Primitive Residue Fields:}} $\Q(\ff)$
\\ \\
The path types that occur in this case are I., V$_1$., VI$_3$. and VIII$_2$.

\subsubsection{Case 84: $L = 2$, $a \geq 5$}

{
\begin{center}
    \begin{tabular}{|c|c|}
     Residue Field & Multiplicity \\  \hline
$\Q(2^{a-5} \ff)$ & $2$ \\
$K(2^{a-5} \ff)$ & $1$ \\
$K(2^{a-4} \ff)$ & $1$ \\
$\Q(2^{a-2} \ff)$ & $1$ \\
$\Q(2^a \ff)$ & $1$
     \end{tabular}
\end{center}
}
\noindent 
{\textbf{\sc Primitive Residue Fields:}} $\Q(2^{a-5} \ff)$
\\ \\
The path types that occur in this case are I., V$_1$., VI$_3$. and VIII$_2$.

\subsubsection{Case 85: $L = 3$, $a =4$}

{
\begin{center}
    \begin{tabular}{|c|c|}
     Residue Field & Multiplicity \\  \hline
$\Q(\ff)$ & $2$ \\
$K(\ff)$ & $1$ \\
$\Q(2^2 \ff)$ & $1$ \\
$\Q(2^4 \ff)$ & $1$
     \end{tabular}
\end{center}
}
\noindent 
{\textbf{\sc Primitive Residue Fields:}} $\Q(\ff)$
\\ \\
The path types that occur in this case are I., V$_1$., V$_3$., VI$_2$., VIII$_1$.

\subsubsection{Case 86: $L = 3$, $a =5$}

{
\begin{center}
    \begin{tabular}{|c|c|}
     Residue Field & Multiplicity \\  \hline
$\Q(\ff)$ & $2$ \\
$K(\ff)$ & $1$ \\
$K(2 \ff)$ & $1$ \\
$\Q(2^3 \ff)$ & $1$ \\
$\Q(2^5 \ff)$ & $1$
     \end{tabular}
\end{center}
}
\noindent 
{\textbf{\sc Primitive Residue Fields:}} $\Q(\ff)$
\\ \\
The path types that occur in this case are I., V$_1$., V$_3$., VI$_3$, VIII$_2$.

\subsubsection{Case 87: $L = 3$, $a =6$}

{
\begin{center}
    \begin{tabular}{|c|c|}
     Residue Field & Multiplicity \\  \hline
$\Q(\ff)$ & $2$ \\
$K(\ff)$ & $3$ \\
$K(2^2 \ff)$ & $1$ \\
$\Q(2^4 \ff)$ & $1$ \\
$\Q(2^6 \ff)$ & $1$
     \end{tabular}
\end{center}
}
\noindent 
{\textbf{\sc Primitive Residue Fields:}} $\Q(\ff)$
\\ \\
The path types that occur in this case are I., V$_1$., V$_3$., VI$_3$, VIII$_2$.

\subsubsection{Case 88: $L \geq 3$, $a \geq 2L+1$}

{
\begin{center}
    \begin{tabular}{|c|c|}
     Residue Field & Multiplicity \\  \hline
$\Q(2^{a-2L-1} \ff)$ & $2$ \\
$K(2^{a-2L-1} \ff)$ & $2^{L-1}-1$ \\
$K(2^{a-2L} \ff)$ & $2^{L-2}$ \\
$K(2^{a-2L+2} \ff)$ & $2^{L-3}$ \\
$K(2^{a-2L+4} \ff)$ & $2^{L-2}$ \\
$\vdots$ & $\vdots$ \\
$K(2^{a-6} \ff)$ & $2$ \\
$K(2^{a-4} \ff)$ & $1$ \\
$\Q(2^{a-2} \ff)$ & $1$ \\
$\Q(2^a \ff)$ & $1$
     \end{tabular}
\end{center}
}
\noindent 
{\textbf{\sc Primitive Residue Fields:}} $\Q(2^{a-2L-1} \ff)$
\\ \\
The path types that occur in this case are I., V$_1$., V$_3$., VI$_3$, VIII$_2$.

\subsubsection{Case 89: $L \geq 4$, $a = L+1$, $a$ odd}

{
\begin{center}
    \begin{tabular}{|c|c|}
     Residue Field & Multiplicity \\  \hline
$\Q(\ff)$ & $2$ \\
$K(\ff)$ & $2^{\frac{a-3}{2}}-1$ \\
$K(2 \ff)$ & $2^{\frac{a-5}{2}}$ \\
$K(2^3 \ff)$ & $2^{\frac{a-7}{2}}$ \\
$\vdots$ & $\vdots$ \\
$\vdots$ & $\vdots$ \\
$K(2^{a-4} \ff)$ & $1$ \\
$\Q(2^{a-2} \ff)$ & $1$ \\
$\Q(2^{a} \ff)$ & $1$
     \end{tabular}
\end{center}
}
\noindent 
{\textbf{\sc Primitive Residue Fields:}} $\Q(\ff)$
\\ \\
The path types that occur in this case are I., V$_1$., V$_3$., VI$_2$, VIII$_1$.

\subsubsection{Case 90: $L \geq 5$, $a = L+1$, $a$ even}

{
\begin{center}
    \begin{tabular}{|c|c|}
     Residue Field & Multiplicity \\  \hline
$\Q(\ff)$ & $2$ \\
$K(\ff)$ & $2^{\frac{a}{2}-1}-1$ \\
$K(2^2 \ff)$ & $2^{\frac{a}{2}-3}$ \\
$\vdots$ & $\vdots$ \\
$K(2^{a-4} \ff)$ & $1$ \\
$\Q(2^{a-2} \ff)$ & $1$ \\
$\Q(2^{a} \ff)$ & $1$
     \end{tabular}
\end{center}
}
\noindent 
{\textbf{\sc Primitive Residue Fields:}} $\Q(\ff)$
\\ \\
The path types that occur in this case are I., V$_1$., V$_3$., VI$_2$, VIII$_1$.

\subsubsection{Case 91: $L \geq 4$, $a = 2L$}

{
\begin{center}
    \begin{tabular}{|c|c|}
     Residue Field & Multiplicity \\  \hline
$\Q(\ff)$ & $2$ \\
$K(\ff)$ & $2^{L-1}-1$ \\
$K(2^2 \ff)$ & $2^{L-3}$ \\
$K(2^4 \ff)$ & $2^{L-4}$ \\
$\vdots$ & $\vdots$ \\
$K(2^{a-4} \ff)$ & $1$ \\
$\Q(2^{a-2} \ff)$ & $1$ \\
$\Q(2^a \ff)$ & $1$
     \end{tabular}
\end{center}
}
\noindent 
{\textbf{\sc Primitive Residue Fields:}} $\Q(\ff)$
\\ \\
The path types that occur in this case are I., V$_1$., V$_3$., VI$_3$, VIII$_2$.

\subsubsection{Case 92: $L \geq 4$, $a = 2L-1$}

{
\begin{center}
    \begin{tabular}{|c|c|}
     Residue Field & Multiplicity \\  \hline
$\Q(\ff)$ & $2$ \\
$K(\ff)$ & $2^{L-2}-1$ \\
$K(2 \ff)$ & $2^{L-3}$ \\
$K(2^3 \ff)$ & $2^{L-4}$ \\
$\vdots$ & $\vdots$ \\
$K(2^{a-4} \ff)$ & $1$ \\
$\Q(2^{a-2} \ff)$ & $1$ \\
$\Q(2^a \ff)$ & $1$
     \end{tabular}
\end{center}
}
\noindent 
{\textbf{\sc Primitive Residue Fields:}} $\Q(\ff)$
\\ \\
The path types that occur in this case are I., V$_1$., V$_3$., VI$_3$, VIII$_2$.

\subsubsection{Case 93: $L \geq 4$, $a = 2L-2$}

{
\begin{center}
    \begin{tabular}{|c|c|}
     Residue Field & Multiplicity \\  \hline
$\Q(\ff)$ & $2$ \\
$K(\ff)$ & $2^{L-2}-1$ \\
$K(2^2 \ff)$ & $2^{L-4}$ \\
$K(2^4 \ff)$ & $2^{L-5}$ \\
$\vdots$ & $\vdots$ \\
$K(2^{a-4} \ff)$ & $1$ \\
$\Q(2^{a-2} \ff)$ & $1$ \\
$\Q(2^a \ff)$ & $1$
     \end{tabular}
\end{center}
}
\noindent 
{\textbf{\sc Primitive Residue Fields:}} $\Q(\ff)$
\\ \\
The path types that occur in this case are I., V$_1$., V$_3$., VI$_3$, VIII$_2$.

\subsubsection{Case 94: $L \geq 4$, $L+2 \leq a \leq 2L-3$, $a$ odd}

{
\begin{center}
    \begin{tabular}{|c|c|}
     Residue Field & Multiplicity \\  \hline
$\Q(\ff)$ & $2$ \\
$K(\ff)$ & $2^{\frac{a-1}{2}-1}-1$ \\
$K(2 \ff)$ & $2^{\frac{a-1}{2}-2}$ \\
$\vdots$ & $\vdots$ \\
$K(2^{a-4} \ff)$ & $1$ \\
$\Q(2^{a-2} \ff)$ & $1$ \\
$\Q(2^a \ff)$ & $1$
     \end{tabular}
\end{center}
}
\noindent 
{\textbf{\sc Primitive Residue Fields:}} $\Q(\ff)$
\\ \\
The path types that occur in this case are I., V$_1$., V$_3$., VI$_3$, VIII$_2$.

\subsubsection{Case 95: $L \geq 6$, $L+2 \leq a \leq 2L-4$, $a$ even}

{
\begin{center}
    \begin{tabular}{|c|c|}
     Residue Field & Multiplicity \\  \hline
$\Q(\ff)$ & $2$ \\
$K(\ff)$ & $2^{\frac{a}{2}-1}-1$ \\
$K(2^2 \ff)$ & $2^{\frac{a}{2}-3}$ \\
$\vdots$ & $\vdots$ \\
$K(2^{a-4} \ff)$ & $1$ \\
$\Q(2^{a-2} \ff)$ & $1$ \\
$\Q(2^a \ff)$ & $1$
     \end{tabular}
\end{center}
}
\noindent 
{\textbf{\sc Primitive Residue Fields:}} $\Q(\ff)$
\\ \\
The path types that occur in this case are I., V$_1$., V$_3$., VI$_3$, VIII$_2$.


\begin{thebibliography}{DEvHMZB20}


\bibitem[Ao95]{Aoki95} N. Aoki, Torsion points on abelian varieties with complex multiplication. Algebraic
cycles and related topics (Kitasakado, 1994), 122, World Sci. Publ., River Edge, NJ,
1995.

\bibitem[Ar08]{Arai08} K. Arai, \emph{On uniform lower bound of the Galois images associated to elliptic curves.}
J. Théor. Nombres Bordeaux 20 (2008), 23--43. 

\bibitem[BC20a]{BCI} A. Bourdon and P.L. Clark, \emph{Torsion points and Galois representations on CM elliptic curves}, 
Pacific J. Math. 305 (2020),  43--88.

\bibitem[BC20b]{BCII} A. Bourdon and P.L. Clark, \emph{Torsion points and rational isogenies on CM elliptic curves}.  
J. Lond. Math. Soc. (2) 102 (2020), 580--622.

\bibitem[BCP17]{BCP17} A. Bourdon, P.L. Clark, P. Pollack, \emph{Anatomy of torsion in the CM case.}
Math. Z. 285 (2017),  795--820. 

\bibitem[BCS17]{BCS17} A. Bourdon, P.L. Clark and J. Stankewicz, \emph{Torsion points on CM elliptic curves over real number
fields}.  Trans. Amer. Math. Soc. 369 (2017), 8457-–8496.

\bibitem[BP17]{BP17} A. Bourdon and P. Pollack, \emph{Torsion subgroups of CM elliptic curves over odd degree number fields.} Int. Math. Res. Not. IMRN 2017, no. 16, 4923--4961.

\bibitem[CA]{CA} P.L. Clark, \emph{Commutative Algebra}.  \url{http://math.uga.edu/~pete/integral2015.pdf}

\bibitem[CCM21]{CCM21} M. Chou, P.L. Clark and M. Milosevic, \emph{Acyclotomy of torsion in the CM case.}  To appear in the Ramanujan Journal. 

\bibitem[CCRS14]{CCRS14} P.L. Clark, P. Corn, A. Rice and J. Stankewicz, \emph{Computation on elliptic curves with complex multiplication}. LMS J. Comput. Math. 17 (2014), 509--535.


\bibitem[CCS13]{CCS13} P.L. Clark, B. Cook and J. Stankewicz, \emph{Torsion points on elliptic curves with complex multiplication (with an 
appendix by Alex Rice).} Int. J. Number Theory 9 (2013), 447--479.


\bibitem[CGPS]{CGPS} P.L. Clark, T. Genao, P. Pollack and F. Saia, \emph{The least degree of a CM point on a modular curve}. 
 J. Lond. Math. Soc. (2) 105 (2022), 825--883. 



\bibitem[CMP18]{CMP18} P.L. Clark, M. Milosevic, P. Pollack, \emph{Typically bounding torsion.} J. Number Theory 192 (2018), 150--167.

\bibitem[CN21]{Cremona-Najman21} J. Cremona and F. Najman, \emph{$\Q$-curves over odd degree number fields}.  
Res. Number Theory 7 (2021), Paper No. 62, 30 pp. 

\bibitem[CP15]{CP15} P.L. Clark and P. Pollack, \emph{The truth about torsion in the CM case.} C. R. Math. Acad. Sci. Paris 353 (2015),  683--688.

\bibitem[CP17]{CP17} P.L. Clark and P. Pollack, \emph{The truth about torsion in the CM case, II.} Q. J. Math. 68 (2017), 1313--1333.

\bibitem[CP18]{CP18} P.L. Clark and P. Pollack, \emph{Pursuing polynomial bounds on torsion}. Israel J. Math. 227 (2018), 889--909. 

\bibitem[CP22]{CP22} F. Campagna and R. Pengo, \emph{Entanglement in the family of division fields of elliptic curves with complex multiplication.} Pacific J. Math. 317 (2022),  21--66. 

\bibitem[CS22b]{CS22b} P.L. Clark and F. Saia, \emph{CM elliptic curves: volcanoes, reality and applications, Part II}.

\bibitem[CT12]{CT12} A. Cadoret and A. Tamagawa, \emph{A uniform open image theorem for $\ell$-adic representations, I}. Duke Mathematical Journal 161 (2012), 2605--2634. 

\bibitem[Co1]{Conrad} K. Conrad, \emph{The conductor ideal.}  \url{www.math.uconn.edu/~kconrad/blurbs/gradnumthy/conductor.pdf}

\bibitem[Co2]{Conrad2} K. Conrad, \emph{Tensor Products II} \url{https://kconrad.math.uconn.edu/blurbs/linmultialg/tensorprod2.pdf}

\bibitem[Cx89]{Cox89} D. Cox, \emph{Primes of the form $x^2+ny^2$.  Fermat, class field theory and complex multiplication.}  John Wiley $\&$ Sons, New York, 1989.

\bibitem[DEvHMZB20]{DEvHMZB20} M. Derickx, A. Etropolski, M. van Hoeij, J.S. Morrow and D. Zureick-Brown, \emph{Sporadic 
Cubic Torsion}.  To appear in \emph{Algebra and Number Theory}.

\bibitem[DR73]{Deligne-Rapoport73} P. Deligne and M. Rapoport, Les sch\'emas de modules de courbes elliptiques. Modular functions of one variable, II (Proc. Internat. Summer School, Univ. Antwerp, Antwerp, 1972), pp. 143–316. Lecture Notes in Math.,
Vol. 349, Springer, Berlin, 1973.

\bibitem[FM02]{Foquet-Morain02} M. Foquet and F. Morain, \emph{Isogeny volcanoes and the SEA algorithm.} Algorithmic number theory (Sydney, 2002), 276--291, Lecture Notes in Comput. Sci., 2369, Springer, Berlin, 2002. 


\bibitem[HK13]{Halter-Koch} F. Halter-Koch, \emph{Quadratic irrationals}. An introduction to classical number theory. Pure and Applied Mathematics. CRC Press, Boca Raton, FL, 2013.

\bibitem[JT15]{Jensen-Thorup15} C.U. Jensen and A. Thorup, \emph{Gorenstein orders.}
J. Pure Appl. Algebra 219 (2015), 551--562. 


\bibitem[Ka92]{Kamienny92} S. Kamienny, \emph{Torsion points on elliptic curves and q-coefficients of modular forms.}
Invent. Math. 109 (1992), 221--229. 

\bibitem[KM88]{Kenku-Momose88} M.A. Kenku and F. Momose, \emph{Torsion points on elliptic curves defined over quadratic fields.} Nagoya Math. J. 109 (1988), 125--149. 

\bibitem[Ko96]{Kohel96}  D. R. Kohel, \emph{Endomorphism rings of elliptic curves over finite fields.} Ph.D. thesis, Univ. California, Berkeley, 1996.

\bibitem[Kw99]{Kwon99} S. Kwon, \emph{Degree of isogenies of elliptic curves with complex multiplication}. J. Korean Math. Soc. 36
(1999), 945--958.

\bibitem[LD15]{Lv-Deng15} C. Lv and Y.P. Deng, \emph{On orders in number fields: Picard groups, ring class fields and applications.} 
Sci. China Math. 58 (2015), 1627--1638.

\bibitem[LR22]{LR22} \'{A}. Lozano-Robledo, \emph{Galois representations attached to elliptic curves with 
complex multiplication}.  Algebra Number Theory 16 (2022), 777--837.

\bibitem[LV14]{Larson-Vaintrob14} E. Larson and D. Vaintrob, \emph{Determinants of subquotients of Galois representations associated with abelian varieties}. Journal of the Institute of Mathematics of Jussieu 13 (2014), 517--559.

\bibitem[Ma76]{Mazur76} B. Mazur, \emph{Rational points on modular curves.} Modular functions of one variable, V (Proc. Second Internat. Conf., Univ. Bonn, Bonn, 1976), pp. 107–148. Lecture Notes in Math., Vol. 601, Springer, Berlin, 1977.

\bibitem[Ma77]{Mazur77} B. Mazur, \emph{Modular curves and the Eisenstein ideal. With an appendix by Mazur and M. Rapoport.} Inst. Hautes \' Etudes Sci. Publ. Math. (1977), 33--186.

\bibitem[Me96]{Merel96} L. Merel, \emph{Bornes pour la torsion des courbes elliptiques sur les corps de nombres.}   Invent. Math. 124 (1996), 437--449.

\bibitem[N]{Neukirch} J. Neukirch, \emph{Algebraic number theory.} Translated from the 1992 German original and with a note by Norbert Schappacher. With a foreword by G. Harder. Grundlehren der Mathematischen Wissenschaften [Fundamental Principles of Mathematical Sciences], 322. Springer-Verlag, Berlin, 1999. 

 \bibitem[Ol74]{Olson74} L. Olson, \emph{Points of finite order on elliptic curves with complex multiplication}. Manuscripta math. 14 (1974), 195--205.


 \bibitem[Pa89]{Parish89} J.L. Parish, \emph{Rational Torsion in Complex-Multiplication Elliptic Curves}.  Journal of Number Theory 33 (1989), 257--265.

\bibitem[Ro97]{Rohrlich97} D.E. Rohrlich, \emph{Modular curves, Hecke correspondence, and L-functions.} Modular forms and Fermat's last theorem (Boston, MA, 1995), 41–100, Springer, New York, 1997. 


\bibitem[RS17]{Rosen-Shnidman17} J. Rosen and A. Shnidman, \emph{Extensions of CM elliptic curves and orbit counting on the projective line.}
Res. Number Theory 3 (2017), Paper No. 9, 13 pp.

\bibitem[SiI]{SilvermanI} J.H. Silverman, \emph{The arithmetic of elliptic curves}. Second edition. Graduate Texts in Mathematics, 106. Springer, Dordrecht, 2009.

 \bibitem[SiII]{SilvermanII} J.H. Silverman, \emph{Advanced Topics in the Arithmetic of Elliptic Curves}, Graduate Texts in
 Mathematics 151, Springer-Verlag, 1994.

\bibitem[Si88]{Silverberg88} A. Silverberg, \emph{Torsion points on abelian varieties of CM-type}.  Compositio Math.  68  (1988),  241--249.

\bibitem[Si92]{Silverberg92} A. Silverberg,  \emph{Points
of finite order on abelian varieties}.  In \emph{$p$-adic methods
in number theory and algebraic geometry}, 175--193,
Contemp. Math. 133, Amer. Math. Soc., Providence, RI, 1992.


\bibitem[St01]{Stevenhagen01} P. Stevenhagen, \emph{Hilbert's 12th problem, complex multiplication and Shimura reciprocity.} Class field theory -- its centenary and prospect (Tokyo, 1998), 161–-176,
Adv. Stud. Pure Math., 30, Math. Soc. Japan, Tokyo, 2001.


\bibitem[Su12]{Sutherland12} A.V. Sutherland, \emph{Isogeny volcanoes}.  \url{https://arxiv.org/abs/1208.5370} 

\bibitem[Sz]{Szamuely} T. Szamuely, \emph{Galois groups and fundamental groups.} Cambridge Studies in Advanced Mathematics, 117. Cambridge University Press, Cambridge, 2009.

\bibitem[Vo07]{Voight07} J. Voight, \emph{Quadratic forms that represent almost the same primes}. Math. Comp. 76 (2007), 1589--1617.

\bibitem[We73]{Weinberger73} P.J. Weinberger, \emph{Exponents of the class groups of complex quadratic fields.}
Acta Arith. 22 (1973), 117--124. 

\end{thebibliography}
\end{document}